\documentclass[11pt]{article}

\usepackage{amsmath,amssymb,amsthm,bm,mathrsfs}
\usepackage{enumitem}
\usepackage{microtype}
\usepackage{hyperref}
\usepackage{cleveref}

\usepackage{color, graphicx}
\usepackage{booktabs}
\usepackage{float}

  \usepackage[margin=1in, marginparwidth=3cm]{geometry}

\newcommand{\R}{\mathbb{R}}

\newtheorem{theorem}{Theorem}
\newtheorem{lemma}{Lemma}

\newtheorem{remark}{Remark}

\title{A consistent-splitting generalized scalar auxiliary variable scheme for the perturbed Boussinesq system}
\author{M Nader Alhomsi\footnote{Department of Mathematics,  Central Michigan University, 
Mount Pleasant, MI 48858. Email: Alhom1n@cmich.edu}\,, Jiahong Wu\footnote{Department of Mathematics,  University of Notre Dame, 
Notre Dame, IN 46556. Email: jwu29@nd.edu}\,
and Xiaoming Zheng\footnote{Department of Mathematics, Central Michigan University, Mount Pleasant, MI 48858. Email: zheng1x@cmich.edu}
}
\date{}

\begin{document}
\maketitle

\begin{abstract}
We propose and analyze a second-order consistent-splitting scheme, based on the generalized scalar auxiliary variable (GSAV) approach, for the two-dimensional perturbed Boussinesq system. The system is obtained by subtracting a stable, linearly-stratified hydrostatic equilibrium from the standard Boussinesq system. The time discretization extends the consistent-splitting generalized BDF2 framework of Huang and Shen \cite{HuangShen2023} for the Navier--Stokes equations, treating the nonlinear convection and advection together with the linear buoyancy and stratification couplings explicitly, so that each time step reduces to a small number of decoupled linear systems.
We prove an unconditional weak stability theorem for the GSAV scheme and derive optimal second-order error estimates for the velocity, pressure, and temperature.
A careful tracing reveals that the error constant depends on the inverse viscosity and inverse thermal diffusivity through a quadruply-nested exponential, so the scheme is not robust as either tends to zero. Numerical experiments confirm the second-order convergence and reproduce the expected internal-wave dynamics and exponential relaxation toward hydrostatic balance in a long-time stratified-flow simulation.
\end{abstract}

\textbf{Keywords:} Perturbed Boussinesq system, Generalized scalar auxiliary variable (GSAV), Consistent-splitting scheme, Unconditional stability, Robustness on viscosity, Stratified flow.

\section{Introduction}

Buoyancy-driven flows lie at the heart of countless geophysical and engineering processes, from atmospheric circulation and oceanic dynamics to convective heat transfer, internal wave propagation, and mantle convection~\cite{Gill1982,Pedlosky1987,KunduCohen2015,Vallis2017,Sutherland2010}.
A widely-used mathematical model for such phenomena is the Boussinesq approximation, in which density variations are assumed small and are retained only through the gravity-coupled buoyancy term, while the velocity field remains divergence-free.
Stably stratified environments---such as the atmosphere, the ocean thermocline, and stellar interiors---support hydrostatic equilibria in which a vertical temperature gradient balances gravity.
In this work, we adopt the classical linearly-stratified equilibrium, with constant lapse rate measured by the (squared) Brunt--V\"ais\"al\"a frequency~\cite{Gill1982,Vallis2017}.
Subtracting this equilibrium from the full Boussinesq system yields the \emph{perturbed} Boussinesq system studied here, in which buoyancy drives vertical motion in the momentum equation and vertical motion advects the linear background gradient as a constant-coefficient stratification term in the temperature equation.
This coupling is the source of internal gravity waves and  supports the relaxation of disturbances toward hydrostatic balance~\cite{Sutherland2010,Vallis2017}.

To make this precise, let $\Omega\subset\mathbb{R}^2$ be a bounded domain with sufficiently smooth boundary $\partial\Omega$.
Denote the velocity vector as $\mathbf{U}$, the pressure as $P$, and the
temperature as $\Theta$. The full Boussinesq system is
\begin{align}
\partial_t \mathbf{U} + \mathbf{U} \cdot \nabla \mathbf{U}
&= -\nabla P + \nu\,\Delta \mathbf{U} + \Theta\,\mathbf{e}_2,
\qquad \mathbf{x} = (x_1, x_2) \in \Omega,\ t > 0,
\label{eq:momentum}\\
\partial_t \Theta + \mathbf{U} \cdot \nabla \Theta
&= \gamma\,\Delta\Theta,
\label{eq:temperature}\\
\nabla \cdot \mathbf{U} &= 0,
\label{eq:incompressibility}
\end{align}
where $\nu > 0$ is the kinematic viscosity and $\gamma > 0$ is the thermal
diffusivity, and $\mathbf{e}_2$ is the unit vector in the vertical direction.
Denote a hydrostatic equilibrium as
\begin{equation}
(\mathbf{U}_{\mathrm{he}},\, P_{\mathrm{he}},\, \Theta_{\mathrm{he}})
= \Bigl(0,\; 
\int_0^{x_2} \mathcal{T}(z) dz,\; 
\mathcal{T}(x_2)=\alpha x_2 + \overline{\theta}
\Bigr),
\label{def_he}
\end{equation}
where $\overline{\theta}=\frac{1}{|\Omega|} \int_\Omega \Theta dx_1dx_2$ is the spatial average. 
Define the perturbation as
\begin{equation}
(\mathbf{u},\, p,\, \theta)
= (\mathbf{U} - \mathbf{U}_{\mathrm{he}},\;
   P - P_{\mathrm{he}},\;
   \Theta - \Theta_{\mathrm{he}}).
\label{def_perturbations}
\end{equation}
Substituting the perturbation decomposition~\eqref{def_perturbations} into the full system~\eqref{eq:momentum}--\eqref{eq:incompressibility} yields the perturbed Boussinesq system for $\mathbf{u}=(u,v)$, $p$, and $\theta$:
\begin{equation}\label{eq:bouss}
\begin{cases}
\partial_t \mathbf{u} + \mathbf{u}\cdot\nabla \mathbf{u} + \nabla p - \nu \Delta \mathbf{u}
= \theta\,\mathbf{e}_2 + \mathbf{f},\\[2mm]
\partial_t \theta + \mathbf{u}\cdot\nabla \theta + \alpha\, v - \gamma \Delta \theta
= g,\\[1mm]
\nabla\cdot \mathbf{u} = 0,
\end{cases}
\end{equation}
The  Dirichlet boundary conditions are imposed:
$\mathbf{u} = 0$, $\theta=0$ on $\partial\Omega$.

\medskip

The mathematical analysis of the Boussinesq system has attracted intense interest. The fully dissipative system in two space dimensions—with both kinematic viscosity and thermal diffusion—is globally well-posed and develops no finite-time singularities; the substantial difficulties arise once the dissipation is reduced. In the partial-dissipation setting, Hou and Li~\cite{HouLi2005} established global well-posedness of the 2D system with full velocity viscosity but zero thermal diffusivity, and Chae~\cite{Chae2006} proved global regularity when either the kinematic viscosity or the thermal diffusivity is absent. Subsequent works substantially extended these results to systems with only anisotropic or directional dissipation. For vertical dissipation, Adhikari, Cao, and Wu~\cite{AdhikariCaoWu2010,AdhikariCaoWu2011} obtained global regularity and well-posedness results for the 2D Boussinesq equations with vertical viscosity and vertical diffusivity, and Cao and Wu~\cite{CaoWu2013} proved global regularity for the 2D anisotropic Boussinesq equations with only vertical dissipation. For horizontal dissipation, Danchin and Paicu~\cite{DanchinPaicu2011} constructed global solutions for the anisotropic Boussinesq system in two dimensions in which the diffusion acts in the horizontal direction in only one equation, exploiting the fact that the dissipation occurs in a direction perpendicular to the buoyancy force. We refer to these works and the references therein for the now-extensive theory of partially dissipative Boussinesq systems.

A second line of investigation, which is the one directly relevant to the perturbed system \eqref{eq:bouss}, concerns the stability of small perturbations of a stable hydrostatic equilibrium and the rate at which they decay to zero.  Doering, Wu, Zhao, and Zheng~\cite{DoeringWuZhaoZheng2018} initiated a quantitative study of the long-time behavior of the 2D Boussinesq equations \emph{without} buoyancy diffusion and demonstrated algebraic decay of the velocity perturbation.  Tao, Wu, Zhao, and Zheng~\cite{TaoWuZhaoZheng2020}
subsequently solved several problems left open in~\cite{DoeringWuZhaoZheng2018} and established the nonlinear stability of perturbations near hydrostatic equilibrium together with explicit algebraic decay rates.  Closely related stability and exponential-decay
results in the anisotropic setting have been obtained by Dong, Wu, Xu, and Zhu~\cite{DongWuXuZhu2021} for the 2D Boussinesq equations with only horizontal dissipation, and by Ji, Yan, and Wu~\cite{JiYanWu2022} who derived \emph{optimal} decay rates for the 3D anisotropic Boussinesq system near the hydrostatic balance.  Earlier $H^{1}$-stability results
with partial dissipation are due to Ji, Li, Wei, and
Wu~\cite{JiLiWeiWu2019}.  In a different but related direction, Castro, C\'ordoba, and Lear~\cite{CastroCordobaLear2019} proved asymptotic
stability of stratified solutions of the 2D Boussinesq equations with a velocity damping term.  These works collectively establish that, under suitable dissipation, smooth perturbations of a stable hydrostatic
equilibrium remain regular for all time and decay to zero, with algebraic or exponential rates depending on the dissipation structure. We emphasize, however, that each of the references above treats a \emph{partial} or \emph{anisotropic} dissipation regime, whereas the
system \eqref{eq:bouss} considered in the present paper is
\emph{fully dissipative}, with full kinematic viscosity
$\nu\Delta\mathbf{u}$ and full thermal diffusion $\gamma\Delta\theta$, posed on a bounded two-dimensional domain with homogeneous Dirichlet boundary conditions.  For this fully dissipative case,  exponential decay of unforced perturbations to zero is proved in Lemma\,\ref{lemma:exp_decay}.  The present paper verifies this analytic theory in Section~\ref{sec:numerical}.

On the numerical side, the Boussinesq system inherits the classical difficulties of incompressible Navier-Stokes solvers: the divergence-free constraint and the resulting saddle-point structure. The direct monolithic time discretization solves a nonlinear, coupled velocity-pressure system at each time step, whose numerical methods can be found in a review paper \cite{John2021}.
In contrast, the splitting schemes divide the time step into sequential stages. The projection methods, originating with the work of Chorin~\cite{Chorin1968} and Temam~\cite{Temam1968}, decouple velocity and pressure through a Helmholtz projection; we refer to~\cite{GuermondMinevShen2006} for a comprehensive overview and \cite{Zheng2025IterativeProjection} for recent development in an iterative projection method. The projection methods typically fail to achieve optimal accuracy due to inconsistent boundary conditions imposed on the pressure or velocity~\cite{GuermondMinevShen2006}.
In contrast,  gauge formulations~\cite{weinan2003gauge} and consistent-splitting schemes~\cite{guermond2003new,JohnstonLiu2004,shen2007error,
HuangShen2023,HuangShen2025,AlhomsiWuZhengSubmitted}
retain the optimal second-order accuracy while keeping the algebraic cost per step low. 

The scalar auxiliary variable (SAV) approach, introduced by Shen, Xu, and Yang~\cite{ShenXuYang2018JCP,ShenXuYang2019SIREV} for gradient flows, augments the system with a scalar variable updated to mimic the continuous energy law, yielding unconditional stability of a modified discrete energy at the cost of only a few decoupled linear solves per time step.
Huang and Shen~\cite{huang2022new} subsequently introduced a new class of SAV schemes for general dissipative systems, referred to in the literature as GSAV, a framework that covers the perturbed Boussinesq system~\eqref{eq:bouss}.
The same authors then combined the GSAV method with an innovative time-shifted second-order BDF consistent-splitting scheme (with shift parameter $k$) for the Navier--Stokes equations~\cite{HuangShen2023}, which forms the foundation of the present work.

Several recent papers have applied SAV-type approaches to the Boussinesq system.
Zhang, Yuan, and Chen~\cite{ZhangYuanChen2024} proposed a fully decoupled first-order SAV pressure-correction scheme and established rigorous error estimates.
Li, Li, and Luo~\cite{LiLiLuo2026} analyzed a two-step BDF2 scheme based on the exponential SAV (E-SAV) variant, which decouples velocity from temperature.
Jiang and Yang~\cite{JiangYang2023} developed stabilized SAV-BDF2 and Crank--Nicolson--leapfrog \emph{ensemble} schemes for the evolutionary Boussinesq equations.
Most directly comparable to the present work, Wagner, Wohlmuth, and Zawallich~\cite{WagnerWohlmuthZawallich2026} extended the second-order consistent-splitting GSAV-BDF2 (with time shift $k$) framework of Huang and Shen~\cite{HuangShen2023} from the Navier--Stokes equations to the standard Boussinesq equation, employing an exponential time integrator for the auxiliary-variable update and two independent extrapolation widths for velocity and temperature; their scheme is reformulated for an $H^1$-conforming finite element discretization (e.g., Taylor--Hood), with two-dimensional error analysis and three-dimensional numerical experiments.

Compared with these existing works, the present paper differs in three essential respects.
First, all four of the above references analyze the \emph{standard} Boussinesq system, whereas the present paper analyzes the \emph{perturbed} Boussinesq system~\eqref{eq:bouss} carrying a linear stratification coupling $\alpha v$ in the temperature equation.
Second, among these works only~\cite{WagnerWohlmuthZawallich2026} combines the GSAV approach with the time-shifted second-order BDF consistent-splitting framework of~\cite{HuangShen2023}; we adopt the same time-stepping framework but retain the standard GSAV update as in \cite{HuangShen2023}.
Third, our error analysis explicitly tracks the dependence of the final error constant on the viscosity $\nu$ and the thermal diffusivity $\gamma$, and shows that this constant contains negative powers of $\nu$ and $\gamma$ (see Remark~\ref{rem:C53_dependence}); consequently, the explicit treatment of the convection terms adopted here may not be robust in the vanishing-viscosity limit $\nu, \gamma \to 0^+$, and a fully implicit or stabilized treatment of the convection would be required to obtain bounds uniform in $\nu$ and $\gamma$.
The combination of (i) the perturbed Boussinesq system with its additional stratification coupling, (ii) a consistent extrapolation of all coupling terms at the BDF time level $t^{n+k}$ that preserves full second-order accuracy in both the velocity and temperature equations, and (iii) an explicit accounting of the $\nu^{-1}$ and $\gamma^{-1}$ dependence of the error constant is, to the best of our knowledge, new.

The rest of the paper is organized as follows.
Section~\ref{sec:prelim} collects preliminary lemmas, including the energy law for the continuous Boussinesq system and the exponential decay of its energy in the unforced case.
Section~\ref{sec:GSAV_scheme} introduces the proposed consistent-splitting GSAV scheme and establishes the weak stability theorem.
Section~\ref{sec:error} carries out the error analysis.
Section~\ref{sec:numerical} reports the numerical results.
Section~\ref{sec:discussion} concludes with a discussion.

\medskip



\section{Preliminaries}\label{sec:prelim}
We denote the Lebesgue space of squared integrable functions as $L^2(\Omega)$ and the Sobolev spaces $H^n(\Omega)$ and their norms as $\|\cdot \|$ and $\|\cdot\|_n$, respectively. In the above,  $(\cdot, \cdot)$ is the inner product in the $L^2(\Omega)$ space. The vectorized spaces are denoted as  
${\bf L}^2(\Omega)=(L^2(\Omega))^d$, 
${\bf H}^n(\Omega)=(H^n(\Omega))^d$,
and 
${\bf H}^n_0(\Omega)=(H^n_0(\Omega))^d$,
where $d$ is the space dimension and taken as $2$ in this work.

The perturbed system has the following energy and properties.
\begin{lemma}[Energy law]\label{lem:energy_law}
Define the weighted energy
\begin{equation}\label{eq:energy}
E(\mathbf{u},\theta)=\frac12\|\mathbf{u}\|^2+\frac{1}{2\alpha}\|\theta\|^2,
\end{equation}
where $\|\cdot\|$ and $(\cdot,\cdot)$ denote the $L^2(\Omega)$ norm and inner product, respectively.
The solution of \eqref{eq:bouss} satisfies
\begin{equation}\label{eq:energy_law}
\frac{d}{dt}E(\mathbf{u},\theta)
= -\nu\|\nabla\mathbf{u}\|^2-\frac{\gamma}{\alpha}\|\nabla\theta\|^2
+(\mathbf{f},\mathbf{u})+\frac{1}{\alpha}(g,\theta).
\end{equation}
\end{lemma}

\begin{lemma}[Exponential decay of unforced perturbations]
\label{lemma:exp_decay}
Assume $\alpha>0$, that $\Omega\subset\mathbb{R}^2$ is a bounded domain
with smooth boundary, and let $\lambda_{\min}>0$ denote the smallest
Dirichlet eigenvalue of $-\Delta$ on $\Omega$.  Let
$(\mathbf{u},p,\theta)$ be a smooth solution of
\eqref{eq:bouss} on $[0,\infty)$ with $\mathbf{f}\equiv\mathbf{0}$
and $g\equiv 0$.  Then the weighted energy
\eqref{eq:energy} satisfies
\begin{equation}\label{eq:exp_decay}
E(\mathbf{u}(t),\theta(t))
\;\le\;
E(\mathbf{u}_0,\theta_0)\,\exp\bigl(-2\lambda_{\min}\cdot  \min(\nu,\gamma)\,t\bigr),
\qquad \forall\,t\in[0,T).
\end{equation}
Hence, $\|\mathbf{u}(t)\|, \|\theta(t)\|\to 0$ exponentially
as $t\to\infty$.
\end{lemma}

As in \cite{HuangShen2023},  the following estimates are often used in this work:
$\forall {\bf u}, {\bf v}, {\bf w}\in {\bf H}^1_0(\Omega)\cap {\bf H}^2(\Omega)$, 
\begin{align}
|({\bf u}\cdot\nabla {\bf v}, {\bf w})|
& \le C\,
\|{\bf u}\|^{1/2}\,\|\nabla {\bf u}\|^{1/2} \,
\|\nabla {\bf v}\|^{1/2}\,\| {\bf v}\|_2^{1/2} \|{\bf w}\|,
\quad d=2, 
\label{convection_estimate}
\\
\|{\bf u}\cdot\nabla{\bf u}\|^2
&\le C\|{\bf u}\|\|\nabla{\bf u}\|^2 \|\Delta {\bf u}\|,
\quad d=2,
\label{convection_estimate2}
\\
\|{\bf u}\cdot\nabla {\bf v}\|
&\le C\|\nabla {\bf u}\|\,\|\Delta {\bf v}\|,
  \text{ and } 
\|{\bf u}\cdot\nabla {\bf v}\|
\le C\|{\bf u}\|_2\,\|\nabla {\bf v}\|,
\quad d\le 3,
\label{convection_estimate3}
\end{align}
\begin{align}
|({\bf u}\cdot\nabla {\bf v}, {\bf w})|
&\le
\left\{
\begin{aligned}
&C \|\mathbf u\|_{1}\|\mathbf v\|_{1}\|\mathbf w\|_{1},
\quad 
C \|\mathbf u\|_{2}\|\mathbf v\|_{0}\|\mathbf w\|_{1},\\
&C \|\mathbf u\|_{2}\|\mathbf v\|_{1}\|\mathbf w\|_{0},
\quad 
C \|\mathbf u\|_{1}\|\mathbf v\|_{2}\|\mathbf w\|_{0},
\quad
C \|\mathbf u\|_{0}\|\mathbf v\|_{2}\|\mathbf w\|_{1},
\end{aligned}
\right\}
\qquad d\le 4,
\label{Teman_estimate}
\end{align}
\begin{align}
\|{\bf u}\|_2
&\le C \|\Delta {\bf u} \|
\quad \text{(elliptic regularity)},
\label{elliptic_regularity}
\end{align}
\begin{align}
\|{\bf u} \|
&\le C \| \nabla {\bf u} \|
\quad \text{(Poincare/Sobolev inequality)},
\label{Sobole-inequality}
\end{align}
\begin{align}
\|\nabla p_s({\bf u})\|^2 
&\le \left(\frac{1}{2}+\varepsilon\right) \|\Delta {\bf u}\|^2
+ C_S(\varepsilon) \| \nabla {\bf u} \|^2,
\quad \text{ where } C_S(\varepsilon)=\frac{C}{\varepsilon^3} \text{ when } \varepsilon\to 0+.
\label{eq:stokes_pressure_estimate}
\end{align}
The inequalities \eqref{convection_estimate}-\eqref{Teman_estimate} are  taken from \cite{HuangShen2023}.
The upper bound in  \eqref{eq:stokes_pressure_estimate} is proved in \cite{AlhomsiWuZhengSubmitted}, a refinement of Theorem 1 in \cite{LiuLiuPego2007}. The symbol    
$p_s({\bf u})$ is the  Stokes pressure for any ${\bf u} \in {\bf H}^{2}(\Omega)$ defined by
$\nabla p_s({\bf u})
= \bigl( \Delta \mathcal{P} - \mathcal{P}\Delta \bigr){\bf u}$, 
where $\mathcal{P}$ is the Leray--Helmholtz projection onto divergence-free vector fields with zero normal component on
$\partial\Omega$. 
Based on \cite{LiuLiuPego2007},
\begin{equation}
(\nabla p_{s}(\mathbf{u}),\nabla q)
= -(\nabla\times\nabla\times\mathbf{u},\nabla q),
\quad \forall q\in H^1(\Omega).
\label{Stokespressure_identity}
\end{equation}
In the Helmholtz decomposition
${\bf u} = \mathcal{P}{\bf u} + \nabla \phi$, the quantity $\phi \in H^{1}(\Omega)$ satisfies 
$\bigl( {\bf u} - \nabla \phi, \nabla q \bigr) = 0$, 
$\forall q \in H^{1}(\Omega)$. 
The form $C_S(\varepsilon)=C/\varepsilon^3$ holds only when $\varepsilon\to 0+$.

\begin{remark}
All the constants $C$ from \eqref{convection_estimate} to \eqref{eq:stokes_pressure_estimate} 
are only dependent on $\Omega$, thus taken as a single value (their maximum). Throughout this work, the symbol $C$ denotes this fixed constant, rather than a generic constant that may vary from step to step.
\end{remark}

\newtheorem{prelemma}{Lemma}[section]

\begin{lemma}[Algebraic identity]\label{lem:alg_identity}
Let $x,y,z\in\R^d$ and $k\ge \tfrac12$. Then
\begin{equation}\label{eq:lemma}
\begin{aligned}
\big[(2k+1)x - 4ky + (2k-1)z\big]\cdot \big[(k+1)x - ky\big]
&= A\big(|x|^2-|y|^2\big) + |Bx-Dy|^2 - |By-Dz|^2 \\
&\quad + E|x-y|^2 - F|y-z|^2 + G|x-2y+z|^2,
\end{aligned}
\end{equation}
where
$A=\frac{1}{2k}$,
$B=\frac{(k+1)\sqrt{2k-1}}{\sqrt{2k}}$,
$D=\sqrt{\frac{k(2k-1)}{2}}$,
$E=\frac{2k+3}{2}$, 
$F=\frac{2k-1}{2}$, 
$G=\frac{(k+1)(2k-1)}{2}$.
\end{lemma}

The following discrete Gr\"onwall lemma is taken from \cite{HuangShen2023}.
\begin{lemma}[Discrete Gr\"onwall Lemma {\cite{HuangShen2023}}]
\label{lem:discrete_gronwall_2}
Let $a_{n}, b_{n}, c_{n}, d_{n}$ be four nonnegative sequences satisfying
\begin{equation}\label{eq:disc_gronwall_assumption}
a_{m}+\tau\sum_{n=1}^{m}b_{n}
\le \tau\sum_{n=0}^{m-1}a_{n}d_{n}
+\tau\sum_{n=0}^{m-1}c_{n}+Z,
\qquad m\ge 1,
\end{equation}
where $Z$ and $\tau$ are two positive constants.  Then
\begin{equation}\label{eq:disc_gronwall_conclusion}
a_{m}+\tau\sum_{n=1}^{m}b_{n}
\le \exp\!\left(\tau\sum_{n=0}^{m-1}d_{n}\right)
\!\left(\tau\sum_{n=0}^{m-1}c_{n}+Z\right),
\qquad m\ge 1.
\end{equation}
\end{lemma}

The following two results are taken from \cite{AlhomsiWuZhengSubmitted}.
\begin{lemma}
\label{lem:k_condition}
Let $\phi>0$. Suppose that $k>0$ satisfies
\begin{equation}\label{eq:lemma23_cond}
\frac{k-1}{k}\ge \frac{\varepsilon_2}{2}+\frac{1}{4\varepsilon_2}+\varepsilon\,\phi,
\end{equation}
for some $\varepsilon_2>0$ and $\varepsilon>0$.
Then
$\inf_{\varepsilon_2>0,\ \varepsilon>0}
\left(\frac{\varepsilon_2}{2}+\frac{1}{4\varepsilon_2}+\varepsilon\,\phi\right)
=\frac{1}{\sqrt2}$ attained at $\varepsilon_2=\frac{1}{\sqrt2}$ and $\varepsilon\to 0^+$.
Consequently,
$k\ge \frac{1}{1-\frac{1}{\sqrt2}} \approx 3.41$,
and therefore the minimal integer choice is $k=4$. Moreover, for $k=4$ and $\varepsilon_2=1/\sqrt2$,
\eqref{eq:lemma23_cond} holds when
$0<\varepsilon\le \frac{1}{\phi}\left(\frac34-\frac{1}{\sqrt2}\right)$.
\end{lemma}

\begin{lemma}
\label{lemma_xi_eta}
Assume $|1-\xi| \le C_0 \delta t$ holds for $C_0>1$ and 
$0<\delta t \le \frac{1}{1+2C_0^2}$. 
Let $\eta=1-(1-\xi)^2$.
Then 
\begin{align}
\tfrac{1}{2} < \xi < \tfrac{3}{2},\quad 
\tfrac34 < \eta \le 1.
\label{estimate_xi_eta}
\end{align}
\end{lemma}


\section{The GSAV Scheme for the Boussinesq System}
\label{sec:GSAV_scheme}

\subsection{Second-order consistent-splitting GSAV scheme}
\label{subsec:scheme}
We define the continuous scalar auxiliary variable (SAV) as
\begin{equation}
r(t)=E({\bf u}(t), \theta(t)) + \bar{C}.
\label{r_def}
\end{equation}
where $E$ is the energy defined in \eqref{eq:energy} and  $\overline{C}>0$ is a constant specified in
Theorem~\ref{thm:weak_stability}. 
Denote the exact solution of \eqref{eq:bouss} at time $t^{i}=i\delta t$ as $(\mathbf{u}(t^{i}),p(t^{i}),\theta(t^i))$, and the numerical solution as $(\mathbf{u}^{i},p^{i},\theta^i)$. The initial data of problem \eqref{eq:bouss} are denoted as 
$\mathbf{u}^{0}$, $\theta^{0}$, $p^{0}$.
We assume $\mathbf{u}^{1}$, $\theta^{1}$, and $p^{1}$
are obtained by a second-order one-step method
(e.g.\ the second-order Runge--Kutta method),
and we set
\begin{equation}\label{eq:r1_def}
r^{1}
= E(\mathbf{u}^{1},\theta^{1}) + \overline{C},
\end{equation}

We introduce the following notations to facilitate the presentation, 
\begin{align}
\widehat{\psi}^{\,n}
&= (k+1)\psi^{n}-k\psi^{n-1}, 
\label{hat_def}
\\
\widetilde{\psi}^{n}
&= (k+1)\overline{\psi}^{n} - k\overline{\psi}^{n-1},
\label{tilde_def}
\end{align}
where the quantity $\psi$ is a generic variable representing ${\bf u}$, $\theta$, $p$, etc.
Denote the generalized BDF2 finite difference operator with time shift $k$ proposed in \cite{HuangShen2023} as
\begin{align}
D_{k}\,\phi^{n+1} =\frac{(2k+1)\phi^{n+1}-4k\phi^{n}+(2k-1)\phi^{n-1}}{2\delta t}.
\end{align}
For $n\ge 1$, assume that
$\mathbf{u}^{n-1}$, $\mathbf{u}^{n}$,
$\theta^{n-1}$, $\theta^{n}$,
$p^{n-1}$, $p^{n}$, $r^{n}$
are known. 
At the $n$-th time step, we compute
$\overline{\mathbf{u}}^{\,n+1}$,
$\overline{\theta}^{\,n+1}$,
$p^{n+1}$,
$r^{n+1}$,
$\mathbf{u}^{n+1}$,
$\theta^{\,n+1}$
via the following process.
\begin{align}
D_{k}\overline{\mathbf{u}}^{\,n+1}
&= \nu\Delta\!\bigl(
k\overline{\mathbf{u}}^{\,n+1}-(k-1)\overline{\mathbf{u}}^{\,n}
\bigr)
-\widehat{\mathbf{u}}^{\,n}\cdot\nabla\widehat{\mathbf{u}}^{\,n}
-\nabla\widehat{p}^{\,n}
+\widehat{\theta}^{\,n}\mathbf{e}_{2}
+\mathbf{f}^{\,n+k}, 
\label{eq:GSAV_Bouss_compact_a_corr}\\
D_{k}\overline{\theta}^{\,n+1}
&= \gamma\Delta\!\bigl(
k\overline{\theta}^{\,n+1}-(k-1)\overline{\theta}^{\,n}
\bigr)
-
\widehat{\mathbf{u}}^{\,n}\cdot\nabla\widehat{\theta}^{\,n}
-\alpha\widehat{v}^{\,n}
+g^{\,n+k},
\label{eq:GSAV_Bouss_compact_b_skew_corr}\\
(\nabla p^{n+1},\nabla q)
&= \Bigl(
\mathbf{f}^{\,n+1}
+\overline{\theta}^{\,n+1}\mathbf{e}_{2}
-\overline{\mathbf{u}}^{\,n+1}
\cdot\nabla\overline{\mathbf{u}}^{\,n+1}
-\nu\nabla\times\nabla\times\overline{\mathbf{u}}^{\,n+1},
\nabla q
\Bigr),
\quad\forall\,q\in H^{1}(\Omega),
\label{eq:GSAV_Bouss_compact_c_corr}\\
\frac{r^{n+1}-r^{n}}{\delta t}
&= \frac{r^{n+1}}{\overline{E}^{n+1} + \overline{C}}
\Bigl(
-\nu\|\nabla\overline{\mathbf{u}}^{\,n+1}\|^{2}
-\frac{\gamma}{\alpha}
\|\nabla\overline{\theta}^{\,n+1}\|^{2}
+(\mathbf{f}^{\,n+1},\overline{\mathbf{u}}^{\,n+1})
+\frac{1}{\alpha}
(g^{\,n+1},\overline{\theta}^{\,n+1})
\Bigr);
\label{eq:GSAV_Bouss_compact_d_corr}\\
\xi^{n+1}
&= \frac{r^{n+1}}{E(\overline{\mathbf{u}}^{\,n+1},
\overline{\theta}^{\,n+1})+\overline{C}},
\qquad
\eta^{n+1}
= 1-(1-\xi^{n+1})^{2},
\label{eq:GSAV_Bouss_compact_e_corr}\\
\mathbf{u}^{n+1}
&= \eta^{n+1}\overline{\mathbf{u}}^{\,n+1},
\qquad
\theta^{\,n+1}
= \eta^{n+1}\overline{\theta}^{\,n+1},
\label{eq:GSAV_Bouss_compact_f_corr}
\end{align}
where 
$\widehat{v}^{\,n}$ is the second component
of $\widehat{\mathbf{u}}^{\,n}$ and 
$\overline{E}^{n+1}
= E(\overline{\mathbf{u}}^{\,n+1},\overline{\theta}^{\,n+1})$.

\subsection{Weak stability}
\label{subsec:weak_stability}

\begin{theorem}[Weak stability of the Boussinesq
GSAV scheme]
\label{thm:weak_stability}
Take $\delta t_{\max}>0$. Let $T>0$, $0<\delta t<\delta t_{\max}$, and $N=\lfloor T/\delta t\rfloor$.
Suppose the forcing terms satisfy
$\sup_{t\in[0,T]}\|\mathbf{f}(\cdot,t)\|\le C_{f}$
and $\sup_{t\in[0,T]}\|g(\cdot,t)\|\le C_{g}$,
and let $\overline{C}$ be any constant satisfying
\begin{equation}\label{eq:Cbar_choice}
\overline{C}
\ge \max\!\left\{
8\delta t^{2}_{\max} C_{f}^{2},\;
8C_{f}^{2},\;
\frac{8\delta t^{2}}{\alpha}C_{g}^{2},\;
\frac{8}{\alpha}C_{g}^{2},\;
1
\right\}.
\end{equation}
Then the following statements hold for the
scheme \eqref{eq:GSAV_Bouss_compact_a_corr}--\eqref{eq:GSAV_Bouss_compact_f_corr}.

\begin{enumerate}[label=\emph{(\alph*)},
leftmargin=*, widest=b]

\item \emph{(Positivity)}
The sequences $\{r^{n}\}$ and $\{\xi^{n}\}$
remain nonnegative throughout the computation:
\[
r^{n}\ge 0
\quad\text{and}\quad
\xi^{n}\ge 0,
\qquad \text{for all }1\le n\le N.
\]

\item \emph{(Uniform bound)}
There exists a constant $M_{T}>0$ independent of $\delta t$ and $n$, such that
\begin{equation}\label{eq:weak_stab_bound}
\nu\delta t\sum_{j=0}^{n}
\xi^{j}\|\nabla\overline{\mathbf{u}}^{\,j}\|^{2}
+\frac{\gamma}{\alpha}\delta t\sum_{j=0}^{n}
\xi^{j}\|\nabla\overline{\theta}^{\,j}\|^{2}
+\|\mathbf{u}^{n}\|
+\|\theta^{n}\|
+r^{n}
\le M_{T},
\end{equation}
for all $1\le n\le N$.

\end{enumerate}
\end{theorem}

\begin{remark}
The proof uses only the GSAV relations \eqref{eq:GSAV_Bouss_compact_d_corr}--\eqref{eq:GSAV_Bouss_compact_f_corr}, and does not invoke the discrete update equations \eqref{eq:GSAV_Bouss_compact_a_corr}, \eqref{eq:GSAV_Bouss_compact_b_skew_corr}, or \eqref{eq:GSAV_Bouss_compact_c_corr}. The result is therefore independent of the spatial and temporal discretizations and holds for any time step $0<\delta t <\delta t_{\max}$.
The bound is termed \emph{weak} because it  remains trivially true when $\|\overline{\bm u}^n\|$ becomes unbounded: in such a blow-up scenario, $r^n$, $\xi^n$, and $\eta^n$ all vanish, and the inequality \eqref{eq:weak_stab_bound} reduces to $0 \le M_T$. Blow-up of this kind is indeed observed for the Navier--Stokes equations in \cite{AlhomsiWuZhengSubmitted}.
\end{remark}
\begin{proof}

Part (a).
First,  $r^{1}=E(\mathbf{u}^{1},\theta^{1})+\overline{C}
\ge\overline{C}\ge 1>0$ based on the choice of $\overline{C}$.
Assume $r^n\ge 0$ for any $1\le n\le N-1$. 
By the Cauchy--Schwarz inequality and
Young's inequality with the choice
\eqref{eq:Cbar_choice}, one obtains
\begin{equation}\label{eq:force_bounds_dt}
\left|
\frac{\delta t\,(\mathbf{f}^{n+1},\overline{\mathbf{u}}^{\,n+1})}
{\overline{E}^{n+1}+\overline{C}}
\right|\le\frac{1}{4},
\qquad
\left|
\frac{\frac{\delta t}{\alpha}
(g^{n+1},\overline{\theta}^{\,n+1})}
{\overline{E}^{n+1}+\overline{C}}
\right|\le\frac{1}{4},
\end{equation}
and, without the factor $\delta t$,
\begin{equation}\label{eq:force_bounds_nodt}
\left|
\frac{(\mathbf{f}^{n+1},\overline{\mathbf{u}}^{\,n+1})}
{\overline{E}^{n+1}+\overline{C}}
\right|\le\frac{1}{4},
\qquad
\left|
\frac{\frac{1}{\alpha}
(g^{n+1},\overline{\theta}^{\,n+1})}
{\overline{E}^{n+1}+\overline{C}}
\right|\le\frac{1}{4}.
\end{equation}
Rearranging \eqref{eq:GSAV_Bouss_compact_d_corr}
gives
\begin{equation}\label{eq:r_update_den}
r^{n+1}
= \frac{r^{n}}{1+\delta t\,D^{n+1}},
\end{equation}
where
\begin{equation}\label{eq:Dn_def}
D^{n+1}
= \frac{\nu\|\nabla\overline{\mathbf{u}}^{\,n+1}\|^{2}
+\frac{\gamma}{\alpha}\|\nabla\overline{\theta}^{\,n+1}\|^{2}}
{\overline{E}^{n+1}+\overline{C}}
-\frac{(\mathbf{f}^{n+1},\overline{\mathbf{u}}^{\,n+1})
+\frac{1}{\alpha}(g^{n+1},\overline{\theta}^{\,n+1})}
{\overline{E}^{n+1}+\overline{C}}.
\end{equation}
Applying \eqref{eq:force_bounds_dt}
to the forcing part of $D^{n+1}$:
\[
\delta t\,D^{n+1}
\ge 0 - \frac{1}{4} - \frac{1}{4}
= -\frac{1}{2},
\]
so $1+\delta t\,D^{n+1}\ge\frac{1}{2}>0$.
Therefore $r^{n+1}\ge 0$. By induction, $r^n\ge 0$ for all $1\le n\le N$.
Because $\overline{E}^{n}+\overline{C}\ge\overline{C}\ge 1$,
we have
$\xi^{n}=r^{n}/(\overline{E}^{n}+\overline{C})\ge 0$ for all $1\le n\le N$.

\medskip
Part (b).
Rewriting \eqref{eq:GSAV_Bouss_compact_d_corr}
in increment form yields, for $1\le j\le N-1$:
\begin{equation}\label{eq:r_increment}
r^{j+1}-r^{j}
= -\delta t\,\xi^{j+1}\mathcal{L}^{j+1}
+\delta t\,\xi^{j+1}\!\left(
(\mathbf{f}^{j+1},\overline{\mathbf{u}}^{\,j+1})
+\frac{1}{\alpha}
(g^{j+1},\overline{\theta}^{\,j+1})
\right),
\end{equation}
where the dissipative term 
\begin{align*}
\mathcal{L}^{j+1}
\triangleq \nu\|\nabla\overline{\mathbf{u}}^{\,j+1}\|^{2}
+\frac{\gamma}{\alpha}
\|\nabla\overline{\theta}^{\,j+1}\|^{2}
\ge 0.
\end{align*}
Summing up over $j=1,\cdots, n-1$ and dropping the dissipative terms  yields 
\begin{eqnarray}\label{eq:r_increment_upper}
r^{n}-r^1
&\le & 
\delta t\, \xi^n 
\Big(
(\mathbf{f}^{n},\overline{\mathbf{u}}^{\,n})
+ \frac{1}{\alpha}
(g^{n},\overline{\theta}^{\,n})
\Big)
+ \delta t
\sum_{j=2}^{n-1} 
\xi^j
\Big(
(\mathbf{f}^{j},\overline{\mathbf{u}}^{\,j})
+ \frac{1}{\alpha}
(g^{j},\overline{\theta}^{\,j})
\Big)
\nonumber \\
&=& 
\delta t\, r^n
\frac{\Big(
(\mathbf{f}^{n},\overline{\mathbf{u}}^{\,n})
+ \frac{1}{\alpha}
(g^{n},\overline{\theta}^{\,n})
\Big)}
{\overline{E}^{n}+\overline{C}}
+ 
\delta t\, \sum_{j=2}^{n-1} r^j
\frac{\Big(
(\mathbf{f}^{j},\overline{\mathbf{u}}^{\,j})
+ \frac{1}{\alpha}
(g^{j},\overline{\theta}^{\,j})
\Big)}
{\overline{E}^{j}+\overline{C}}
\nonumber\\
&\le &
\frac{1}{2} r^n + \frac{\delta t}{2} \sum_{j=2}^{n-1} r^j.
\end{eqnarray}
By the Gr\"onwall Lemma\,\ref{lem:discrete_gronwall_2}, it implies that 
\begin{equation}\label{eq:r_bound}
r^{n}\le 2\exp(T) \cdot r^1
\triangleq C_{T}\cdot r^{1},
\qquad 1\le n\le N.
\end{equation}

\medskip

Summing \eqref{eq:r_increment} over $j=1,\ldots,n$:
\begin{equation}\label{eq:diss_sum}
\delta t\sum_{j=1}^{n}
\xi^{j+1}\mathcal{L}^{j+1}
= r^{1}-r^{n+1}
+\delta t\sum_{j=1}^{n}\xi^{j+1}
\!\left(
(\mathbf{f}^{j+1},\overline{\mathbf{u}}^{\,j+1})
+\frac{1}{\alpha}(g^{j+1},\overline{\theta}^{\,j+1})
\right).
\end{equation}
Using the estimate \eqref{eq:force_bounds_nodt}, 
the identity 
$\xi^{j+1}(\overline{E}^{j+1}+\overline{C})=r^{j+1}\ge 0$,
and the bound \eqref{eq:r_bound}
 yields 
\[
\delta t\sum_{j=1}^{n}\xi^{j+1}
\left(|(\mathbf{f}^{j+1},\overline{\mathbf{u}}^{\,j+1})|
+\frac{1}{\alpha}|(g^{j+1},\overline{\theta}^{\,j+1})|
\right)
\le\frac{\delta t}{2}\sum_{j=1}^{n}r^{j+1}
\le\frac{T}{2}\cdot C_{T}r^{1}.
\]
Substituting into \eqref{eq:diss_sum} and using $r^{n+1}\ge 0$ implies
\begin{equation}\label{eq:dissipative_bound}
\nu\delta t\sum_{j=1}^{n}
\xi^{j+1}\|\nabla\overline{\mathbf{u}}^{\,j+1}\|^{2}
+\frac{\gamma}{\alpha}\delta t\sum_{j=1}^{n}
\xi^{j+1}\|\nabla\overline{\theta}^{\,j+1}\|^{2}
\le r^{1}\!\left(1+\frac{T}{2}C_{T}\right)
\triangleq D_{T}.
\end{equation}

\medskip

From Part (a), \eqref{eq:r_bound}, and
$\overline{E}^{n+1}+\overline{C}
\ge\frac{1}{2}\|\overline{\mathbf{u}}^{\,n+1}\|^{2}+\overline{C}$, one obtains
\begin{equation}\label{eq:xi_bound_u}
0\le\xi^{n+1}
= \frac{r^{n+1}}{\overline{E}^{n+1}+\overline{C}}
\le \frac{C_{T}\,r^{1}}
{\frac{1}{2}\|\overline{\mathbf{u}}^{\,n+1}\|^{2}+\overline{C}}
\le \frac{2C_{T}\,r^{1}}
{\|\overline{\mathbf{u}}^{\,n+1}\|^{2}+2},
\end{equation}
where $\overline{C}\ge 1$ is used in the last step.
Similarly, from
$\overline{E}^{n+1}+\overline{C}
\ge\frac{1}{2\alpha}\|\overline{\theta}^{\,n+1}\|^{2}
+\overline{C}$, it follows that
\begin{equation}\label{eq:xi_bound_theta}
0\le\xi^{n+1}
\le \frac{2\alpha C_{T}r^{1}}
{\|\overline{\theta}^{\,n+1}\|^{2}+2\alpha}.
\end{equation}
Since $\eta^{n+1}=\xi^{n+1}(2-\xi^{n+1})$ and $\xi^{n+1}\ge 0$, 
we have
\begin{equation}\label{eq:eta_abs_bound}
|\eta^{n+1}|
= \xi^{n+1}|2-\xi^{n+1}|
\le \xi^{n+1}(2+\xi^{n+1}).
\end{equation}

Combining \eqref{eq:GSAV_Bouss_compact_f_corr}, \eqref{eq:eta_abs_bound}, and \eqref{eq:xi_bound_u} yields
\begin{equation}\label{eq:u_norm_pre}
\|\mathbf{u}^{n+1}\|
= \| \eta^{n+1} \overline{\mathbf{u}}^{\,n+1}\|
\le \xi^{n+1}(2+\xi^{n+1})
\|\overline{\mathbf{u}}^{\,n+1}\|
\le \frac{2C_{T}r^{1}}
{\|\overline{\mathbf{u}}^{\,n+1}\|^{2}+2}
\!\left(
2+\frac{2C_{T}r^{1}}
{\|\overline{\mathbf{u}}^{\,n+1}\|^{2}+2}
\right)
\|\overline{\mathbf{u}}^{\,n+1}\|.
\end{equation}
Setting $s=\|\overline{\mathbf{u}}^{\,n+1}\|\ge 0$
and noting that the function
$h(s)=\frac{s}{s^{2}+2}$ satisfies
$h(s)\le\frac{1}{2\sqrt{2}}$ for all $s\ge 0$
(the maximum is attained at $s=\sqrt{2}$),
we obtain
\begin{equation}\label{eq:u_L2_bound_final}
\|\mathbf{u}^{n+1}\|
\le 2C_{T}\,r^{1}
\cdot\frac{1}{2\sqrt{2}}
\cdot\!\left(2+C_{T}\,r^{1}\right)
=
\frac{C_{T}\,r^{1}}{\sqrt{2}}
\cdot\!\left(2+C_{T}\,r^{1}\right)
\triangleq M_{T,u}.
\end{equation}

Similarly, we can get from $\theta^{n+1}=\eta^{n+1}\overline{\theta}^{\,n+1}$ that 
\[
\|\theta^{n+1}\|
\le \xi^{n+1}(2+\xi^{n+1})
\|\overline{\theta}^{\,n+1}\|.
\]
Setting $y=\|\overline{\theta}^{\,n+1}\|$ and
substituting \eqref{eq:xi_bound_theta},
we use $\frac{y}{y^{2}+2\alpha}
\le\frac{1}{2\sqrt{2\alpha}}$
(maximum at $y=\sqrt{2\alpha}$) to get
\begin{equation}\label{eq:theta_L2_bound_final}
\|\theta^{n+1}\|
\le 2\alpha C_{T}\,r^{1}
\cdot\frac{1}{2\sqrt{2\alpha}}
\cdot\!\left(2+\alpha C_{T}\,r^{1}\right)
= \sqrt{\frac{\alpha}{2}}\,C_{T}\,r^{1}
(2+\alpha C_{T}r^{1})
\triangleq M_{T,\theta}.
\end{equation}

Define
\begin{equation}\label{eq:MT_choice}
M_{T}
= D_{T}
+M_{T,u}
+M_{T,\theta}
+C_{T}\,r^{1}
+ \nu\, \delta t_{\max} \sum_{j=0}^1 \xi^j \|\nabla\overline{\bf u}^j\|^2
+ \frac{\gamma\,\delta t_{\max}}{\alpha} \sum_{j=0}^1 \xi^j 
\|\nabla\overline{\theta}^j\|^2.
\end{equation}
Here, $\overline{\bf u}^0$, $\overline{\bf u}^1$, 
$\overline{\theta}^0$, $\overline{\theta}^1$ are defined as the initial condition and the solution at the first time step respectively, and $\xi^0=\xi^1=1$. 
Collecting \eqref{eq:r_bound},
\eqref{eq:dissipative_bound},
\eqref{eq:u_L2_bound_final},
and \eqref{eq:theta_L2_bound_final}
establishes \eqref{eq:weak_stab_bound}.
The constants $C_{T}$, $D_{T}$, $M_{T,u}$,
$M_{T,\theta}$, and $M_{T}$  are independent of $\delta t$ and $n$.
This completes the proof.

\end{proof}


\section{Error Analysis}
\label{sec:error}
\begin{theorem}\label{thm:error_2d_general_k}
Let \(d=2\) and assume that the exact solution \((\mathbf u,\theta,p)\) of the Boussinesq system satisfies the regularity
$\partial_t\mathbf u, \partial_t\theta\in L^2(0,T;H^1)$, 
$\partial_{tt}\mathbf u, \partial_{tt}\theta\in L^2(0,T;H^2)$,
$\partial_{ttt}\mathbf u,  \partial_{ttt}\theta\in L^2(0,T;L^2)$, 
$\partial_{tt}p\in L^2(0,T;H^1)$.
Let \(\overline{\mathbf u}^{n+1},\mathbf u^{n+1},\overline{\theta}^{\,n+1},\theta^{\,n+1},p^{n+1}\) be computed by the numerical scheme \eqref{eq:GSAV_Bouss_compact_a_corr} to \eqref{eq:GSAV_Bouss_compact_f_corr} with a parameter \(k\ge 4\). 
Assume $\overline{C}$ satisfies \eqref{eq:Cbar_choice}. 
Define the errors
\[
\overline{\boldsymbol{e}}^n=\overline{\mathbf u}^n-\mathbf u(t^n),\quad
\boldsymbol{e}^n=\mathbf u^n-\mathbf u(t^n),\quad
\overline e_{\theta}^n=\overline{\theta}^n-\theta(t^n),\quad
e_{\theta}^n=\theta^n-\theta(t^n),\quad
e_p^n=p^n-p(t^n).
\]
Then there exists $C_0>0$ such that when $\delta t\le \frac{1}{1+2C_0^2}$, 
\begin{align}
&\|\nabla\overline{\boldsymbol{e}}^{n+1}\|^2+\|\nabla\boldsymbol{e}^{n+1}\|^2
+\|\nabla\overline e_{\theta}^{n+1}\|^2+\|\nabla e_{\theta}^{n+1}\|^2\notag \\
&\qquad +\delta t\sum_{i=0}^{n+1}\Bigl(\|\Delta\overline{\boldsymbol{e}}^i\|^2+\|\Delta\boldsymbol{e}^i\|^2
+\|\Delta\overline e_{\theta}^i\|^2+\|\Delta e_{\theta}^i\|^2+\|\nabla e_p^i\|^2\Bigr)\le C_{\mathrm{Bou}}\delta t^4,
\label{error_analysis}
\end{align}
for all \(n+1\le T/\delta t\) where the constant $C_{\mathrm{Bou}}>0$ is independent of $\delta t$ but depends on $\nu^{-1}$ and $\gamma^{-1}$ (see Remark\,\ref{rem:C53_dependence}).
\end{theorem}

\begin{proof}
Because $\overline{C}$ satisfies \eqref{eq:Cbar_choice}, the conclusions of Theorem\,\ref{thm:weak_stability} hold. 
Similar to the proof of Theorem~7 of \cite{HuangShen2023},
the proof proceeds by induction on $n$, establishing the bound $|1-\xi^i|\le C_0\delta t$ for all $i\le \frac{T}{\delta t}$, where the value of $C_0$ is determined in \eqref{eq:C0_choice_step3}. 
The base case $n=1$ follows from the second-order initialization \cite{guermond2003new}, and the values of $C_0$ and $\delta t$ are chosen below to make this assumption hold for $i=1$. For induction, we assume 
\begin{equation}
C_0>1 \text{ and } |1-\xi^i|\le C_0\delta t,
\quad \forall i\le n, 
\label{induction_assumption}
\end{equation}
and the goal is to prove $|1-\xi^{n+1}|\le C_0\delta t$. In the following, the proof is divided into three steps.

\subsection{Proof Step 1: bounds for $\|\nabla\overline{\mathbf{u}}^{m}\|$, $\|\Delta\overline{\mathbf{u}}^{m}\|$, $\|\nabla\overline{\theta}^{m}\|$, $\|\Delta\overline{\theta}^{m}\|$  $\forall 0\le m \le n+1$.}
Choose $\delta t$ small enough so that
\begin{equation}\label{eq:dt_small_step1}
\delta t \le \min\!\left\{
\frac{1}{2C_0^{2}},\, 1
\right\}.
\end{equation}
Then, by Lemma~\ref{lemma_xi_eta}, we have
\begin{equation}\label{eq:eta_bounds_step1}
\frac{1}{2} < \xi^{i} < \frac{3}{2},
\qquad
\frac{3}{4} < \eta^{i} \le 1,
\qquad \forall\, 0 \le i \le n.
\end{equation}
Since $\mathbf{u}^{i} = \eta^{i}\overline{\mathbf{u}}^{i}$
and $\theta^{i} = \eta^{i}\overline{\theta}^{i}$,
the bounds \eqref{eq:eta_bounds_step1} and \eqref{eq:weak_stab_bound} give
\begin{equation}\label{eq:bar_unbar_L2}
\|\overline{\mathbf{u}}^{i}\|
= \frac{\|\mathbf{u}^{i}\|}{\eta^{i}}
< \frac{4}{3}\|\mathbf{u}^{i}\|
\le \frac{4}{3}M_T,
\qquad
\|\overline{\theta}^{i}\|
< \frac{4}{3}\|\theta^{i}\|
\le \frac{4}{3}M_T,
\qquad \forall\, 0 \le i \le n.
\end{equation}
From the weak stability bound
\eqref{eq:weak_stab_bound} and $\xi^{i} > \tfrac{1}{2}$ when $0\le i \le n$, we obtain
\begin{equation}\label{eq:WS_H1_step1}
\delta t \sum_{i=0}^{m}
\|\nabla\overline{\mathbf{u}}^{i}\|^{2}
\le \frac{2 M_{T}}{\nu},
\qquad
\delta t \sum_{i=0}^{m}
\|\nabla\overline{\theta}^{i}\|^{2}
\le \frac{2 \alpha M_{T}}{\gamma},
\qquad \forall\, 0\le m \le n.
\end{equation}
%
For the $H^{1}$ norms of the extrapolated temperature,
using $\widehat{\theta}^{i} = (k+1)\theta^{i} - k\theta^{i-1}$
with $\theta^{i} = \eta^{i}\overline{\theta}^{i}$
and $|\eta^{i}| \le 1$, one attains
\begin{equation}\label{eq:hat_theta_H1_step1}
\|\nabla\widehat{\theta}^{i}\|^{2}
\le 2(k+1)^{2}\|\nabla\theta^{i}\|^{2}
+ 2k^{2}\|\nabla\theta^{i-1}\|^{2}
\le 2(k+1)^{2}
\left(
\|\nabla\overline{\theta}^{i}\|^{2}
+ \|\nabla\overline{\theta}^{i-1}\|^{2}
\right).
\end{equation}
Summing \eqref{eq:hat_theta_H1_step1}
from $i = 1$ to $m$ and applying
\eqref{eq:WS_H1_step1} yields
\begin{equation}\label{eq:hat_theta_sum_step1}
\delta t \sum_{i=0}^{m}
\|\nabla\widehat{\theta}^{i}\|^{2}
\le 4(k+1)^{2} \cdot \frac{2 M_{T}\alpha}{\gamma}
=: \mathcal{B}_{\theta},
\qquad \forall\, 0\le m \le n.
\end{equation}
The constant $\mathcal{B}_{\theta}$ is independent
of $\delta t$ and $m$.

\subsubsection{Step 1A. Bounds for $\|\nabla\overline{\mathbf{u}}^{m}\|$, $\|\Delta\overline{\mathbf{u}}^{m}\|$  $\forall 0\le m \le n+1$.}
We take the $L^2$ inner product of the
momentum equation \eqref{eq:GSAV_Bouss_compact_a_corr}
at the $i$-th time step with $-\Delta\widetilde{\mathbf{u}}^{i+1}$,
where
$
\widetilde{\mathbf{u}}^{i+1}
= (k+1)\overline{\mathbf{u}}^{i+1} - k\overline{\mathbf{u}}^{i}
$,
and get
\begin{equation}\label{eq:mom_inner_step1}
I_{1} + I_{2} = I_{3} + I_{4} + I_{5} + I_{6},
\end{equation}
where
\begin{align*}
I_{1}
&= \left(
D_k\overline{\bf u}^{i+1}, 
-\Delta\widetilde{\mathbf{u}}^{i+1}
\right), \quad
I_{2}
= \nu\left(
\Delta\!\left(k\overline{\mathbf{u}}^{i+1}
- (k-1)\overline{\mathbf{u}}^{i}\right),
\Delta\widetilde{\mathbf{u}}^{i+1}
\right), \\
I_{3}
&= \left(
\widehat{\mathbf{u}}^{i} \cdot \nabla\widehat{\mathbf{u}}^{i},
\Delta\widetilde{\mathbf{u}}^{i+1}
\right),
\quad
I_{4}
= \left(
\nabla\widehat{p}^{i},
\Delta\widetilde{\mathbf{u}}^{i+1}
\right),
\quad
I_{5}
= \left(
\widehat{\theta}^{i}\,\mathbf{e}_{2},
-\Delta\widetilde{\mathbf{u}}^{i+1}
\right),
\quad
I_{6}
= \left(
\mathbf{f}^{i+k},
-\Delta\widetilde{\mathbf{u}}^{i+1}
\right).
\end{align*}

\medskip
\noindent\textit{Term $I_{1}$: time difference.}
Applying Lemma~\ref{lem:alg_identity}
with $x = \nabla\overline{\mathbf{u}}^{i+1}$,
$y = \nabla\overline{\mathbf{u}}^{i}$,
$z = \nabla\overline{\mathbf{u}}^{i-1}$,
and integrating over $\Omega$ yields
\begin{align}
I_{1}
= \frac{1}{2\delta t}
\Bigl[
&A\!\left(
\|\nabla\overline{\mathbf{u}}^{i+1}\|^{2}
- \|\nabla\overline{\mathbf{u}}^{i}\|^{2}
\right)
+ \|B\nabla\overline{\mathbf{u}}^{i+1}
- D\nabla\overline{\mathbf{u}}^{i}\|^{2}
- \|B\nabla\overline{\mathbf{u}}^{i}
- D\nabla\overline{\mathbf{u}}^{i-1}\|^{2}
\nonumber\\
&+ E\|\nabla(\overline{\mathbf{u}}^{i+1}
- \overline{\mathbf{u}}^{i})\|^{2}
- F\|\nabla(\overline{\mathbf{u}}^{i}
- \overline{\mathbf{u}}^{i-1})\|^{2}
+ G\|\nabla(\overline{\mathbf{u}}^{i+1}
- 2\overline{\mathbf{u}}^{i}
+ \overline{\mathbf{u}}^{i-1})\|^{2}
\Bigr],
\label{eq:I1_step1}
\end{align}
where the constants $A, B, D, E, F, G > 0$
 are given in
Lemma~\ref{lem:alg_identity}.
In particular, $A = \tfrac{1}{2k}$ and $E > F \ge 0$.

\medskip
\noindent\textit{Term $I_{2}$: viscous.}
Using the algebraic identity
\begin{equation}\label{eq:visc_identity}
(k a - (k-1)b)\cdot((k+1)a - kb)
= \frac{k-1}{k}\,|(k+1)a - kb|^{2}
+ \frac{1}{k}\,|a|^{2}
+ \frac{1}{2}
\!\left(|a|^{2} - |b|^{2} + |a - b|^{2}\right)
\end{equation}
with $a = \Delta\overline{\mathbf{u}}^{i+1}$,
$b = \Delta\overline{\mathbf{u}}^{i}$,
and integrating over $\Omega$, we obtain
\begin{align}
I_{2}
= \nu\frac{k-1}{k}
\|\Delta\widetilde{\mathbf{u}}^{i+1}\|^{2}
+ \frac{\nu}{k}
\|\Delta\overline{\mathbf{u}}^{i+1}\|^{2}
+ \frac{\nu}{2}
\Bigl(
\|\Delta\overline{\mathbf{u}}^{i+1}\|^{2}
- \|\Delta\overline{\mathbf{u}}^{i}\|^{2}
+ \|\Delta(\overline{\mathbf{u}}^{i+1}
- \overline{\mathbf{u}}^{i})\|^{2}
\Bigr).
\label{eq:I2_step1}
\end{align}

\medskip
\noindent\textit{Term $I_{3}$: convection.}
The following lemma is cited from \cite{AlhomsiWuZhengSubmitted}:
\begin{lemma}
\label{lem:conv_extrap}
Let $k\ge 1$ and let
$\mathbf{v}_{1},\mathbf{v}_{2}\in\mathbf{H}^{2}(\Omega)\cap\mathbf{H}^{1}_{0}(\Omega)$, with
$\|\mathbf{v}_{j}\|\le M$ for $j=1,2$. Set
$\widehat{\mathbf{v}}:=(k+1)\mathbf{v}_{1}-k\mathbf{v}_{2}$.
Then for any
$\mathbf{w}\in\mathbf{H}^{2}(\Omega)\cap\mathbf{H}^{1}_{0}(\Omega)$
and any $\varepsilon>0$, there exists a $C$ depending only on $\Omega$ such that
\begin{equation}\label{eq:conv_extrap_bound}
\bigl|(\widehat{\mathbf{v}}\cdot\nabla\widehat{\mathbf{v}},\,\Delta\mathbf{w})\bigr|
\le \varepsilon\bigl(\|\Delta\mathbf{w}\|^{2}+\|\Delta\mathbf{v}_{1}\|^{2}+\|\Delta\mathbf{v}_{2}\|^{2}\bigr)
+\frac{4C^{3}(k+1)^{6}M^{2}}{\varepsilon^{3}}
\bigl(\|\nabla\mathbf{v}_{1}\|^{4}+\|\nabla\mathbf{v}_{2}\|^{4}\bigr).
\end{equation}
\end{lemma}
Applying Lemma~\ref{lem:conv_extrap} with $\mathbf v_1=\overline{\mathbf u}^i$, $\mathbf v_2=\overline{\mathbf u}^{i-1}$, $\mathbf w=\widetilde{\mathbf u}^{i+1}$, $M=M_T$, 
for any $\varepsilon > 0$,
\begin{align}
|I_{3}|
\le\;
&\varepsilon
\Bigl(
\|\Delta\widetilde{\mathbf{u}}^{i+1}\|^{2}
+ \|\Delta\overline{\mathbf{u}}^{i}\|^{2}
+ \|\Delta\overline{\mathbf{u}}^{i-1}\|^{2}
\Bigr)
+ \frac{4C^{3}(k+1)^{6}M_{T}^{2}}{\varepsilon^{3}}
\Bigl(
\|\nabla\overline{\mathbf{u}}^{i}\|^{4}
+ \|\nabla\overline{\mathbf{u}}^{i-1}\|^{4}
\Bigr).
\label{eq:I3_step1}
\end{align}

\medskip
\noindent\textit{Term $I_{4}$: pressure.}
Applying the operation $\widehat{p}^{i} = (k+1)p^{i} - kp^{i-1}$ to  \eqref{eq:GSAV_Bouss_compact_c_corr}
with the notations $\widetilde{\mathbf{u}}^{i}
= (k+1)\overline{\mathbf{u}}^{i} - k\overline{\mathbf{u}}^{i-1}$
and $\widetilde{\theta}^{i}
= (k+1)\overline{\theta}^{i} - k\overline{\theta}^{i-1}$
we obtain
\begin{equation}
(\nabla\widehat{p}^{i}, \nabla q)
=
(\widehat{\bf f}^i + \widetilde{\theta}^i {\bf e}_2
- (k+1)\overline{\mathbf{u}}^{i}\cdot\nabla\overline{\mathbf{u}}^{i}
+ k\overline{\mathbf{u}}^{i-1}\cdot\nabla\overline{\mathbf{u}}^{i-1}
- \nu \nabla\times\nabla\times \widetilde{\bf u}^i,
\nabla q).
\end{equation}
where $\widehat{\mathbf{f}}^{i}
= (k+1)\mathbf{f}^{i} - k\mathbf{f}^{i-1}$.
Adopting the Stokes pressure identity
\eqref{Stokespressure_identity}
for ${\bf u}=\widetilde{\bf u}^i$
and taking $q=\widehat{p}^i$, we obtain 
\begin{equation}\label{eq:phat_bound_step1}
\|\nabla\widehat{p}^{i}\|
\le
\bigl\|\widehat{\mathbf{f}}^{i}
+ \widetilde{\theta}^{i}\mathbf{e}_{2}
- (k+1)\overline{\mathbf{u}}^{i}\cdot\nabla\overline{\mathbf{u}}^{i}
+ k\overline{\mathbf{u}}^{i-1}\cdot\nabla\overline{\mathbf{u}}^{i-1}\bigr\|
+ \nu\|\nabla p_{s}(\widetilde{\mathbf{u}}^{i})\|.
\end{equation}
From
\eqref{eq:bar_unbar_L2} and
\eqref{eq:weak_stab_bound},
\begin{equation}\label{eq:tilde_theta_L2_bounded}
\|\widetilde{\theta}^{i}\|
\le (k+1)\|\overline{\theta}^{i}\| + k\|\overline{\theta}^{i-1}\|
\le (2k+1)\cdot\frac{4M_{T}}{3}
=: C_{101}.
\end{equation}
By the Cauchy--Schwarz inequality applied directly
to $(\nabla\widehat{p}^{i},-\Delta\widetilde{\mathbf{u}}^{i+1})$,
followed by \eqref{eq:phat_bound_step1} and
Young's inequality applied
\emph{separately} to the two groups
with parameters $\varepsilon_{p}>0$ (for the $4$-term group)
and $\varepsilon_{2}>0$ (for the Stokes pressure group),
\begin{align}
|I_{4}|
&= |(\nabla\widehat{p}^{i},-\Delta\widetilde{\mathbf{u}}^{i+1})|
\le \|\nabla\widehat{p}^{i}\|\,\|\Delta\widetilde{\mathbf{u}}^{i+1}\|
\nonumber\\
&\le \frac{1}{4\varepsilon_{p}}
\bigl\|\widehat{\mathbf{f}}^{i}
+ \widetilde{\theta}^{i}\mathbf{e}_{2}
- (k+1)\overline{\mathbf{u}}^{i}\cdot\nabla\overline{\mathbf{u}}^{i}
+ k\overline{\mathbf{u}}^{i-1}\cdot\nabla\overline{\mathbf{u}}^{i-1}\bigr\|^{2}
+ \varepsilon_{p}\|\Delta\widetilde{\mathbf{u}}^{i+1}\|^{2}
\nonumber\\
&\quad+ \frac{\nu}{2\varepsilon_{2}}\|\nabla p_{s}(\widetilde{\mathbf{u}}^{i})\|^{2}
+ \frac{\nu\varepsilon_{2}}{2}\|\Delta\widetilde{\mathbf{u}}^{i+1}\|^{2}.
\label{eq:I4_young_step1}
\end{align}
For the $4$-term group we apply
$(a+b+c+d)^{2}\le 4(a^{2}+b^{2}+c^{2}+d^{2})$
with $a=\widehat{\mathbf{f}}^{i}$,
$b=\widetilde{\theta}^{i}\mathbf{e}_{2}$,
$c=-(k+1)\overline{\mathbf{u}}^{i}\cdot\nabla\overline{\mathbf{u}}^{i}$, and
$d=k\overline{\mathbf{u}}^{i-1}\cdot\nabla\overline{\mathbf{u}}^{i-1}$, together with
$\|\widetilde{\theta}^{i}\|\le C_{101}$ from \eqref{eq:tilde_theta_L2_bounded}:
\begin{align}
&\bigl\|\widehat{\mathbf{f}}^{i}
+ \widetilde{\theta}^{i}\mathbf{e}_{2}
- (k+1)\overline{\mathbf{u}}^{i}\cdot\nabla\overline{\mathbf{u}}^{i}
+ k\overline{\mathbf{u}}^{i-1}\cdot\nabla\overline{\mathbf{u}}^{i-1}\bigr\|^{2}
\nonumber\\
&\le 4\|\widehat{\mathbf{f}}^{i}\|^{2}
+ 4 C_{101}^{2}
+ 4(k+1)^{2}\|\overline{\mathbf{u}}^{i}\cdot\nabla\overline{\mathbf{u}}^{i}\|^{2}
+ 4k^{2}\|\overline{\mathbf{u}}^{i-1}\cdot\nabla\overline{\mathbf{u}}^{i-1}\|^{2}.
\label{eq:phat_sq_step1}
\end{align}
For the two convection terms in \eqref{eq:phat_sq_step1},
applying \eqref{convection_estimate2}, the
bound $\|\overline{\mathbf{u}}^{j}\|\le\tfrac{4}{3}M_{T}$
for $0\le j\le n$ from \eqref{eq:bar_unbar_L2}, and
Young's inequality with $\varepsilon_{1}'>0$:
\begin{align}
&4(k+1)^{2}\|\overline{\mathbf{u}}^{i}\cdot\nabla\overline{\mathbf{u}}^{i}\|^{2}
+ 4k^{2}\|\overline{\mathbf{u}}^{i-1}\cdot\nabla\overline{\mathbf{u}}^{i-1}\|^{2}
\nonumber\\
&\le 4C(k+1)^{2}\|\overline{\mathbf{u}}^{i}\|\|\nabla\overline{\mathbf{u}}^{i}\|^{2}\|\Delta\overline{\mathbf{u}}^{i}\|
+ 4Ck^{2}\|\overline{\mathbf{u}}^{i-1}\|\|\nabla\overline{\mathbf{u}}^{i-1}\|^{2}\|\Delta\overline{\mathbf{u}}^{i-1}\|
\nonumber\\
&\le \frac{16}{3}C(k+1)^{2}M_{T}\|\nabla\overline{\mathbf{u}}^{i}\|^{2}\|\Delta\overline{\mathbf{u}}^{i}\|
+ \frac{16}{3}Ck^{2}M_{T}\|\nabla\overline{\mathbf{u}}^{i-1}\|^{2}\|\Delta\overline{\mathbf{u}}^{i-1}\|
\nonumber\\
&\le \frac{64C^2 (k+1)^{4}M_{T}^{2}}{9\varepsilon_{1}'}
\Bigl(\|\nabla\overline{\mathbf{u}}^{i}\|^{4}+\|\nabla\overline{\mathbf{u}}^{i-1}\|^{4}\Bigr)
+ \varepsilon_{1}'\Bigl(\|\Delta\overline{\mathbf{u}}^{i}\|^{2}+\|\Delta\overline{\mathbf{u}}^{i-1}\|^{2}\Bigr).
\label{eq:nonlin_phat_step1}
\end{align}
For the Stokes pressure term in
\eqref{eq:I4_young_step1}, we apply
\eqref{eq:stokes_pressure_estimate} with
parameter $\varepsilon_{1} > 0$,
and then use
$\|\nabla\widetilde{\mathbf{u}}^{i}\|^{2}
\le 2(k+1)^{2}
(\|\nabla\overline{\mathbf{u}}^{i}\|^{2}
+ \|\nabla\overline{\mathbf{u}}^{i-1}\|^{2})$:
\begin{align}
\frac{\nu}{2\varepsilon_{2}}
\|\nabla p_{s}(\widetilde{\mathbf{u}}^{i})\|^{2}
&\le \frac{\nu}{2\varepsilon_{2}}
\left[
\left(\frac{1}{2} + \varepsilon_{1}\right)
\|\Delta\widetilde{\mathbf{u}}^{i}\|^{2}
+ C_{S}(\varepsilon_{1})
\|\nabla\widetilde{\mathbf{u}}^{i}\|^{2}
\right]
\nonumber\\
&\le \frac{\nu}{2\varepsilon_{2}}
\left(\frac{1}{2} + \varepsilon_{1}\right)
\|\Delta\widetilde{\mathbf{u}}^{i}\|^{2}
+ \frac{\nu(k+1)^{2}C_{S}(\varepsilon_{1})}{\varepsilon_{2}}
\Bigl(
\|\nabla\overline{\mathbf{u}}^{i}\|^{2}
+ \|\nabla\overline{\mathbf{u}}^{i-1}\|^{2}
\Bigr).
\label{eq:stokes_phat_step1}
\end{align}
Substituting \eqref{eq:nonlin_phat_step1} into
\eqref{eq:phat_sq_step1}, then plugging the result
together with \eqref{eq:stokes_phat_step1} into
\eqref{eq:I4_young_step1}, we obtain
\begin{align}
|I_4|
&\le\;
\left(\varepsilon_{p}+\frac{\nu\varepsilon_{2}}{2}\right)
\|\Delta\widetilde{\mathbf{u}}^{i+1}\|^{2}
+ \frac{\nu}{2\varepsilon_{2}}
\left(\frac{1}{2} + \varepsilon_{1}\right)
\|\Delta\widetilde{\mathbf{u}}^{i}\|^{2}
+ \frac{\varepsilon_{1}'}{4\varepsilon_{p}}
\Bigl(
\|\Delta\overline{\mathbf{u}}^{i}\|^{2}
+ \|\Delta\overline{\mathbf{u}}^{i-1}\|^{2}
\Bigr)
\nonumber \\
& + \frac{\nu(k+1)^{2}C_{S}(\varepsilon_{1})}{\varepsilon_{2}}
\Bigl(
\|\nabla\overline{\mathbf{u}}^{i}\|^{2}
+ \|\nabla\overline{\mathbf{u}}^{i-1}\|^{2}
\Bigr)
+ \frac{2C^{2}(k+1)^{4}M_{T}^{2}}{\varepsilon_{p}\varepsilon_{1}'}
\Bigl(
\|\nabla\overline{\mathbf{u}}^{i}\|^{4}
+ \|\nabla\overline{\mathbf{u}}^{i-1}\|^{4}
\Bigr)
\nonumber\\
&+ \frac{1}{\varepsilon_{p}}
\|\widehat{\mathbf{f}}^{i}\|^{2}
+ \frac{C_{101}^{2}}{\varepsilon_{p}}.
\label{eq:I4_final_step1}
\end{align}
Here the last term $\tfrac{C_{101}^{2}}{\varepsilon_{p}}$
is a bounded constant arising from
$\|\widetilde{\theta}^{i}\| \le C_{101}$.
Note that the Stokes-pressure parameter $\varepsilon_{2}$
and the $3$-term Young's parameter $\varepsilon_{p}$
are now decoupled, so the $(k+1)^{2}$ factor that
appears in the convection bound no longer multiplies
the Stokes-pressure coefficient on $\|\Delta\widetilde{\mathbf{u}}^{i}\|^{2}$.

\medskip
\noindent\textit{Term $I_{5}$: buoyancy.}
By the Cauchy--Schwarz inequality, Poincare  inequality
$\|\widetilde{\theta}^{i}\|
\le C\|\nabla\widetilde{\theta}^{i}\|$ 
and Young's inequality with $\varepsilon_{b} > 0$, 
one gets
\begin{equation}\label{eq:I5_final_step1}
|I_{5}|
\le C
\|\nabla\widehat{\theta}^{i}\|
\|\Delta\widetilde{\mathbf{u}}^{i+1}\|
\le \frac{C^{2}}{2\varepsilon_{b}}
\|\nabla\widehat{\theta}^{i}\|^{2}
+ \frac{\varepsilon_{b}}{2}
\|\Delta\widetilde{\mathbf{u}}^{i+1}\|^{2}.
\end{equation}
After multiplying by $2\delta t$ and summing
from $i = 1$ to $m$, the first term
on the right-hand side of \eqref{eq:I5_final_step1}
is controlled by
\eqref{eq:hat_theta_sum_step1}:
\begin{equation}\label{eq:I5_sum_step1}
\frac{C^{2}}{\varepsilon_{b}}
\delta t \sum_{i=1}^{m}
\|\nabla\widehat{\theta}^{i}\|^{2}
\le \frac{C^{2}\,\mathcal{B}_{\theta}}{\varepsilon_{b}},
\end{equation}
which is a bounded constant independent of
$\delta t$ and $m$.

\medskip
\noindent\textit{Term $I_{6}$: external forcing.}
By the Cauchy--Schwarz inequality and
Young's inequality with $\varepsilon_{f} > 0$,
\begin{equation}\label{eq:I6_final_step1}
|I_{6}|
\le \frac{1}{4\varepsilon_{f}}
\|\mathbf{f}^{i+k}\|^{2}
+ \varepsilon_{f}
\|\Delta\widetilde{\mathbf{u}}^{i+1}\|^{2}.
\end{equation}
Since $\|\mathbf{f}(\cdot,t)\| \le C_{f}$
for all $t \in [0,T]$, the sum
$\delta t \sum_{i=1}^{m}
\|\mathbf{f}^{i+k}\|^{2}
\le T C_{f}^{2}$ is a bounded constant.

\medskip
Substituting \eqref{eq:I1_step1}--\eqref{eq:I6_final_step1}
into \eqref{eq:mom_inner_step1} and multiplying
through by $2\delta t$ yields 
\begin{align}
&A\!\left(
\|\nabla\overline{\mathbf{u}}^{i+1}\|^{2}
- \|\nabla\overline{\mathbf{u}}^{i}\|^{2}
\right)
+ \|B\nabla\overline{\mathbf{u}}^{i+1}
- D\nabla\overline{\mathbf{u}}^{i}\|^{2}
- \|B\nabla\overline{\mathbf{u}}^{i}
- D\nabla\overline{\mathbf{u}}^{i-1}\|^{2}
\nonumber\\
&\quad
+ E\|\nabla(\overline{\mathbf{u}}^{i+1}
- \overline{\mathbf{u}}^{i})\|^{2}
- F\|\nabla(\overline{\mathbf{u}}^{i}
- \overline{\mathbf{u}}^{i-1})\|^{2}
+ G\|\nabla(\overline{\mathbf{u}}^{i+1}
- 2\overline{\mathbf{u}}^{i}
+ \overline{\mathbf{u}}^{i-1})\|^{2}
\nonumber\\
&\quad
+ \nu\delta t
\Bigl(
\|\Delta\overline{\mathbf{u}}^{i+1}\|^{2}
- \|\Delta\overline{\mathbf{u}}^{i}\|^{2}
+ \|\Delta(\overline{\mathbf{u}}^{i+1}
- \overline{\mathbf{u}}^{i})\|^{2}
\Bigr)
+ \frac{2\nu\delta t}{k}
\|\Delta\overline{\mathbf{u}}^{i+1}\|^{2}
\nonumber\\
&\quad
+ 2\nu\delta t
\left[
\frac{k-1}{k}
-\frac{\varepsilon_{2}}{2}
-\frac{\varepsilon}{\nu}
-\frac{\varepsilon_b}{2\nu}
-\frac{\varepsilon_f}{\nu}
-\frac{\varepsilon_{p}}{\nu}
\right]
\|\Delta\widetilde{\mathbf{u}}^{i+1}\|^{2}
- \frac{\nu\delta t}{\varepsilon_{2}}
\left(\frac{1}{2} + \varepsilon_{1}\right)
\|\Delta\widetilde{\mathbf{u}}^{i}\|^{2}
\nonumber\\
\le\;
&
C_{102} \,\delta t
\Bigl(
\|\nabla\overline{\mathbf{u}}^{i}\|^{4}
+ \|\nabla\overline{\mathbf{u}}^{i-1}\|^{4}
\Bigr)
+ \left[
\frac{\varepsilon_{1}'}{2\varepsilon_{p}}
+ 2\varepsilon
\right]
\delta t
\Bigl(
\|\Delta\overline{\mathbf{u}}^{i}\|^{2}
+ \|\Delta\overline{\mathbf{u}}^{i-1}\|^{2}
\Bigr)
\nonumber\\
&+ \frac{2\nu(k+1)^{2}
C_{S}(\varepsilon_{1})}{\varepsilon_{2}}
\delta t
\Bigl(
\|\nabla\overline{\mathbf{u}}^{i}\|^{2}
+ \|\nabla\overline{\mathbf{u}}^{i-1}\|^{2}
\Bigr)
+ \frac{C^{2}}{\varepsilon_{b}}\delta t
\|\nabla\widehat{\theta}^{i}\|^{2}
\nonumber\\
&
+ \frac{1}{2\varepsilon_{f}}\delta t
\|\mathbf{f}^{i+k}\|^{2}
+ \frac{2\delta t}{\varepsilon_{p}}
\|\widehat{\mathbf{f}}^{i}\|^{2}
+ \frac{2C_{101}^{2}}{\varepsilon_{p}}\delta t,
\label{eq:mom_perstep_step1}
\end{align}
where
\begin{equation}\label{eq:Cprs_def}
C_{102}
= \frac{4C^{2}(k+1)^{4}M_{T}^{2}}
{\varepsilon_{p}\varepsilon_{1}'}
+ \frac{8C^{3}(k+1)^{6}M_{T}^{2}}{\varepsilon^{3}}.
\end{equation}
We sum \eqref{eq:mom_perstep_step1} from $i=1$ to $m$ (with $m\le n$). The non-negative G-terms and  $\|\Delta(\overline{\mathbf{u}}^{i+1}-\overline{\mathbf{u}}^{i})\|^{2}$ are dropped. 
The two tilde-Laplacian sums combine to a single term on the left: $\kappa^{u} \sum_{i=1}^{m}\|\Delta\widetilde{\mathbf{u}}^{i+1}\|^{2}$ where 
  \begin{equation}
  \label{kappa_u_def}
\kappa^{u}=  2\nu\delta t\!\left[
  \frac{k-1}{k}
  - \frac{\varepsilon_{2}}{2}
  - \frac{1}{2\varepsilon_{2}}\!\left(\frac{1}{2}+\varepsilon_{1}\right)
  - \frac{\varepsilon}{\nu}
  - \frac{\varepsilon_{b}}{2\nu}
  - \frac{\varepsilon_{f}}{\nu}
  - \frac{\varepsilon_{p}}{\nu}
  \right].
  \end{equation}
To kill this term, we impose $\kappa^{u}\ge 0$. 
According to Lemma\,\ref{lem:k_condition}, 
the condition is satisfied for $k\ge 4$ by choosing
$\varepsilon_{2}=1/\sqrt{2}$ and 
$\varepsilon_{1}$, $\varepsilon_{b}$, $\varepsilon_{f}$,
$\varepsilon_{p}$, $\varepsilon > 0$
sufficiently small depending only on
$k$, $\nu$, and $\Omega$.
Under these choices, the
$\sum_{i=1}^{m}\|\Delta\widetilde{\mathbf{u}}^{i+1}\|^{2}$
contribution is nonneg\-ative on the left-hand side
and is dropped.

The bar-Laplacian bracket on the right of
\eqref{eq:mom_perstep_step1},
$\bigl[\varepsilon_{1}'/(2\varepsilon_{p})+2\varepsilon\bigr]
\delta t\sum_{i=1}^{m}(\|\Delta\overline{\mathbf{u}}^{i}\|^{2}
+\|\Delta\overline{\mathbf{u}}^{i-1}\|^{2})$,
is moved to the left and combined with
$(2\nu\delta t/k)\sum_{i=1}^{m}\|\Delta\overline{\mathbf{u}}^{i+1}\|^{2}$.
The net coefficient of
$\sum_{i=1}^{m}\|\Delta\overline{\mathbf{u}}^{i+1}\|^{2}$
on the left becomes
\begin{equation}\label{eq:stab_cond_mom_bar_step1}
C_{103}=\bigl(\tfrac{2\nu}{k}-\tfrac{\varepsilon_{1}'}{\varepsilon_{p}}
-4\varepsilon\bigr)\delta t
\end{equation}
We impose the second stability condition 
$C_{103}\;\ge\;\frac{\nu\delta t}{k}$,
which is satisfied by choosing $\varepsilon_{1}'>0$
small enough relative to $\varepsilon_{p}$ and
$\varepsilon>0$ small enough.

The buoyancy sum is bounded by \eqref{eq:I5_sum_step1}.
Therefore, the summed inequality reads
\begin{align}
&A\|\nabla\overline{\mathbf{u}}^{m+1}\|^{2}
+ \|B\nabla\overline{\mathbf{u}}^{m+1}
- D\nabla\overline{\mathbf{u}}^{m}\|^{2}
+ (E - F)
\sum_{i=1}^{m}
\|\nabla(\overline{\mathbf{u}}^{i+1}
- \overline{\mathbf{u}}^{i})\|^{2}
+ \nu\delta t
\|\Delta\overline{\mathbf{u}}^{m+1}\|^{2}
\nonumber\\
+& \frac{\nu\delta t}{k}
\sum_{i=1}^{m}
\|\Delta\overline{\mathbf{u}}^{i+1}\|^{2}
\;\le\;
2\,C_{102}
\delta t
\sum_{i=1}^{m}
\|\nabla\overline{\mathbf{u}}^{i}\|^{4}
+ \frac{4\nu(k+1)^{2}
C_{S}(\varepsilon_{1})}{\varepsilon_{2}}
\delta t
\sum_{i=1}^{m}
\|\nabla\overline{\mathbf{u}}^{i}\|^{2}
+ M_{0}^{u},
\label{eq:mom_summed_step1}
\end{align}
where
\begin{eqnarray}\label{eq:M0u_def}
M_{0}^{u}
&=& A\|\nabla\overline{\mathbf{u}}^{1}\|^{2}
+ \|B\nabla\overline{\mathbf{u}}^{1}
- D\nabla\overline{\mathbf{u}}^{0}\|^{2}
+ F\|\nabla(\overline{\mathbf{u}}^{1}
- \overline{\mathbf{u}}^{0})\|^{2}
+ \nu\delta t
\|\Delta\overline{\mathbf{u}}^{1}\|^{2}
+\frac{\nu\delta t}{\varepsilon_{2}}
\!\left(\frac{1}{2} + \varepsilon_{1}\right)
\|\Delta\widetilde{\mathbf{u}}^{1}\|^{2}
\nonumber \\
&+& 
\!\left(\frac{\varepsilon_{1}'}{2\varepsilon_{p}}
+ 2\varepsilon\right)\!\delta t
\|\Delta\overline{\mathbf{u}}^{0}\|^{2}
+ 2\!\left(\frac{\varepsilon_{1}'}{2\varepsilon_{p}}
+ 2\varepsilon\right)\!\delta t
\|\Delta\overline{\mathbf{u}}^{1}\|^{2}
+ C_{102}\,\delta t\,
\|\nabla\overline{\mathbf{u}}^{0}\|^{4}
\nonumber \\
&+&
\frac{2\nu(k+1)^{2}C_{S}(\varepsilon_{1})}{\varepsilon_{2}}
\delta t\,\|\nabla\overline{\mathbf{u}}^{0}\|^{2}
+ \frac{C^{2}
\mathcal{B}_{\theta}}{\varepsilon_{b}}
+ \frac{T C_{f}^{2}}{2\varepsilon_{f}}
+ \frac{2T(2k+1)^{2}C_{f}^{2}}{\varepsilon_{p}}
+ \frac{2C_{101}^{2}T}{\varepsilon_{p}}
\end{eqnarray}
collects all initial-data terms and all
bounded constants from the buoyancy,
forcing, and temperature-in-pressure.
So, $M_{0}^{u}$ depends only on the initial data
and fixed parameters.

\medskip
We rewrite \eqref{eq:mom_summed_step1} in the form
of Lemma~\ref{lem:discrete_gronwall_2}.
Set $\tau = \delta t$,
$a_{i}
= A\|\nabla\overline{\mathbf{u}}^{i}\|^{2}$,
$b_{i}
= \frac{\nu}{k}
\|\Delta\overline{\mathbf{u}}^{i}\|^{2}$,
$c_{i}
= \frac{4\nu(k+1)^{2}
C_{S}(\varepsilon_{1})}{\varepsilon_{2}}
\|\nabla\overline{\mathbf{u}}^{i}\|^{2}$,
$d_{i}
= 4k\,C_{102}
\|\nabla\overline{\mathbf{u}}^{i}\|^{2}$.
Then \eqref{eq:mom_summed_step1} can be rewritten as 
\begin{equation}
a_{m+1} + \tau\sum_{i=1}^{m+1} b_{i}
\le \tau \sum_{i=0}^m a_i d_i
+ \tau\sum_{i=0}^m c_i + M^u_0,
\quad 1\le m\le n.
\end{equation}
In addition, we deduce from \eqref{eq:WS_H1_step1} that
\begin{equation}\label{eq:Gronwall_exp_bound}
\tau\sum_{i=0}^{m} d_{i}
\le 4k\,C_{102}\cdot\frac{2M_{T}}{\nu}
\triangleq \mathcal{D}_{u}
\quad\text{and}\quad
\tau\sum_{i=0}^{m} c_{i}
\le \frac{4\nu(k+1)^{2}
C_{S}(\varepsilon_{1})}{\varepsilon_{2}}
\cdot\frac{2M_{T}}{\nu}
\triangleq \mathcal{C}_{u},
\end{equation}
both of which are bounded constants independent
of $\delta t$ and $m$.

Applying Gr\"onwall Lemma~\ref{lem:discrete_gronwall_2} gives
\begin{equation}\label{eq:C1tilde_vel_step1}
a_{m+1} + \tau\sum_{i=1}^{m+1} b_{i}
\le \exp\!\left(\mathcal{D}_{u}\right)
\!\left(\mathcal{C}_{u} + M_{0}^{u}\right).
\end{equation}
Dividing through by $A = \tfrac{1}{2k}$,
we obtain
\begin{equation}\label{eq:C1tilde_velocity_final}
\|\nabla\overline{\mathbf{u}}^{m+1}\|^{2}
+ 2\nu\delta t
\sum_{i=1}^{m+1}
\|\Delta\overline{\mathbf{u}}^{i}\|^{2}
\le 2k\,
\exp\!\left(\mathcal{D}_{u}\right)
\!\left(\mathcal{C}_{u} + M_{0}^{u}\right)
=: \widetilde{C}_{1}^{u},
\qquad \forall\, m \le n.
\end{equation}
Since $|\eta^{m}| \le 1$ for $0\le m\le n$ from \eqref{eq:eta_bounds_step1}, we have 
\begin{equation}
\|\nabla\mathbf{u}^{m}\|^{2}
= |\eta^{m}|^{2}
\|\nabla\overline{\mathbf{u}}^{m}\|^{2}
\le \widetilde{C}_{1}^{u}.
\end{equation}
Define
\begin{align}\label{Creg_def_C1}
C_{\mathrm{reg}}^{u} = \sup_{t\in [0,T+1]} \|\mathbf{u}(t)\|_{2}^{2},
\qquad
C_{1}^{u}  = 2(k+1)^2 \, \max \{ \widetilde{C}_1^{u}, 
 C_{\mathrm{reg}}^{u} \},
\end{align}
then we have 
\begin{equation}\label{eq:C1_bnd_for_step2}
\nu\delta t\sum_{i=0}^{n}
\!\left(\|\Delta\overline{\mathbf{u}}^{\,i}\|^{2}
+\|\Delta\widetilde{\mathbf{u}}^{\,i}\|^{2}\right), \;
\|\widehat{\mathbf{u}}(t^{i})\|_{2}^{2},\;
\|\mathbf{u}(t^{i})\|_{2}^{2}\le C_{1}^{u},
\qquad 0\le i\le n.
\end{equation}

\subsubsection{Step 1B. Temperature estimate.}

We take the $L^{2}$ inner product of the
temperature equation
\eqref{eq:GSAV_Bouss_compact_b_skew_corr}
at the $i$-th time step with $-\Delta\widetilde{\theta}^{\,i+1}$,
where
$\widetilde{\theta}^{\,i+1}
= (k+1)\overline{\theta}^{\,i+1} - k\overline{\theta}^{\,i}$.
Denoting the resulting terms
$J_{1}, \ldots, J_{5}$, the equation reads
\begin{equation}\label{eq:temp_inner_step1}
J_{1} + J_{2} = J_{3} + J_{4} + J_{5},
\end{equation}
where
\begin{align*}
J_{1}
&= \left(
D_k \overline{\theta}^{i+1},
-\Delta\widetilde{\theta}^{\,i+1}
\right),
\quad
J_{2}
= \gamma\left(
\Delta\!\left(k\overline{\theta}^{\,i+1}
- (k-1)\overline{\theta}^{\,i}\right),
\Delta\widetilde{\theta}^{\,i+1}
\right),\\[4pt]
J_{3}
&= \bigl(
\widehat{\mathbf{u}}^{\,i}\cdot\nabla\widehat{\theta}^{\,i},
\Delta\widetilde{\theta}^{\,i+1}
\bigr),
\quad
J_{4}
= \alpha\bigl(
\widehat{v}^{\,i},
-\Delta\widetilde{\theta}^{\,i+1}
\bigr),
\quad
J_{5}
= \bigl(
g^{\,i+k},
-\Delta\widetilde{\theta}^{\,i+1}
\bigr),
\end{align*}
where $\widehat{\mathbf{u}}^{\,i}
= (k+1)\mathbf{u}^{\,i} - k\mathbf{u}^{\,i-1}$,
$\widehat{\theta}^{\,i}
= (k+1)\theta^{\,i} - k\theta^{\,i-1}$,
and $\widehat{v}^{\,i}$ is the second component
of $\widehat{\mathbf{u}}^{\,i}$.

\medskip
\noindent\textit{Term $J_{1}$: time-difference.}
By Lemma~\ref{lem:alg_identity} applied with
$x = \nabla\overline{\theta}^{\,i+1}$,
$y = \nabla\overline{\theta}^{\,i}$,
$z = \nabla\overline{\theta}^{\,i-1}$
and integrating over $\Omega$:
\begin{align}
J_{1}
= \frac{1}{2\delta t}
\Bigl[
&A\!\left(
\|\nabla\overline{\theta}^{\,i+1}\|^{2}
- \|\nabla\overline{\theta}^{\,i}\|^{2}
\right)
+ \|B\nabla\overline{\theta}^{\,i+1}
- D\nabla\overline{\theta}^{\,i}\|^{2}
- \|B\nabla\overline{\theta}^{\,i}
- D\nabla\overline{\theta}^{\,i-1}\|^{2}
\nonumber\\
&+ E\|\nabla(\overline{\theta}^{\,i+1}
- \overline{\theta}^{\,i})\|^{2}
- F\|\nabla(\overline{\theta}^{\,i}
- \overline{\theta}^{\,i-1})\|^{2}
+ G\|\nabla(\overline{\theta}^{\,i+1}
- 2\overline{\theta}^{\,i}
+ \overline{\theta}^{\,i-1})\|^{2}
\Bigr].
\label{eq:J1_step1}
\end{align}

\medskip
\noindent\textit{Term $J_{2}$: diffusion.}
By the identity \eqref{eq:visc_identity}
with $a = \Delta\overline{\theta}^{\,i+1}$,
$b = \Delta\overline{\theta}^{\,i}$,
and $\gamma$ in place of $\nu$:
\begin{align}
J_{2}
= \gamma\frac{k-1}{k}
\|\Delta\widetilde{\theta}^{\,i+1}\|^{2}
+ \frac{\gamma}{k}
\|\Delta\overline{\theta}^{\,i+1}\|^{2}
+ \frac{\gamma}{2}
\Bigl(
\|\Delta\overline{\theta}^{\,i+1}\|^{2}
- \|\Delta\overline{\theta}^{\,i}\|^{2}
+ \|\Delta(\overline{\theta}^{\,i+1}
- \overline{\theta}^{\,i})\|^{2}
\Bigr).
\label{eq:J2_step1}
\end{align}

\medskip
\noindent\textit{Term $J_{3}$: convection.}
Based on \eqref{Teman_estimate} and \eqref{elliptic_regularity}, one obtains
\begin{align}
|J_{3}|
&\le C
\|\widehat{\mathbf{u}}^{\,i}\|_{2}
\|\widehat{\theta}^{\,i}\|_{1}
\|\Delta\widetilde{\theta}^{\,i+1}\|
\le C
\|\Delta\widehat{\mathbf{u}}^{\,i}\|
\|\widehat{\theta}^{\,i}\|_{1}
\|\Delta\widetilde{\theta}^{\,i+1}\|.
\label{eq:J3_holder}
\end{align}
Applying Young's inequality with
$\varepsilon_{\mathrm{tc}} > 0$:
\begin{equation}\label{eq:J3_young}
|J_{3}|
\le \varepsilon_{\mathrm{tc}}
\|\Delta\widetilde{\theta}^{\,i+1}\|^{2}
+ \frac{C^{2}}{4\varepsilon_{\mathrm{tc}}}
\|\Delta\widehat{\mathbf{u}}^{\,i}\|^{2}
\|\widehat{\theta}^{\,i}\|_{1}^{2}.
\end{equation}
Expanding using $|\eta^{i}|\le 1$ for $0\le i\le n$
 results in 
\begin{equation}\label{eq:expand_hat_u_lap}
\|\Delta\widehat{\mathbf{u}}^{\,i}\|^{2}
\le 2(k+1)^{2}
\!\left(
\|\Delta\overline{\mathbf{u}}^{\,i}\|^{2}
+\|\Delta\overline{\mathbf{u}}^{\,i-1}\|^{2}
\right),
\quad
\|\widehat{\theta}^{\,i}\|_{1}^{2}
\le 2(k+1)^{2}
\!\left(
\|\nabla\overline{\theta}^{\,i}\|^{2}
+\|\nabla\overline{\theta}^{\,i-1}\|^{2}
\right).
\end{equation}
Substituting \eqref{eq:expand_hat_u_lap}
into \eqref{eq:J3_young},
the right-hand side of \eqref{eq:J3_young}
becomes
\begin{equation}\label{eq:J3_substituted}
|J_{3}| \le
\varepsilon_{\mathrm{tc}}
\|\Delta\widetilde{\theta}^{\,i+1}\|^{2}
+\frac{C^{2}(k+1)^{4}}{\varepsilon_{\mathrm{tc}}}
\!\left(
\|\Delta\overline{\mathbf{u}}^{\,i}\|^{2}
+\|\Delta\overline{\mathbf{u}}^{\,i-1}\|^{2}
\right)
\!\left(
\|\nabla\overline{\theta}^{\,i}\|^{2}
+\|\nabla\overline{\theta}^{\,i-1}\|^{2}
\right).
\end{equation}

\medskip
\noindent\textit{Term $J_{4}$: coupling.}
Using the Cauchy-Schwarz  inequality, the Poincare inequality, and  Young's inequality with $\varepsilon_{c} > 0$ yields
\begin{equation}\label{eq:J4_final_step1}
|J_{4}|
= \alpha |(\widehat{v}^{\,i},
\Delta\widetilde{\theta}^{\,i+1})|
\le \alpha \| \widehat{\bf u}^i \|
\| \Delta\widetilde{\theta}^{\,i+1}\|
\le \frac{\alpha^{2}C^{2}}{2\varepsilon_{c}}
\|\nabla\widehat{\mathbf{u}}^{\,i}\|^{2}
+ \frac{\varepsilon_{c}}{2}
\|\Delta\widetilde{\theta}^{\,i+1}\|^{2}.
\end{equation}
After multiplying by $2\delta t$
and summing from $i=1$ to $m$,
using
$\|\nabla\widehat{\mathbf{u}}^{\,i}\|^{2}
\le 2(k+1)^{2}
(\|\nabla\overline{\mathbf{u}}^{\,i}\|^{2}
+\|\nabla\overline{\mathbf{u}}^{\,i-1}\|^{2})$
and \eqref{eq:WS_H1_step1}:
\begin{equation}\label{eq:J4_sum_step1}
\frac{\alpha^{2}C^{2}}{\varepsilon_{c}}
\delta t\sum_{i=1}^{m}
\|\nabla\widehat{\mathbf{u}}^{\,i}\|^{2}
\le\frac{4\alpha^{2}C^{2}(k+1)^{2}}{\varepsilon_{c}}
\cdot\frac{2M_{T}}{\nu}
=:\mathcal{B}_{c},
\end{equation}
which is a \emph{bounded constant}
independent of $\delta t$ and $m$.

\medskip
\noindent\textit{Term $J_{5}$: heat source.}
By the Cauchy--Schwarz inequality and Young's inequality with
$\varepsilon_{g} > 0$:
\begin{equation}\label{eq:J5_final_step1}
|J_{5}|
\le \frac{1}{4\varepsilon_{g}}
\|g^{\,i+k}\|^{2}
+ \varepsilon_{g}
\|\Delta\widetilde{\theta}^{\,i+1}\|^{2}.
\end{equation}

\medskip
Substituting \eqref{eq:J1_step1}--\eqref{eq:J5_final_step1}
into \eqref{eq:temp_inner_step1}
and multiplying by $2\delta t$,
collecting the terms
$\varepsilon_{\mathrm{tc}}
+\varepsilon_{c}/2+\varepsilon_{g}$
multiplying $\|\Delta\widetilde{\theta}^{\,i+1}\|^{2}$
onto the left-hand side via $J_2$:
\begin{align}
&A\!\left(
\|\nabla\overline{\theta}^{\,i+1}\|^{2}
-\|\nabla\overline{\theta}^{\,i}\|^{2}
\right)
+\|B\nabla\overline{\theta}^{\,i+1}
-D\nabla\overline{\theta}^{\,i}\|^{2}
-\|B\nabla\overline{\theta}^{\,i}
-D\nabla\overline{\theta}^{\,i-1}\|^{2}
\nonumber\\
&\quad
+E\|\nabla(\overline{\theta}^{\,i+1}
-\overline{\theta}^{\,i})\|^{2}
-F\|\nabla(\overline{\theta}^{\,i}
-\overline{\theta}^{\,i-1})\|^{2}
+G\|\nabla(\overline{\theta}^{\,i+1}
-2\overline{\theta}^{\,i}
+\overline{\theta}^{\,i-1})\|^{2}
\nonumber\\
&\quad
+\gamma\delta t
\!\left(
\|\Delta\overline{\theta}^{\,i+1}\|^{2}
-\|\Delta\overline{\theta}^{\,i}\|^{2}
+\|\Delta(\overline{\theta}^{\,i+1}
-\overline{\theta}^{\,i})\|^{2}
\right)
+\frac{2\gamma\delta t}{k}
\|\Delta\overline{\theta}^{\,i+1}\|^{2}
\nonumber\\
&\quad
+2\delta t
\left[
\gamma\frac{k-1}{k}
-\varepsilon_{\mathrm{tc}}
-\frac{\varepsilon_{c}}{2}
-\varepsilon_{g}
\right]
\|\Delta\widetilde{\theta}^{\,i+1}\|^{2}
\nonumber\\
\le\;&
\frac{2C^{2}(k+1)^{4}\delta t}{\varepsilon_{\mathrm{tc}}}
\!\left(
\|\Delta\overline{\mathbf{u}}^{\,i}\|^{2}
+\|\Delta\overline{\mathbf{u}}^{\,i-1}\|^{2}
\right)
\!\left(
\|\nabla\overline{\theta}^{\,i}\|^{2}
+\|\nabla\overline{\theta}^{\,i-1}\|^{2}
\right)
\nonumber\\
&+\frac{2\alpha^{2}C^{2}(k+1)^2}{\varepsilon_{c}}
\delta t
\!\left(
\|\nabla\overline{\mathbf{u}}^{\,i}\|^{2}
+\|\nabla\overline{\mathbf{u}}^{\,i-1}\|^{2}
\right)
+\frac{\delta t}{2\varepsilon_{g}}
\|g^{\,i+k}\|^{2}.
\label{eq:temp_perstep_step1}
\end{align}

The bracket in \eqref{eq:temp_perstep_step1}
is nonneg\-ative provided
\begin{equation}\label{eq:stab_cond_temp_step1}
\gamma\frac{k-1}{k}
\ge \varepsilon_{\mathrm{tc}}
+\frac{\varepsilon_{c}}{2}
+\varepsilon_{g}.
\end{equation}
This holds for $k\ge 2$ (in particular
for $k\ge 4$) with sufficiently small
$\varepsilon_{\mathrm{tc}}$,
$\varepsilon_{c}$, $\varepsilon_{g} > 0$
depending only on $k$ and $\gamma$.

Summing \eqref{eq:temp_perstep_step1}
from $i=1$ to $m$ (with $m\le n$)
and dropping the non-negative $G$ and 
 $\|\Delta\widetilde\theta^{i+1}\|^2$ terms yields
\begin{align}
&A\|\nabla\overline{\theta}^{\,m+1}\|^{2}
+\|B\nabla\overline{\theta}^{\,m+1}
-D\nabla\overline{\theta}^{\,m}\|^{2}
+\gamma\delta t
\|\Delta\overline{\theta}^{\,m+1}\|^{2}
+\frac{2\gamma\delta t}{k}
\sum_{i=1}^{m}
\|\Delta\overline{\theta}^{\,i+1}\|^{2}
\nonumber\\
\le\;&
\frac{2C^{2}(k+1)^{4}}{\varepsilon_{\mathrm{tc}}}
\delta t\sum_{i=1}^{m}
\!\left(
\|\Delta\overline{\mathbf{u}}^{\,i}\|^{2}
+\|\Delta\overline{\mathbf{u}}^{\,i-1}\|^{2}
\right)
\!\left(
\|\nabla\overline{\theta}^{\,i}\|^{2}
+\|\nabla\overline{\theta}^{\,i-1}\|^{2}
\right)
+M_{0}^{\theta}.
\label{eq:temp_summed_step1}
\end{align}
where
$M_{0}^{\theta}
= A\|\nabla\overline{\theta}^{\,1}\|^{2}
+\|B\nabla\overline{\theta}^{\,1}
-D\nabla\overline{\theta}^{\,0}\|^{2}
+F\|\nabla(\overline{\theta}^{\,1}
-\overline{\theta}^{\,0})\|^{2}
+\gamma\delta t
\|\Delta\overline{\theta}^{\,1}\|^{2}
+\mathcal{B}_{c}
+\frac{TC_{g}^{2}}{2\varepsilon_{g}}$.
Writing $E_{i}=\|\Delta\overline{\mathbf{u}}^{\,i}\|^{2}
+\|\Delta\overline{\mathbf{u}}^{\,i-1}\|^{2}$
and using the conventions $E_{0}=0$ and $E_{m+1}:=0$,
the sum on the right side of
\eqref{eq:temp_summed_step1} regroups as
\begin{equation}\label{eq:rect_identity}
\sum_{i=1}^{m} E_{i}
\bigl(\|\nabla\overline{\theta}^{\,i}\|^{2}
+\|\nabla\overline{\theta}^{\,i-1}\|^{2}\bigr)
= E_{1}\|\nabla\overline{\theta}^{\,0}\|^{2}
+\sum_{i=1}^{m-1}(E_{i}+E_{i+1})
\|\nabla\overline{\theta}^{\,i}\|^{2}
+ E_{m}\|\nabla\overline{\theta}^{\,m}\|^{2}.
\end{equation}
With $A=\tfrac{1}{2k}$ from
Lemma~\ref{lem:alg_identity}, define for
$i=0,1,\ldots,m+1$ the following quantities:
$a_{i}^{\theta}
= A\|\nabla\overline{\theta}^{\,i}\|^{2}$,
$b_{i}^{\theta}
= \tfrac{2\gamma}{k}\|\Delta\overline{\theta}^{\,i}\|^{2}$,
$c_{i}^{\theta}=0$,
$d_{i}^{\theta}
= \frac{2C^{2}(k+1)^{4}}{A\,\varepsilon_{\mathrm{tc}}}
(E_{i}+E_{i+1})$, 
$C^{\theta}=M_{0}^{\theta}
+\tfrac{2\gamma\delta t}{k}\|\Delta\overline{\theta}^{\,1}\|^{2}$.
Substituting \eqref{eq:rect_identity} into
\eqref{eq:temp_summed_step1}, dropping some 
non-negative terms, we obtain
\begin{equation}\label{eq:gronwall_form}
a_{m+1}^{\theta}
+\delta t\sum_{i=1}^{m+1} b_{i}^{\theta}
\le \delta t\sum_{i=0}^{m} a_{i}^{\theta}\,d_{i}^{\theta}
+\delta t\sum_{i=0}^{m} c_{i}^{\theta}
+ C^{\theta}.
\end{equation}
Based on \eqref{eq:C1tilde_velocity_final},
one obtains the following inequality
\begin{equation}\label{eq:Dtheta_final}
\delta t\sum_{i=0}^{m} d_{i}^{\theta}
\le \frac{8C^{2}(k+1)^{4}\,\widetilde{C}_{1}^{u}}
{A\,\nu\,\varepsilon_{\mathrm{tc}}}
=: \mathcal{D}_{\theta},
\end{equation}
which is a bounded constant independent of
$\delta t$ and $m$.
Applying Gr\"onwall Lemma~\ref{lem:discrete_gronwall_2} and
dividing by $A=\tfrac{1}{2k}$:
\begin{equation}\label{eq:C1tilde_theta_final}
\|\nabla\overline{\theta}^{\,m+1}\|^{2}
+4k\gamma\delta t
\sum_{i=1}^{m+1}
\|\Delta\overline{\theta}^{\,i}\|^{2}
\le 2k\,e^{\mathcal{D}_{\theta}}
C^{\theta}
=: \widetilde{C}_{1}^{\theta},
\qquad\forall\,m \le n.
\end{equation}
Since $|\eta^{m}| \le 1$ for $0\le m\le n$ from \eqref{eq:eta_bounds_step1},
\begin{equation}
\|\nabla\theta^{m}\|^{2}
=|\eta^m|^2 \|\nabla\overline{\theta}^{m}\|^{2}
\le\widetilde{C}_{1}^{\theta}.
\end{equation}
Let 
\begin{equation}\label{C_1theta_constants}
C_{1}^{\theta}
= 2(k+1)^2 \max\!\left\{ \widetilde{C}_{1}^{\theta},\; C_{\mathrm{reg}}^{\theta} \right\},
\qquad
C_{\mathrm{reg}}^{\theta} := \sup_{t\in [0,T+1]} \|\theta(t)\|_2^2.
\end{equation}
Then
\begin{equation}
\gamma\delta t \, \sum_{i=0}^n \bigl(
\|\Delta\bar{\theta}^i\|^2 + \|\Delta\widetilde{\theta}^i\|^2 \bigr), \quad
\|\widehat{\theta}(t^i)\|_2^2, \quad
\|\theta(t^i)\|_2^2, \le C_{1}^{\theta},
\quad 0\le i\le n.
\label{C_1theta_def}
\end{equation}

\subsection{Proof Step 2. Error bounds for intermediate quantities}
We derive the error equations by subtracting the
exact Boussinesq system from the numerical scheme.
According to the nomenclature system in \eqref{hat_def} and \eqref{tilde_def}, one has the tilde/hat errors
\begin{equation}\label{eq:error_notation}
\begin{aligned}
&\widetilde{\mathbf{e}}^{\,i} = (k+1)\overline{\mathbf{e}}^{\,i} - k\overline{\mathbf{e}}^{\,i-1}, \quad
\widehat{\mathbf{e}}_{u}^{\,i} = (k+1)\mathbf{e}^{\,i} - k\mathbf{e}^{\,i-1}, \quad
\widehat{e}_{p}^{\,i} = (k+1)e_{p}^{\,i} - ke_{p}^{\,i-1},
\\
&\widetilde{e}_{\theta}^{\,i} = (k+1)\overline{e}_{\theta}^{\,i} - k\overline{e}_{\theta}^{\,i-1},
\quad
\widehat{e}_{\theta}^{\,i} = (k+1)e_{\theta}^{\,i} - ke_{\theta}^{\,i-1}.
\end{aligned}
\end{equation}
Since $\mathbf{u}^{\,i} = \eta^{\,i}\overline{\mathbf{u}}^{\,i}$,
we have
$\mathbf{e}^{\,i}
= \eta^{\,i}\overline{\mathbf{e}}^{\,i}
+(\eta^{\,i}-1)\mathbf{u}(t^{i})$.
From the induction hypothesis 
 \eqref{induction_assumption}
and $1-\eta^{\,i}=(1-\xi^{\,i})^{2}$, we obtain
\begin{equation}\label{eq:unbar_to_bar}
|\eta^{\,i}-1|
= (1-\xi^{\,i})^{2}
\le C_{0}^{2}\delta t^{2},
\quad
\forall i\le n.
\end{equation}
Therefore, using $|\eta^{\,i}|\le 1$, $\forall i\le n$, in \eqref{eq:eta_bounds_step1} and 
$\|\nabla\overline{\mathbf{u}}^{\,i}\|^{2}\le \tilde{C}_{1}^{u}$,
$\forall i\le n+1$, in \eqref{eq:C1tilde_velocity_final} yields $\forall i\le n$, 
\begin{align}
\|\nabla\widehat{\mathbf{e}}_{u}^{\,i}\|^{2}
&\le 4(k+1)^{2}\|\nabla\overline{\mathbf{e}}^{\,i}\|^{2}
+4k^{2}\|\nabla\overline{\mathbf{e}}^{\,i-1}\|^{2}
+4(k+1)^{2}C_{0}^{4}\delta t^{4}
\!\left(
\|\nabla\overline{\mathbf{u}}^{\,i}\|^{2}
+\|\nabla\overline{\mathbf{u}}^{\,i-1}\|^{2}
\right)\nonumber\\
&\le C(k)
\!\left(
\|\nabla\overline{\mathbf{e}}^{\,i}\|^{2}
+\|\nabla\overline{\mathbf{e}}^{\,i-1}\|^{2}
\right)
+8(k+1)^{2}C_{0}^{4}\tilde{C}_{1}^{u}\delta t^{4}.
\label{eq:hat_eu_bound}
\end{align}
An identical argument gives
\begin{equation}\label{eq:hat_etheta_bound}
\|\nabla\widehat{e}_{\theta}^{\,i}\|^{2}
\le C(k)
\!\left(
\|\nabla\overline{e}_{\theta}^{\,i}\|^{2}
+\|\nabla\overline{e}_{\theta}^{\,i-1}\|^{2}
\right)
+8(k+1)^{2}C_{0}^{4}\tilde{C}_{1}^{\theta}\delta t^{4}.
\end{equation}

\subsubsection{Step 2A. Velocity error estimate}
The exact solution satisfies the momentum scheme \eqref{eq:GSAV_Bouss_compact_a_corr} up to truncation residuals:
\begin{align}\label{u_Consistencyequation}
D_{k}\mathbf{u}(t^{i+1})
=&\nu\Delta\!\bigl(k\mathbf{u}(t^{i+1})-(k-1)\mathbf{u}(t^{i})\bigr)
-\widehat{\mathbf{u}}(t^{i})\cdot\nabla\widehat{\mathbf{u}}(t^{i})
-\nabla\widehat{p}(t^{i})
+\widehat{\theta}(t^{i})\mathbf{e}_{2}
+\mathbf{f}(t^{i+k}) \notag\\
&-\frac{R_{u}^{\,i}}{2\delta t}
-\nu Q_{u}^{\,i}
-P_{u}^{\,i}
-S_{u}^{\,i}
+T_{u}^{\,i},
\end{align}
with the truncation terms
\begin{align}
P_{u}^{\,i}
&= \nabla \bigl(p(t^{i+k}) - \widehat{p}(t^i) \bigr),
\label{eq:Pu}\\
Q_{u}^{\,i}
&= -\Delta\mathbf{u}(t^{i+k})
+\Delta\!\bigl(k\mathbf{u}(t^{i+1})
-(k-1)\mathbf{u}(t^{i})\bigr),
\label{eq:Qu}\\
R_{u}^{\,i}
&= -2\delta t\, 
\bigl( D_k \mathbf{u}(t^{i+1})
-\partial_t \mathbf{u}(t^{i+k}) \bigr),
\label{eq:Ru}\\
S_{u}^{\,i}
&= \mathbf{u}(t^{i+k})\cdot\nabla\mathbf{u}(t^{i+k})
-\widehat{\mathbf{u}}(t^{i})\cdot\nabla\widehat{\mathbf{u}}(t^{i}),
\label{eq:Su}\\
T_{u}^{\,i}
&= \theta(t^{i+k})\,\mathbf{e}_{2}
-\widehat{\theta}(t^{i})\,\mathbf{e}_{2},
\label{eq:Tu}
\end{align}
where $\widehat{\mathbf{u}}(t^{i})
=(k+1)\mathbf{u}(t^{i})-k\mathbf{u}(t^{i-1})$, 
$\widehat{\theta}(t^{i})
= (k+1)\theta(t^{i})-k\theta(t^{i-1})$,
and
$\widehat{p}(t^{i})
= (k+1)p(t^{i})-k p(t^{i-1})$.
Subtracting the consistency equation \eqref{u_Consistencyequation} from the momentum
scheme at step $i$ in \eqref{eq:GSAV_Bouss_compact_a_corr}, we obtain
\begin{align}
&D_k\overline{\mathbf{e}}^{\,i+1}
-\nu \Delta\!\bigl(k\overline{\mathbf{e}}^{\,i+1}
-(k-1)\overline{\mathbf{e}}^{\,i}\bigr)
\nonumber\\
&
=  -\left(
\widehat{\mathbf{u}}^{\,i}\cdot\nabla\widehat{\mathbf{u}}^{\,i}
-\widehat{\mathbf{u}}(t^{i})\cdot\nabla\widehat{\mathbf{u}}(t^{i})
\right)
-\nabla\widehat{e}_{p}^{\,i}
+\widehat{e}_{\theta}^{\,i}\,\mathbf{e}_{2}
+\frac{R_{u}^{\,i}}{2\delta t}
+\nu Q_{u}^{\,i}
+P_{u}^{\,i}
+S_{u}^{\,i}
-T_{u}^{\,i}.
\label{eq:mom_error_eq}
\end{align}

We take the $L^{2}(\Omega)$ inner product of
\eqref{eq:mom_error_eq}  with $-\Delta\widetilde{\mathbf{e}}^{\,i+1}$ and get
\begin{equation}\label{eq:mom_err_inner}
\mathcal{I}_{1} + \mathcal{I}_{2}
= \mathcal{I}_{3}+\mathcal{I}_{4}
+\mathcal{I}_{5}+\mathcal{I}_{6}
+\mathcal{I}_{7}+\mathcal{I}_{8},
\end{equation}
where
\begin{alignat*}{2}
\mathcal{I}_{1}
&= \left(
D_k\overline{\mathbf{e}}^{\,i+1}, 
-\Delta\widetilde{\mathbf{e}}^{\,i+1}
\right),
\quad&
\mathcal{I}_{2}
&= \nu\!\left(
-\Delta\!\left(k\overline{\mathbf{e}}^{\,i+1}
-(k-1)\overline{\mathbf{e}}^{\,i}\right),
-\Delta\widetilde{\mathbf{e}}^{\,i+1}
\right),\\[4pt]
\mathcal{I}_{3}
&= \left(
\widehat{\mathbf{u}}^{\,i}\cdot\nabla\widehat{\mathbf{u}}^{\,i}
-\widehat{\mathbf{u}}(t^{i})\cdot\nabla\widehat{\mathbf{u}}(t^{i}),
\Delta\widetilde{\mathbf{e}}^{\,i+1}
\right),
\quad
&
\mathcal{I}_{4}
&= \left(
\nabla\widehat{e}_{p}^{\,i},
\Delta\widetilde{\mathbf{e}}^{\,i+1}
\right),\\[4pt]
\mathcal{I}_{5}
&= \left(
\widehat{e}_{\theta}^{\,i}\,\mathbf{e}_{2},
-\Delta\widetilde{\mathbf{e}}^{\,i+1}
\right),
&
\mathcal{I}_{6}
&= \left(
P_{u}^{\,i}+\nu Q_{u}^{\,i},
-\Delta\widetilde{\mathbf{e}}^{\,i+1}
\right),\\[4pt]
\mathcal{I}_{7}
&= \tfrac{1}{2\delta t}\!\left(
R_{u}^{\,i},
-\Delta\widetilde{\mathbf{e}}^{\,i+1}
\right),
&
\mathcal{I}_{8}
&= \left(
S_{u}^{\,i}-T_{u}^{\,i},
-\Delta\widetilde{\mathbf{e}}^{\,i+1}
\right).
\end{alignat*}

\medskip
\noindent\textit{Term $\mathcal{I}_{1}$:
time-difference.}
By Lemma~\ref{lem:alg_identity} with
$x = \nabla\overline{\mathbf{e}}^{\,i+1}$,
$y = \nabla\overline{\mathbf{e}}^{\,i}$,
$z = \nabla\overline{\mathbf{e}}^{\,i-1}$,
\begin{align}
\mathcal{I}_{1}
= \frac{1}{2\delta t}\Bigl[
&A\!\left(
\|\nabla\overline{\mathbf{e}}^{\,i+1}\|^{2}
-\|\nabla\overline{\mathbf{e}}^{\,i}\|^{2}
\right)
+\|B\nabla\overline{\mathbf{e}}^{\,i+1}
-D\nabla\overline{\mathbf{e}}^{\,i}\|^{2}
-\|B\nabla\overline{\mathbf{e}}^{\,i}
-D\nabla\overline{\mathbf{e}}^{\,i-1}\|^{2}
\nonumber\\
&+E\|\nabla(\overline{\mathbf{e}}^{\,i+1}
-\overline{\mathbf{e}}^{\,i})\|^{2}
-F\|\nabla(\overline{\mathbf{e}}^{\,i}
-\overline{\mathbf{e}}^{\,i-1})\|^{2}
+G\|\nabla(\overline{\mathbf{e}}^{\,i+1}
-2\overline{\mathbf{e}}^{\,i}
+\overline{\mathbf{e}}^{\,i-1})\|^{2}
\Bigr].
\label{eq:I1_err}
\end{align}

\medskip
\noindent\textit{Term $\mathcal{I}_{2}$: viscous.}
By the identity \eqref{eq:visc_identity}
with $a=\Delta\overline{\mathbf{e}}^{\,i+1}$,
$b=\Delta\overline{\mathbf{e}}^{\,i}$,
\begin{equation}\label{eq:I2_err}
\mathcal{I}_{2}
= \nu\frac{k-1}{k}
\|\Delta\widetilde{\mathbf{e}}^{\,i+1}\|^{2}
+\frac{\nu}{k}
\|\Delta\overline{\mathbf{e}}^{\,i+1}\|^{2}
+\frac{\nu}{2}
\!\left(
\|\Delta\overline{\mathbf{e}}^{\,i+1}\|^{2}
-\|\Delta\overline{\mathbf{e}}^{\,i}\|^{2}
+\|\Delta(\overline{\mathbf{e}}^{\,i+1}
-\overline{\mathbf{e}}^{\,i})\|^{2}
\right).
\end{equation}

\medskip
\noindent\textit{Term $\mathcal{I}_{3}$: convection.}
We expand the convection difference as
\begin{equation}\label{eq:conv_expand_err}
\widehat{\mathbf{u}}^{\,i}\cdot\nabla\widehat{\mathbf{u}}^{\,i}
-\widehat{\mathbf{u}}(t^{i})\cdot\nabla\widehat{\mathbf{u}}(t^{i})
= \widehat{\mathbf{e}}_{u}^{\,i}\cdot\nabla\widehat{\mathbf{u}}^{\,i}
+\widehat{\mathbf{u}}(t^{i})\cdot\nabla\widehat{\mathbf{e}}_{u}^{\,i}.
\end{equation}
Using \eqref{Teman_estimate} and the elliptic regularity gives 
\begin{align}
|\mathcal{I}_{3}|
\le\;&
C^2
\|\nabla\widehat{\mathbf{e}}_{u}^{\,i}\|
\!\left(
\|\Delta\widehat{\mathbf{u}}^{\,i}\|
+\|\widehat{\mathbf{u}}(t^{i})\|_2
\right)
\|\Delta\widetilde{\mathbf{e}}^{\,i+1}\|.
\label{eq:I3_err_cs}
\end{align}
Applying Young's inequality with
$\varepsilon_{3} > 0$:
\begin{equation}\label{eq:I3_err_young}
|\mathcal{I}_{3}|
\le
\frac{C^4}{2\varepsilon_{3}}
\|\nabla\widehat{\mathbf{e}}_{u}^{\,i}\|^{2}
\!\left(
\|\Delta\widehat{\mathbf{u}}^{\,i}\|^{2}
+\|\widehat{\mathbf{u}}(t^{i})\|_2^{2}
\right)
+\varepsilon_{3}
\|\Delta\widetilde{\mathbf{e}}^{\,i+1}\|^{2}.
\end{equation}
Substituting \eqref{eq:hat_eu_bound} into
\eqref{eq:I3_err_young}:
\begin{align}
|\mathcal{I}_{3}|
\le\;&
\frac{C^4 C(k)}{2\varepsilon_{3}}
\!\left(
\|\nabla\overline{\mathbf{e}}^{\,i}\|^{2}
+\|\nabla\overline{\mathbf{e}}^{\,i-1}\|^{2}
\right)
\!\left(
\|\Delta\widehat{\mathbf{u}}^{\,i}\|^{2}
+\|\widehat{\mathbf{u}}(t^{i})\|_2^{2}
\right)
\nonumber\\
&+\frac{4C^4(k+1)^{2}C_{0}^{4}\tilde{C}_{1}^{u}\delta t^{4}}
{\varepsilon_{3}}
\!\left(
\|\Delta\widehat{\mathbf{u}}^{\,i}\|^{2}
+\|\widehat{\mathbf{u}}(t^{i})\|_2^{2}
\right)
+\varepsilon_{3}
\|\Delta\widetilde{\mathbf{e}}^{\,i+1}\|^{2}.
\label{eq:I3_err_final}
\end{align}

\medskip
\noindent\textit{Term $\mathcal{I}_{4}$: pressure.}
The exact pressure satisfies the variational equation (see \cite[Eq.~(1.3b)]{JohnstonLiu2004})
\begin{equation}\label{eq:press_exact}
(\nabla p(t^{i}),\nabla q)
= \!\left(
\mathbf{f}(t^{i})
+\theta(t^{i})\mathbf{e}_{2}
-\mathbf{u}(t^{i})\cdot\nabla\mathbf{u}(t^{i})
-\nu\nabla\times\nabla\times\mathbf{u}(t^{i}),
\nabla q
\right),
\quad\forall\,q\in H^{1}(\Omega),
\end{equation}
which mirrors the pressure scheme \eqref{eq:GSAV_Bouss_compact_c_corr} evaluated at the exact solution.
Subtracting \eqref{eq:press_exact} from \eqref{eq:GSAV_Bouss_compact_c_corr}
and using the Stokes pressure identity \eqref{eq:stokes_pressure_estimate} for ${\bf u}=\bar{\bf e}^i$, we have
\begin{equation}\label{eq:press_err_eq}
(\nabla e_{p}^{\,i},\nabla q)
= \!\left(
\mathbf{u}(t^{i})\cdot\nabla\mathbf{u}(t^{i})
-\overline{\mathbf{u}}^{\,i}\cdot\nabla\overline{\mathbf{u}}^{\,i},
\nabla q
\right)
+\nu(\nabla p_{s}(\overline{\mathbf{e}}^{\,i}),\nabla q)
+(\overline{e}_{\theta}^{\,i}\mathbf{e}_{2},\nabla q),
\quad \forall\,q\in H^{1}(\Omega).
\end{equation}
Forming $\widehat{e}_{p}^{\,i}
= (k+1)e_{p}^{\,i}-ke_{p}^{\,i-1}$
and taking $q = \widehat{e}_{p}^{\,i}$ leads to
\begin{equation}\label{eq:phat_err_bound}
\|\nabla\widehat{e}_{p}^{\,i}\|
\le
(k+1)
\|\mathbf{u}(t^{i})\cdot\nabla\mathbf{u}(t^{i})
-\overline{\mathbf{u}}^{\,i}\cdot\nabla\overline{\mathbf{u}}^{\,i}\|
+k
\|\mathbf{u}(t^{i-1})\cdot\nabla\mathbf{u}(t^{i-1})
-\overline{\mathbf{u}}^{\,i-1}\cdot\nabla\overline{\mathbf{u}}^{\,i-1}\|
+\nu\|\nabla p_{s}(\widetilde{\mathbf{e}}^{\,i})\|
+\|\widetilde{e}_{\theta}^{\,i}\|.
\end{equation}

The last term $\|\widetilde{e}_{\theta}^{\,i}\|$
is new compared to the pure Navier--Stokes case;
it arises from the temperature in the pressure.
For the nonlinear differences, the identity
$\mathbf{u}(t^{j})\cdot\nabla\mathbf{u}(t^{j})
-\overline{\mathbf{u}}^{\,j}\cdot\nabla\overline{\mathbf{u}}^{\,j}
=-\overline{\mathbf{e}}^{\,j}\cdot\nabla\overline{\mathbf{u}}^{\,j}
-\mathbf{u}(t^{j})\cdot\nabla\overline{\mathbf{e}}^{\,j}$,
together with $(a+b)^2\le 2a^2+2b^2$ and \eqref{convection_estimate3}, yields
\begin{equation}\label{eq:nonlin_diff_bound}
\|\mathbf{u}(t^{j})\cdot\nabla\mathbf{u}(t^{j})
-\overline{\mathbf{u}}^{\,j}\cdot\nabla\overline{\mathbf{u}}^{\,j}\|^{2}
\le
2C^{2}
\|\nabla\overline{\mathbf{e}}^{\,j}\|^{2}
\|\Delta\overline{\mathbf{u}}^{\,j}\|^{2}
+2C^2 \|\mathbf{u}(t^{j})\|_{2}^{2}
\|\nabla\overline{\mathbf{e}}^{\,j}\|^{2}.
\end{equation}
Using \eqref{eq:stokes_pressure_estimate} implies 
$\|\nabla p_{s}(\widetilde{\mathbf{e}}^{\,i})\|^{2}
\le (\frac{1}{2}+\varepsilon_{3})
\|\Delta\widetilde{\mathbf{e}}^{\,i}\|^{2}
+C_{S}(\varepsilon_{3})
\|\nabla\widetilde{\mathbf{e}}^{\,i}\|^{2}$ for any $\varepsilon_3>0$. 
Combining with \eqref{eq:phat_err_bound}, \eqref{eq:nonlin_diff_bound} 
and Young's inequality with
$\varepsilon_{4}, \varepsilon_{5} > 0$:
\begin{align}
|\mathcal{I}_{4}|
&\le\;
\frac{(k+1)^{2}C^{2}}{\varepsilon_{3}}
\|\nabla\overline{\mathbf{e}}^{\,i}\|^{2}
\!\left(
\|\Delta\overline{\mathbf{u}}^{\,i}\|^{2}
+ \|{\bf u}(t^i)\|_2^2 
\right)
+\frac{(k+1)^{2}C^{2}}{\varepsilon_{3}}
\|\nabla\overline{\mathbf{e}}^{\,i-1}\|^{2}
\!\left(
\|\Delta\overline{\mathbf{u}}^{\,i-1}\|^{2}
+ \|{\bf u}(t^{i-1})\|_2^2 
\right)
\nonumber\\
&+\left(
\varepsilon_{3}+\frac{\nu\varepsilon_{4}}{2}
+ \frac{\varepsilon_5}{2}
\right)
\|\Delta\widetilde{\mathbf{e}}^{\,i+1}\|^{2}
+\frac{\nu}{2\varepsilon_{4}}
\!\left(\frac{1}{2}+\varepsilon_{3}\right)
\|\Delta\widetilde{\mathbf{e}}^{\,i}\|^{2}
+\frac{\nu C_{S}(\varepsilon_{3})}{2\varepsilon_{4}}
\|\nabla\widetilde{\mathbf{e}}^{\,i}\|^{2}
+\frac{1}{2\varepsilon_{5}}
\|\widetilde{e}_{\theta}^{\,i}\|^{2}.
\label{eq:I4_err_final}
\end{align}

\medskip
\noindent\textit{Term $\mathcal{I}_{5}$:
buoyancy error.}
By the Cauchy-Schwarz and Poincare inequalities, we obtain
\begin{equation}\label{eq:I5_err_final}
|\mathcal{I}_{5}|
= |\left(
\widehat{e}_{\theta}^{\,i}\,\mathbf{e}_{2},
-\Delta\widetilde{\mathbf{e}}^{\,i+1}
\right)|
\le \frac{C^{2}}{2\varepsilon_{b}}
\|\nabla\widehat{e}_{\theta}^{\,i}\|^{2}
+\frac{\varepsilon_{b}}{2}
\|\Delta\widetilde{\mathbf{e}}^{\,i+1}\|^{2}.
\end{equation}
Substituting \eqref{eq:hat_etheta_bound}:
\begin{equation}\label{eq:I5_err_expanded}
|\mathcal{I}_{5}|
\le \frac{C^{2}C(k)}{2\varepsilon_{b}}
\!\left(
\|\nabla\overline{e}_{\theta}^{\,i}\|^{2}
+\|\nabla\overline{e}_{\theta}^{\,i-1}\|^{2}
\right)
+\frac{4C^{2}(k+1)^{2}C_{0}^{4}\tilde{C}_{1}^{\theta}\delta t^{4}}
{\varepsilon_{b}}
+\frac{\varepsilon_{b}}{2}
\|\Delta\widetilde{\mathbf{e}}^{\,i+1}\|^{2}.
\end{equation}

\medskip
\noindent\textit{Terms $\mathcal{I}_{6}$, $\mathcal{I}_{7}$, $\mathcal{I}_{8}$: truncation.}
A direct Taylor expansion about $t^{i+k}$ gives, for the truncation residuals defined in \eqref{eq:Pu}--\eqref{eq:Tu},
\begin{align}
P_{u}^{\,i}
&= (k+1)\!\int_{t^{i}}^{t^{i+k}}\!(t^{i}-s)\,\nabla p_{tt}(s)\,ds
  -k\!\int_{t^{i-1}}^{t^{i+k}}\!(t^{i-1}-s)\,\nabla p_{tt}(s)\,ds,
\label{eq:Pu_int}\\
Q_{u}^{\,i}
&= -k\!\int_{t^{i+1}}^{t^{i+k}}\!(t^{i+1}-s)\,\Delta\mathbf{u}_{tt}(s)\,ds
  -(k-1)\!\int_{t^{i}}^{t^{i+k}}\!(t^{i}-s)\,\Delta\mathbf{u}_{tt}(s)\,ds,
\label{eq:Qu_int}\\
R_{u}^{\,i}
&= \tfrac{2k+1}{2}\!\int_{t^{i+1}}^{t^{i+k}}\!(t^{i+1}-s)^{2}\,\mathbf{u}_{ttt}(s)\,ds
  -2k\!\int_{t^{i}}^{t^{i+k}}\!(t^{i}-s)^{2}\,\mathbf{u}_{ttt}(s)\,ds
\nonumber\\
&\quad
  +\tfrac{2k-1}{2}\!\int_{t^{i-1}}^{t^{i+k}}\!(t^{i-1}-s)^{2}\,\mathbf{u}_{ttt}(s)\,ds,
\label{eq:Ru_int}\\
S_{u}^{\,i}
&= \mathbf{u}(t^{i+k})\cdot\nabla\bigl(\mathbf{u}(t^{i+k})-\widehat{\mathbf{u}}(t^{i})\bigr)
  -\bigl(\widehat{\mathbf{u}}(t^{i})-\mathbf{u}(t^{i+k})\bigr)\cdot\nabla\widehat{\mathbf{u}}(t^{i}),
\label{eq:Su_int}\\
T_{u}^{\,i}
&= \biggl[(k+1)\!\int_{t^{i}}^{t^{i+k}}\!(t^{i}-s)\,\theta_{tt}(s)\,ds
  -k\!\int_{t^{i-1}}^{t^{i+k}}\!(t^{i-1}-s)\,\theta_{tt}(s)\,ds\biggr]\mathbf{e}_{2}.
\label{eq:Tu_int}
\end{align}

\smallskip
\noindent\textit{Residual norm bounds.}
Applying Cauchy--Schwarz inequality with
$\int_{t^{j}}^{t^{i+k}}(t^{j}-s)^{2}ds=\tfrac{((i+k-j)\delta t)^{3}}{3}$ and
$\int_{t^{j}}^{t^{i+k}}(t^{j}-s)^{4}ds=\tfrac{((i+k-j)\delta t)^{5}}{5}$, together with $(a+b)^{2}\le 2a^{2}+2b^{2}$ and the consolidation $k^{2}(k+1)^{2}(2k\!+\!1)\le 2(k+1)^{5}$, yields
\begin{align}
\|P_{u}^{\,i}\|^{2}
&\le \frac{4(k+1)^{5}\delta t^{3}}{3}\!\!\int_{t^{i-1}}^{t^{i+k}}\!\!\|\nabla p_{tt}\|^{2}\,ds,
\label{eq:Pu_Qu_norm_bound}
\quad
\|Q_{u}^{\,i}\|^{2}
\le \frac{4(k+1)^{5}\delta t^{3}}{3}\!\!\int_{t^{i-1}}^{t^{i+k}}\!\!\|\Delta\mathbf{u}_{tt}\|^{2}\,ds,
\\
\|R_{u}^{\,i}\|^{2}
&\le \frac{24(k+1)^{5}\delta t^{5}}{5}\!\!\int_{t^{i-1}}^{t^{i+k}}\!\!\|\mathbf{u}_{ttt}\|^{2}\,ds,
\label{eq:Ru_Tu_norm_bound}
\qquad
\|T_{u}^{\,i}\|^{2}
\le \frac{4(k+1)^{5}\delta t^{3}}{3}\!\!\int_{t^{i-1}}^{t^{i+k}}\!\!\|\theta_{tt}\|^{2}\,ds.
\end{align}
For $S_{u}^{\,i}$, the decomposition in \eqref{eq:Su_int} together with \eqref{convection_estimate3} and the bound $\|\mathbf{u}(t^{i+k})\|_{2}^{2},\|\widehat{\mathbf{u}}(t^{i})\|_{2}^{2}\le 2C_{\mathrm{reg}}^{u}$ gives
(the analysis of $\mathbf{u}(t^{i+k})-\widehat{\mathbf{u}}(t^{i})$ is the same as $P^i_u$)
\begin{equation}\label{eq:Su_norm_bound}
\|S_{u}^{\,i}\|^{2}
\le \frac{4C^{2}(k+1)^{5}\,C_{\mathrm{reg}}^{u}\,\delta t^{3}}{3}\!\!\int_{t^{i-1}}^{t^{i+k}}\!\!\|\nabla\mathbf{u}_{tt}\|^{2}\,ds.
\end{equation}

Applying Cauchy--Schwarz and Young's inequalities with weights $\varepsilon_{6},\varepsilon_{7},\varepsilon_{8}>0$ to the inner-product definitions of $\mathcal{I}_{6},\mathcal{I}_{7},\mathcal{I}_{8}$, together with \eqref{eq:Pu_Qu_norm_bound}--\eqref{eq:Su_norm_bound}, yields
\begin{align}
|\mathcal{I}_{6}|
&\le \varepsilon_{6}\|\Delta\widetilde{\mathbf{e}}^{\,i+1}\|^{2}
+\frac{C_{\mathcal{I},u} (k+1)^{5}\delta t^{3}}{\varepsilon_{6}}
\!\!\int_{t^{i-1}}^{t^{i+k}}\!\!\bigl(\|\nabla p_{tt}\|^{2}+\|\Delta\mathbf{u}_{tt}\|^{2}\bigr)ds,
\label{eq:I6_err_final}\\
|\mathcal{I}_{7}|
&\le \varepsilon_{7}\|\Delta\widetilde{\mathbf{e}}^{\,i+1}\|^{2}
+\frac{C_{\mathcal{I},u}(k+1)^{5}\delta t^{3}}{\varepsilon_{7}}
\!\!\int_{t^{i-1}}^{t^{i+k}}\!\!\|\mathbf{u}_{ttt}\|^{2}\,ds,
\label{eq:I7_err_final}\\
|\mathcal{I}_{8}|
&\le \varepsilon_{8}\|\Delta\widetilde{\mathbf{e}}^{\,i+1}\|^{2}
+\frac{C_{\mathcal{I},u}\,(k+1)^{5}\,\delta t^{3}}{\varepsilon_{8}}
\!\!\int_{t^{i-1}}^{t^{i+k}}\!\!\bigl(\|\nabla\mathbf{u}_{tt}\|^{2}+\|\theta_{tt}\|^{2}\bigr)ds.
\label{eq:I8_err_final}
\end{align}
where 
\begin{equation}\label{eq:CIu_def}
C_{\mathcal{I},u}
\;:=\;
C^{2}\,\max\!\bigl\{1,\,\nu^{2}\bigr\}\,
        \max\!\bigl\{1,\,C_{\mathrm{reg}}^{u}\bigr\}.
\end{equation}

Define
\begin{equation}\label{eq:Ctrunc_u_def}
\mathcal{T}^{u}
:= (k+1)^{5}\!\!\int_{0}^{T+1}\!\!\Bigl(
\|\nabla p_{tt}\|^{2}+\|\Delta\mathbf{u}_{tt}\|^{2}+\|\mathbf{u}_{ttt}\|^{2}
+ \|\nabla\mathbf{u}_{tt}\|^{2}+\|\theta_{tt}\|^{2}
\Bigr)ds.
\end{equation}
Substituting \eqref{eq:I1_err}--\eqref{eq:I8_err_final}
into \eqref{eq:mom_err_inner} and
multiplying by $2\delta t$:
\begin{align}
&A\!\left(
\|\nabla\overline{\mathbf{e}}^{\,i+1}\|^{2}
-\|\nabla\overline{\mathbf{e}}^{\,i}\|^{2}
\right)
+\|B\nabla\overline{\mathbf{e}}^{\,i+1}
-D\nabla\overline{\mathbf{e}}^{\,i}\|^{2}
-\|B\nabla\overline{\mathbf{e}}^{\,i}
-D\nabla\overline{\mathbf{e}}^{\,i-1}\|^{2}
\nonumber\\
&\quad
+E\|\nabla(\overline{\mathbf{e}}^{\,i+1}
-\overline{\mathbf{e}}^{\,i})\|^{2}
-F\|\nabla(\overline{\mathbf{e}}^{\,i}
-\overline{\mathbf{e}}^{\,i-1})\|^{2}
+G\|\nabla(\overline{\mathbf{e}}^{\,i+1}
-2\overline{\mathbf{e}}^{\,i}
+\overline{\mathbf{e}}^{\,i-1})\|^{2}
\nonumber\\
&\quad
+\nu\delta t\!\left(
\|\Delta\overline{\mathbf{e}}^{\,i+1}\|^{2}
-\|\Delta\overline{\mathbf{e}}^{\,i}\|^{2}
+\|\Delta(\overline{\mathbf{e}}^{\,i+1}
-\overline{\mathbf{e}}^{\,i})\|^{2}
\right)
+\frac{2\nu\delta t}{k}
\|\Delta\overline{\mathbf{e}}^{\,i+1}\|^{2}
\nonumber\\
&\quad
+2\nu\delta t
\left[
\frac{k-1}{k}
-\frac{\varepsilon_{4}}{2}
-\frac{1}{2\varepsilon_{4}}\!\left(\frac{1}{2}+2(k+1)^{2}\varepsilon_{3}\right)
-\frac{\varepsilon_{3}
+\varepsilon_{5}+\varepsilon_{6}
+\varepsilon_{7}+\varepsilon_{8}+\varepsilon_{b}}{2\nu}
\right]
\|\Delta\widetilde{\mathbf{e}}^{\,i+1}\|^{2}
\nonumber\\
\le\;&
C_{\mathrm{abs}}^{u}(\varepsilon_{3})\delta t
\|\nabla\overline{\mathbf{e}}^{\,i}\|^{2}
\!\left(
\|\Delta\overline{\mathbf{u}}^{\,i}\|^{2}
+C_{\mathrm{reg}}^{u}
\right)
+C_{\mathrm{abs}}^{u}(\varepsilon_{3})\delta t
\|\nabla\overline{\mathbf{e}}^{\,i-1}\|^{2}
\!\left(
\|\Delta\overline{\mathbf{u}}^{\,i-1}\|^{2}
+C_{\mathrm{reg}}^{u}
\right)
\nonumber\\
&+C_{\mathrm{abs}}^{u}(\varepsilon_{3})\delta t
\|\nabla\widehat{\mathbf{e}}_{u}^{\,i}\|^{2}
\!\left(
\|\Delta\widetilde{\mathbf{u}}^{\,i}\|^{2}
+C_{\mathrm{reg}}^{u}
\right)
+C_{\mathrm{abs}}^{u}(\varepsilon_{3})\delta t
\|\widetilde{e}_{\theta}^{\,i}\|^{2}
\nonumber\\
&+C_{\mathrm{abs}}^{u}(\varepsilon_{3})\delta t
\!\left(
\|\nabla\overline{e}_{\theta}^{\,i}\|^{2}
+\|\nabla\overline{e}_{\theta}^{\,i-1}\|^{2}
\right)
+C_{\mathrm{abs}}^{u}(\varepsilon_{3})C_{0}^{4}\delta t^{5}
\!\left(
\|\Delta\widetilde{\mathbf{u}}^{\,i}\|^{2}
+C_{\mathrm{reg}}^{u} + \widetilde{C}_{1}^{\theta}
\right)
\nonumber\\
&+C_{\mathrm{abs}}^{u}(\varepsilon_{3})\delta t^{4}
\int_{t^{i-1}}^{t^{i+k}}
\!\left(
\|\nabla p_{tt}\|^{2}
+\|\Delta\mathbf{u}_{tt}\|^{2}
+\|\mathbf{u}_{ttt}\|^{2}
+\|\nabla\mathbf{u}_{tt}\|^{2}
+\|\theta_{tt}\|^{2}
\right)\,ds,
\label{eq:mom_err_perstep}
\end{align}
where
$C_{\mathrm{abs}}^{u}(\varepsilon_{3})
= \max\!\left\{
\frac{2C^4C(k)}{\varepsilon_{3}},\,
\frac{2(k+1)^{2}C^{2}}{\varepsilon_{3}},\,
\frac{\nu C_{S}(\varepsilon_{3})}{\varepsilon_{4}},\,
\frac{C^{2}C(k)}{\varepsilon_{b}},\,
\frac{C_{\mathcal{I},u}(k+1)^{5}}{\varepsilon_{6}},\,
\frac{C_{\mathcal{I},u}(k+1)^{5}}{\varepsilon_{7}},\,
\frac{C_{\mathcal{I},u}(k+1)^{5}}{\varepsilon_{8}}
\right\}$.
The bracket term in \eqref{eq:mom_err_perstep}
is non-negative provided
\begin{equation}\label{eq:stab_err_mom}
\frac{k-1}{k}
\ge \frac{\varepsilon_{4}}{2}
+\frac{1}{2\varepsilon_{4}}
\!\left(\frac{1}{2}+2(k+1)^{2}\varepsilon_{3}\right)
+\frac{\varepsilon_{3}+\varepsilon_{5}
+\varepsilon_{6}+\varepsilon_{7}
+\varepsilon_{8}+\varepsilon_{b}}{2\nu}.
\end{equation}
By Lemma~\ref{lem:k_condition} (with
$\phi = 1/\nu$), choosing
$\varepsilon_{4}=1/\sqrt{2}$ and all other
small parameters sufficiently small relative
to $\nu$, condition \eqref{eq:stab_err_mom}
holds for $k\ge 4$.

We sum \eqref{eq:mom_err_perstep} from $i=1$ to $m$
(with $m\le n$), drop some non-negative terms on the left-hand side,  replace $\|\nabla\widehat{\mathbf{e}}_{u}^{\,i}\|^{2}$ by barred errors via \eqref{eq:hat_eu_bound}:
\begin{align}
&A\|\nabla\overline{\mathbf{e}}^{\,m+1}\|^{2}
+\nu\delta t\|\Delta\overline{\mathbf{e}}^{\,m+1}\|^{2}
+\frac{2\nu\delta t}{k}
\sum_{i=1}^{m}\|\Delta\overline{\mathbf{e}}^{\,i+1}\|^{2}
\nonumber\\
&\quad\le
C_{\mathrm{Gr}}^{u}\,\delta t\sum_{i=0}^{m}
g_{u}^{\,i}\,\|\nabla\overline{\mathbf{e}}^{\,i}\|^{2}
+
C_{\mathrm{Gr}}^{u}\,\delta t\sum_{i=0}^{m}
\!\left(\|\nabla\overline{e}_{\theta}^{\,i}\|^{2}
+\|\widetilde{e}_{\theta}^{\,i}\|^{2}\right)
\nonumber\\
&\qquad+\;
C_{\mathrm{Gr}}^{u}\,C_{0}^{4}\,\delta t^{5}
\sum_{i=0}^{m}
\!\left(\|\Delta\widetilde{\mathbf{u}}^{\,i}\|^{2}
+C_{\mathrm{reg}}^{u}\right)
+C_{\mathrm{Gr}}^{u}\,\delta t^{4}\,\mathcal{T}^{u}
+M_{1}^{u},
\label{eq:mom_err_summed}
\end{align}
where
\begin{equation}\label{eq:C262_def}
C_{\mathrm{Gr}}^{u}(\varepsilon_{3})
:= 4k(k+1)^{2}\,C_{\mathrm{abs}}^{u}(\varepsilon_{3}),
\qquad
g_{u}^{\,i}
:= \|\Delta\overline{\mathbf{u}}^{\,i}\|^{2}
+\|\Delta\widetilde{\mathbf{u}}^{\,i}\|^{2}
+C_{\mathrm{reg}}^{u},
\end{equation}
$M_{1}^{u}$ collects the start-up
contributions at $i=0,1$ and satisfies
$$M_{1}^{u}\le C_{M^u_1}\delta t^{4}$$ under the assumed
$O(\delta t^{2})$ initialization.
Note from \eqref{eq:C1_bnd_for_step2}
\[
\delta t\sum_{i=0}^{m} g_{u}^{\,i}
\;\le\;
\delta t\!\sum_{i=0}^{m}
\!\bigl(\|\Delta\overline{\mathbf{u}}^{\,i}\|^{2}
+\|\Delta\widetilde{\mathbf{u}}^{\,i}\|^{2}\bigr)
+
(m+1)\delta t\,C_{\mathrm{reg}}^{u}
\;\le\;
\tfrac{C_{1}^{u}}{\nu}+ C_{\mathrm{reg}}^{u}\,T.
\]
Define
\begin{equation}\label{eq:Lambda_u_def}
\Lambda^{u}
:= C_{\mathrm{Gr}}^{u}
\!\left(\tfrac{C_{1}^{u}}{\nu} + C_{\mathrm{reg}}^{u}\,T\right).
\end{equation}
For the $C_{0}^{4}\delta t^{5}$ block, the same a priori bound \eqref{eq:C1_bnd_for_step2} gives
$\delta t\sum_{i=0}^{m}
(\|\Delta\widetilde{\mathbf{u}}^{\,i}\|^{2}
+C_{\mathrm{reg}}^{u} + \widetilde{C}_{1}^{\theta})
\le C_{1}^{u}/\nu + (C_{\mathrm{reg}}^{u}+\widetilde{C}_{1}^{\theta})\,T$.
Dropping the (nonneg\-ative)
$\nu\delta t\|\Delta\overline{\mathbf{e}}^{\,m+1}\|^{2}$
from the left,
\eqref{eq:mom_err_summed} becomes $\forall m\le n$, 
\begin{equation}\label{eq:mom_err_gronwall_form}
A\|\nabla\overline{\mathbf{e}}^{\,m+1}\|^{2}
+\frac{2\nu\delta t}{k}
\sum_{i=1}^{m}\|\Delta\overline{\mathbf{e}}^{\,i+1}\|^{2}
\le \Lambda^{u}\,\delta t
\sum_{i=0}^{m}\|\nabla\overline{\mathbf{e}}^{\,i}\|^{2}
+C_{\mathrm{Gr}}^{u}\,\delta t\sum_{i=0}^{m}
\!\left(\|\nabla\overline{e}_{\theta}^{\,i}\|^{2}
+\|\widetilde{e}_{\theta}^{\,i}\|^{2}\right)
+C_{2,\mathrm{trunc}}^{u}\,\delta t^{4},
\end{equation}
with
\begin{equation}\label{eq:C2trunc_u_def}
C_{2,\mathrm{trunc}}^{u}
:= C_{\mathrm{Gr}}^{u}\,C_0^4 \!\left(\frac{C_{1}^{u}}{\nu}+
(C_{\mathrm{reg}}^{u}+\widetilde{C}_{1}^{\theta})\,T\right)
+ C_{\mathrm{Gr}}^{u}\,\mathcal{T}^{u} + C_{M^u_1}.
\end{equation}

\subsubsection{Step 2B. Temperature error estimate.}
The exact solution satisfies the temperature scheme \eqref{eq:GSAV_Bouss_compact_b_skew_corr} up to truncation residuals:
\begin{align}\label{theta_Consistencyequation}
D_{k}\theta(t^{i+1})
=&\gamma\Delta\!\bigl(k\theta(t^{i+1})-(k-1)\theta(t^{i})\bigr)
-\widehat{\mathbf{u}}(t^{i})\cdot\nabla\widehat{\theta}(t^{i})
-\alpha\,\widehat{v}(t^{i})
+g(t^{i+k}) \notag\\
&-\frac{R_{\theta}^{\,i}}{2\delta t}
-\gamma Q_{\theta}^{\,i}
-S_{\theta}^{\,i}
-T_{\theta}^{\,i},
\end{align}
with the truncation terms
\begin{align}
Q_{\theta}^{\,i}
&= -\Delta\theta(t^{i+k})
+\Delta\!\bigl(k\theta(t^{i+1})
-(k-1)\theta(t^{i})\bigr),
\label{eq:Qtheta}\\
R_{\theta}^{\,i}
&= -(2k+1)\theta(t^{i+1})
+4k\theta(t^{i})
-(2k-1)\theta(t^{i-1})
+2\delta t\,\theta_{t}(t^{i+k}),
\label{eq:Rtheta}\\
S_{\theta}^{\,i}
&= \mathbf{u}(t^{i+k})\cdot\nabla\theta(t^{i+k})
-\widehat{\mathbf{u}}(t^{i})\cdot\nabla\widehat{\theta}(t^{i}),
\label{eq:Stheta}\\
T_{\theta}^{\,i}
&= \alpha\bigl(v(t^{i+k})-\widehat{v}(t^{i})\bigr),
\label{eq:Ttheta}
\end{align}
where $\widehat{\mathbf{u}}(t^{i})=(k+1)\mathbf{u}(t^{i})-k\mathbf{u}(t^{i-1})$,
$\widehat{\theta}(t^{i})=(k+1)\theta(t^{i})-k\theta(t^{i-1})$,
and $\widehat{v}(t^{i})=(k+1)v(t^{i})-kv(t^{i-1})$ is the second component of $\widehat{\mathbf{u}}(t^{i})$.
Subtracting the consistency equation \eqref{theta_Consistencyequation} from the temperature scheme at step $i$ in \eqref{eq:GSAV_Bouss_compact_b_skew_corr}, we obtain
\begin{align}
&D_k\overline{e}_{\theta}^{\,i+1}
-\gamma \Delta\!\bigl(k\overline{e}_{\theta}^{\,i+1}
-(k-1)\overline{e}_{\theta}^{\,i}\bigr)
\nonumber\\
&
=  -\left(
\widehat{\mathbf{u}}^{\,i}\cdot\nabla\widehat{\theta}^{\,i}
-\widehat{\mathbf{u}}(t^{i})\cdot\nabla\widehat{\theta}(t^{i})
\right)
-\alpha\,\widehat{e}_{v}^{\,i}
+\frac{R_{\theta}^{\,i}}{2\delta t}
+\gamma Q_{\theta}^{\,i}
+S_{\theta}^{\,i}
+T_{\theta}^{\,i}.
\label{eq:temp_error_eq}
\end{align}

We take the $L^{2}(\Omega)$ inner product of
\eqref{eq:temp_error_eq} 
with $-\Delta\widetilde{e}_{\theta}^{\,i+1}$,
where $\widetilde{e}_{\theta}^{\,i+1}
= (k+1)\overline{e}_{\theta}^{\,i+1}
-k\overline{e}_{\theta}^{\,i}$, and get
\begin{equation}\label{eq:J_identity}
\mathcal{J}_{1}+\mathcal{J}_{2}
\;=\;
\mathcal{J}_{3}+\mathcal{J}_{4}+\mathcal{J}_{5}+\mathcal{J}_{6}+\mathcal{J}_{7},
\end{equation}
where
\begin{align*}
\mathcal{J}_{1} &:= \bigl(D_{k}\overline{e}_{\theta}^{\,i+1},\,-\Delta\widetilde{e}_{\theta}^{\,i+1}\bigr),
&\quad
\mathcal{J}_{2} &:= -\gamma\bigl(\Delta(k\overline{e}_{\theta}^{\,i+1}-(k-1)\overline{e}_{\theta}^{\,i}),\,-\Delta\widetilde{e}_{\theta}^{\,i+1}\bigr),\\
\mathcal{J}_{3} &:= -\bigl(\widehat{\mathbf{u}}^{\,i}\!\cdot\!\nabla\widehat{\theta}^{\,i}-\widehat{\mathbf{u}}(t^{i})\!\cdot\!\nabla\widehat{\theta}(t^{i}),\,-\Delta\widetilde{e}_{\theta}^{\,i+1}\bigr),
&\quad
\mathcal{J}_{4} &:= -\alpha\bigl(\widehat{e}_{v}^{\,i},\,-\Delta\widetilde{e}_{\theta}^{\,i+1}\bigr),\\
\mathcal{J}_{5} &:= \gamma\bigl(Q_{\theta}^{\,i},\,-\Delta\widetilde{e}_{\theta}^{\,i+1}\bigr),
&\quad
\mathcal{J}_{6} &:= \tfrac{1}{2\delta t}\bigl(R_{\theta}^{\,i},\,-\Delta\widetilde{e}_{\theta}^{\,i+1}\bigr),\\
\mathcal{J}_{7} &:= \bigl(S_{\theta}^{\,i}+T_{\theta}^{\,i},\,-\Delta\widetilde{e}_{\theta}^{\,i+1}\bigr).
\end{align*}

\medskip
\noindent\textit{Term $\mathcal{J}_{1}$:
time-difference.}
By Lemma~\ref{lem:alg_identity} with
$x=\nabla\overline{e}_{\theta}^{\,i+1}$,
$y=\nabla\overline{e}_{\theta}^{\,i}$,
$z=\nabla\overline{e}_{\theta}^{\,i-1}$:
\begin{align}
\mathcal{J}_{1}
=\frac{1}{2\delta t}\Bigl[
&A\!\left(
\|\nabla\overline{e}_{\theta}^{\,i+1}\|^{2}
-\|\nabla\overline{e}_{\theta}^{\,i}\|^{2}
\right)
+\|B\nabla\overline{e}_{\theta}^{\,i+1}
-D\nabla\overline{e}_{\theta}^{\,i}\|^{2}
-\|B\nabla\overline{e}_{\theta}^{\,i}
-D\nabla\overline{e}_{\theta}^{\,i-1}\|^{2}
\nonumber\\
&+E\|\nabla(\overline{e}_{\theta}^{\,i+1}
-\overline{e}_{\theta}^{\,i})\|^{2}
-F\|\nabla(\overline{e}_{\theta}^{\,i}
-\overline{e}_{\theta}^{\,i-1})\|^{2}
+G\|\nabla(\overline{e}_{\theta}^{\,i+1}
-2\overline{e}_{\theta}^{\,i}
+\overline{e}_{\theta}^{\,i-1})\|^{2}
\Bigr].
\label{eq:J1_err}
\end{align}

\medskip
\noindent\textit{Term $\mathcal{J}_{2}$:
diffusion.}
By the identity \eqref{eq:visc_identity} with
$a=\Delta\overline{e}_{\theta}^{\,i+1}$,
$b=\Delta\overline{e}_{\theta}^{\,i}$,
and $\gamma$ in place of $\nu$:
\begin{equation}\label{eq:J2_err}
\mathcal{J}_{2}
= \gamma\frac{k-1}{k}
\|\Delta\widetilde{e}_{\theta}^{\,i+1}\|^{2}
+\frac{\gamma}{k}
\|\Delta\overline{e}_{\theta}^{\,i+1}\|^{2}
+\frac{\gamma}{2}
\!\left(
\|\Delta\overline{e}_{\theta}^{\,i+1}\|^{2}
-\|\Delta\overline{e}_{\theta}^{\,i}\|^{2}
+\|\Delta(\overline{e}_{\theta}^{\,i+1}
-\overline{e}_{\theta}^{\,i})\|^{2}
\right).
\end{equation}

\medskip
\noindent\textit{Term $\mathcal{J}_{3}$:
convection.}
Note 
$\widehat{\mathbf{u}}^{\,i}\!\cdot\!\nabla\widehat{\theta}^{\,i}
-\widehat{\mathbf{u}}(t^{i})\!\cdot\!\nabla\widehat{\theta}(t^{i})
=
\widehat{\mathbf{e}}_{u}^{\,i}\!\cdot\!\nabla\widehat{\theta}^{\,i}
+\widehat{\mathbf{u}}(t^{i})\!\cdot\!\nabla\widehat{e}_{\theta}^{\,i}$.
Using \eqref{Teman_estimate} implies
\begin{align}
|\mathcal{J}_{3}|
\le\;&
C
\bigl( 
\|\nabla\widehat{\mathbf{e}}_{u}^{\,i}\|
\|\Delta\widehat{\theta}^{\,i}\|
+
\|\widehat{\mathbf{u}}(t^{i})\|_{2}
\|\nabla\widehat{e}_{\theta}^{\,i}\|
\bigr) 
\|\Delta\widetilde{e}_{\theta}^{\,i+1}\|.
\label{eq:J3_err_cs}
\end{align}
Applying Young's inequality with
$\varepsilon_{9} > 0$:
\begin{align}
|\mathcal{J}_{3}|
\le\;&
\frac{C^2}{\varepsilon_{9}}
\|\nabla\widehat{\mathbf{e}}_{u}^{\,i}\|^{2}
\|\Delta\widehat{\theta}^{\,i}\|^{2}
+\frac{C^{2}}{\varepsilon_{9}}
\|\widehat{\mathbf{u}}(t^{i})\|_{2}^{2}
\|\nabla\widehat{e}_{\theta}^{\,i}\|^{2}
+\frac{\varepsilon_{9}}{2}
\|\Delta\widetilde{e}_{\theta}^{\,i+1}\|^{2}.
\label{eq:J3_err_final}
\end{align}
Inserting  \eqref{eq:hat_eu_bound} and 
\eqref{eq:hat_etheta_bound} to \eqref{eq:J3_err_final}, one obtains
\begin{align}
|\mathcal{J}_3|
\le\;&
\frac{C^2 C(k)}{\varepsilon_9}
\bigl(\|\nabla\overline{\mathbf{e}}^{\,i}\|^2
+\|\nabla\overline{\mathbf{e}}^{\,i-1}\|^2\bigr)
\,\|\Delta\widehat{\theta}^{\,i}\|^2
+\frac{C^2 C(k)}{\varepsilon_9}
\,\|\widehat{\mathbf{u}}(t^i)\|_2^2\,
\bigl(\|\nabla\overline{e}_\theta^{\,i}\|^2
+\|\nabla\overline{e}_\theta^{\,i-1}\|^2\bigr)
\notag\\
&+\frac{8 C^2 (k+1)^2 C_0^4\,\widetilde{C}_1^{u}\,\delta t^4}{\varepsilon_9}
\,\|\Delta\widehat{\theta}^{\,i}\|^2
+\frac{8 C^2 (k+1)^2 C_0^4\,\widetilde{C}_{1}^{\theta}\,\delta t^4}{\varepsilon_9}
\,\|\widehat{\mathbf{u}}(t^i)\|_2^2
+\frac{\varepsilon_9}{2}\,\|\Delta\widetilde{e}_\theta^{\,i+1}\|^2.
\end{align}

Likewise, writing
$\widehat{\theta}^{\,i}-\widetilde{\theta}^{\,i}
=(k+1)(\eta^{\,i}-1)\overline{\theta}^{\,i}
-k(\eta^{\,i-1}-1)\overline{\theta}^{\,i-1}$
and using $|\eta^{\,i}-1|\le C_{0}^{2}\delta t^{2}$,
\begin{equation}\label{eq:lap_hat_theta_bound}
\|\Delta\widehat{\theta}^{\,i}\|^{2}
\le 2\|\Delta\widetilde{\theta}^{\,i}\|^{2}
+4(k+1)^{2}C_{0}^{4}\delta t^{4}
(\|\Delta\overline{\theta}^{\,i}\|^{2}
+\|\Delta\overline{\theta}^{\,i-1}\|^{2})
\le 4(k+1)^{2}(1+C_{0}^{4}\delta t^{4})
(\|\Delta\overline{\theta}^{\,i}\|^{2}
+\|\Delta\overline{\theta}^{\,i-1}\|^{2}).
\end{equation}
Combining \eqref{eq:lap_hat_theta_bound} with \eqref{C_1theta_def}, the sum $\delta t\sum\|\nabla\widehat{\mathbf{e}}_{u}^{\,i}\|^{2}
\|\Delta\widehat{\theta}^{\,i}\|^{2}$
then has a bounded Gr\"onwall coefficient.

\medskip
\noindent\textit{Term $\mathcal{J}_{4}$:
velocity coupling error.}
By Cauchy-Schwarz, Poincare
($\|\widehat{e}_v^{\,i}\|\le C\|\nabla\widehat{e}_v^{\,i}\|$),
and Young's inequality with $\varepsilon_{c}>0$:
\begin{equation}\label{eq:J4_err_final}
|\mathcal{J}_4|
\le \alpha\,\|\widehat{e}_v^{\,i}\|\,\|\Delta\widetilde{e}_\theta^{\,i+1}\|
\le \alpha C\,\|\nabla\widehat{e}_v^{\,i}\|\,\|\Delta\widetilde{e}_\theta^{\,i+1}\|
\le \frac{\alpha^{2}C^{2}}{2\varepsilon_{c}}
\|\nabla\widehat{e}_{v}^{\,i}\|^{2}
+\frac{\varepsilon_{c}}{2}
\|\Delta\widetilde{e}_{\theta}^{\,i+1}\|^{2}.
\end{equation}
Substituting \eqref{eq:hat_eu_bound} yields
\begin{equation}\label{eq:J4_err_expanded}
|\mathcal{J}_4|
\le \frac{\alpha^{2}C^{2}C(k)}{2\varepsilon_{c}}
\!\left(
\|\nabla\overline{\mathbf{e}}^{\,i}\|^{2}
+\|\nabla\overline{\mathbf{e}}^{\,i-1}\|^{2}
\right)
+\frac{4\alpha^{2}C^{2}(k+1)^{2}
C_{0}^{4}\tilde{C}_{1}^{u}\delta t^{4}}{\varepsilon_{c}}
+\frac{\varepsilon_{c}}{2}
\|\Delta\widetilde{e}_{\theta}^{\,i+1}\|^{2}.
\end{equation}

\medskip
\noindent\textit{Terms $\mathcal{J}_{5}$,
$\mathcal{J}_{6}$, $\mathcal{J}_{7}$.}
Using the same arguments as for $\mathcal{I}_{6}$--$\mathcal{I}_{8}$,
with \eqref{eq:Qtheta}--\eqref{eq:Ttheta}:
\begin{align}
|\mathcal{J}_{5}|
&\le \frac{\varepsilon_{10}}{2}
\|\Delta\widetilde{e}_{\theta}^{\,i+1}\|^{2}
+\frac{C_{\mathcal{J},\theta}\,(k+1)^{5}\,\delta t^{3}}{\varepsilon_{10}}
\int_{t^{i-1}}^{t^{i+k}}
\|\Delta\theta_{tt}(s)\|^{2}\,ds,
\label{eq:J5_err_final}\\
|\mathcal{J}_{6}|
&\le \frac{\varepsilon_{11}}{2}
\|\Delta\widetilde{e}_{\theta}^{\,i+1}\|^{2}
+\frac{C_{\mathcal{J},\theta}\,(k+1)^{5}\,\delta t^{3}}{\varepsilon_{11}}
\int_{t^{i-1}}^{t^{i+k}}
\|\theta_{ttt}(s)\|^{2}\,ds,
\label{eq:J6_err_final}\\
|\mathcal{J}_{7}|
&\le \frac{\varepsilon_{12}}{2}
\|\Delta\widetilde{e}_{\theta}^{\,i+1}\|^{2}
+\frac{C_{\mathcal{J},\theta}\,(k+1)^{5}\,\delta t^{3}}{\varepsilon_{12}}
\int_{t^{i-1}}^{t^{i+k}}
\!\left(
\|\theta_{tt}\|^{2}
+\|\nabla\mathbf{u}_{tt}\|^{2}
\right)\,ds,
\label{eq:J7_err_final}
\end{align}
where 
\begin{equation}\label{eq:CJtheta_def}
C_{\mathcal{J},\theta}
\;:=\;
C^2\,\max\!\bigl\{1,\,\gamma^{2},\,\alpha^{2}\bigr\}
\,\max\{1, C^\theta_{reg}\},
\qquad
C_{\mathrm{reg}}^{\theta}=\sup_{t}\|\theta(t)\|_{2}^{2}.
\end{equation}

Define the local (per-step) and global temperature truncation quantities
\begin{equation}\label{eq:Ctrunc_theta_def}
\begin{aligned}
\mathcal{T}_{i}^{\theta}
&:= (k+1)^{5}
\!\!\int_{t^{i-1}}^{t^{i+k}}
\!\!\left(
\|\Delta\theta_{tt}\|^{2}
+\|\theta_{ttt}\|^{2}
+
\|\nabla\theta_{tt}\|^{2}
+
\|\nabla\mathbf{u}_{tt}\|^{2}
\right)ds, \\[4pt]
\mathcal{T}^{\theta}
&:= (k+1)^{5}
\!\!\int_{0}^{T+1}
\!\!\left(
\|\Delta\theta_{tt}\|^{2}
+\|\theta_{ttt}\|^{2}
+
\|\nabla\theta_{tt}\|^{2}
+
\|\nabla\mathbf{u}_{tt}\|^{2}
\right)ds.
\end{aligned}
\end{equation}
Substituting all estimates into the tested
temperature error equation and multiplying
by $2\delta t$:
\begin{align}
&A\!\left(
\|\nabla\overline{e}_{\theta}^{\,i+1}\|^{2}
-\|\nabla\overline{e}_{\theta}^{\,i}\|^{2}
\right)
+\|B\nabla\overline{e}_{\theta}^{\,i+1}
-D\nabla\overline{e}_{\theta}^{\,i}\|^{2}
-\|B\nabla\overline{e}_{\theta}^{\,i}
-D\nabla\overline{e}_{\theta}^{\,i-1}\|^{2}
+E\|\nabla(\overline{e}_{\theta}^{\,i+1}
-\overline{e}_{\theta}^{\,i})\|^{2}
\nonumber\\
&\quad
-F\|\nabla(\overline{e}_{\theta}^{\,i}
-\overline{e}_{\theta}^{\,i-1})\|^{2}
+G\|\nabla(\overline{e}_{\theta}^{\,i+1}
-2\overline{e}_{\theta}^{\,i}
+\overline{e}_{\theta}^{\,i-1})\|^{2}
+\frac{2\gamma\delta t}{k}\|\Delta\overline{e}_{\theta}^{\,i+1}\|^{2}
\nonumber\\
&
+\gamma\delta t\!\left(
\|\Delta\overline{e}_{\theta}^{\,i+1}\|^{2}
-\|\Delta\overline{e}_{\theta}^{\,i}\|^{2}
+\|\Delta(\overline{e}_{\theta}^{\,i+1}-\overline{e}_{\theta}^{\,i})\|^{2}
\right)
+2\gamma\delta t\,
\kappa^{\theta} \|\Delta\widetilde{e}_{\theta}^{\,i+1}\|^{2}
\nonumber\\
\le\;&
C_{\mathrm{abs}}^{\theta}\delta t
\bigl(\|\nabla\overline{\mathbf{e}}^{\,i}\|^{2}
+\|\nabla\overline{\mathbf{e}}^{\,i-1}\|^{2}\bigr)
( \|\Delta\widehat{\theta}^{\,i}\|^{2} +1) 
+C_{\mathrm{abs}}^{\theta}\delta t
\|\widehat{\mathbf{u}}(t^{i})\|_{2}^{2}
\bigl(\|\nabla\overline{e}_{\theta}^{\,i}\|^{2}
+\|\nabla\overline{e}_{\theta}^{\,i-1}\|^{2}\bigr)
\nonumber\\
&
+C_{\mathrm{abs}}^{\theta}\,C_{0}^{4}\,\widetilde{C}_{1}^{u}\,\delta t^{5}\,\|\Delta\widehat{\theta}^{\,i}\|^{2}
+C_{\mathrm{abs}}^{\theta}\,C_{0}^{4}\,\widetilde{C}_{1}^{\theta}\,\delta t^{5}\,\|\widehat{\mathbf{u}}(t^{i})\|_{2}^{2}
+C_{\mathrm{abs}}^{\theta}\,C_{0}^{4}\,\delta t^{5}
+C_{\mathrm{abs}}^{\theta}\,\delta t^{4}\,\mathcal{T}_{i}^{\theta},
\label{eq:temp_err_perstep}
\end{align}
where
$C_{\mathrm{abs}}^{\theta}
= \max\!\left\{
\frac{2C^{2}C(k)}{\varepsilon_{9}},\,
\frac{\alpha^{2}C^{2}C(k)}{\varepsilon_{c}},\,
\frac{16C^{2}(k+1)^{2}}{\varepsilon_{9}},\,
\frac{8\alpha^{2}C^{2}(k+1)^{2}\widetilde{C}_{1}^{u}}{\varepsilon_{c}},\,
\frac{C_{\mathcal{J},\theta}}{\varepsilon_{10}},\,
\frac{C_{\mathcal{J},\theta}}{\varepsilon_{11}},\,
\frac{C_{\mathcal{J},\theta}}{\varepsilon_{12}}
\right\}.$ and 
\begin{equation}
\kappa^{\theta}=\tfrac{k-1}{k}- \tfrac{\varepsilon_{9}+\varepsilon_{10}+\varepsilon_{11}+\varepsilon_{12}+\varepsilon_{c}}{2\gamma}.
\label{kappa_theta_def}
\end{equation}
 Note $\kappa^{\theta}>0$  holds for $k\ge 2$ with the $\varepsilon$-parameters chosen small enough relative to $\gamma$.

We sum \eqref{eq:temp_err_perstep} from $i=1$ to $m$
(with $m\le n$) and drop some non-negative terms on the left-hand side.
On the RHS, reindex the $i\!-\!1$ pieces in the convection/coupling sums
($j=i-1$), bound
$\|\widehat{\mathbf{u}}(t^{i})\|_{2}^{2}\le 2(k+1)^{2}C_{\mathrm{reg}}^{u}$
 (from
$\widehat{\mathbf{u}}(t^{i})=(k+1)\mathbf{u}(t^{i})-k\mathbf{u}(t^{i-1})$
and $\|\mathbf{u}(t)\|_{2}^{2}\le C_{\mathrm{reg}}^{u}$), and note that
each $s\in[0,T+1]$ lies in at most $(k+2)$ of the intervals
$[t^{i-1},t^{i+k}]$ used in \eqref{eq:Ctrunc_theta_def}, so
$\sum_{i=1}^{m}\mathcal{T}_{i}^{\theta}\le (k+2)\,\mathcal{T}^{\theta}$
(the factor $k+2$ is absorbed into $C_{\mathrm{abs}}^{\theta}$). We obtain
\begin{align}\label{eq:temp_err_ready}
& A\|\nabla\overline{e}_{\theta}^{\,m+1}\|^{2}
+\frac{2\gamma\delta t}{k}
\sum_{i=1}^{m}\|\Delta\overline{e}_{\theta}^{\,i+1}\|^{2}
\nonumber\\
\le & C_{\mathrm{abs}}^{\theta}\,\delta t\sum_{i=0}^{m}H_{\theta}^{\,i}\,\|\nabla\overline{\mathbf{e}}^{\,i}\|^{2}
+\,4(k+1)^{2}C_{\mathrm{reg}}^{u}\,C_{\mathrm{abs}}^{\theta}\,\delta t
\sum_{i=0}^{m}\|\nabla\overline{e}_{\theta}^{\,i}\|^{2}
+\,C_{2,\mathrm{trunc}}^{\theta}\,\delta t^{4}.
\end{align}
where the variable Gr\"onwall coefficient and the truncation constant are
\begin{equation}\label{eq:Lambda_theta_def}
\begin{aligned}
H_{\theta}^{\,i}
&:= \|\Delta\widehat{\theta}^{\,i}\|^{2}
+\|\Delta\widehat{\theta}^{\,i+1}\|^{2}+2, \\[4pt]
C_{2,\mathrm{trunc}}^{\theta}
&:= C_{\mathrm{abs}}^{\theta}C_{0}^{4}\!\left(
\frac{16(k+1)^{2}C_{1}^{\theta}\widetilde{C}_{1}^{u}}{\gamma}
+2(k+1)^{2}C_{\mathrm{reg}}^{u}T\,\widetilde{C}_{1}^{\theta}
+T
\right)
+ C_{\mathrm{abs}}^{\theta}\,\mathcal{T}^{\theta}
+ C_{M^{\theta}_{1}},
\end{aligned}
\end{equation}
and $M_{1}^{\theta}\le C_{M^{\theta}_{1}}\delta t^{4}$ collects
the start-up contributions at $i=0,1$ under
the $O(\delta t^{2})$ initialization.

\smallskip
\noindent\textit{Summed bound on $H_{\theta}^{\,i}$.}
Under the standing condition $\delta t\le 1/(1+2C_{0}^{2})$
(so that $C_{0}^{4}\delta t^{4}\le 1$), the bound
\eqref{eq:lap_hat_theta_bound} gives
$\|\Delta\widehat{\theta}^{\,i}\|^{2}
\le 8(k+1)^{2}(\|\Delta\overline{\theta}^{\,i}\|^{2}
+\|\Delta\overline{\theta}^{\,i-1}\|^{2})$.
Summing and using \eqref{C_1theta_def}
($\gamma\delta t\sum_{i=0}^{n}\|\Delta\overline{\theta}^{\,i}\|^{2}\le C_{1}^{\theta}$):
\begin{equation}\label{eq:lap_hat_theta_sumbound}
\delta t\sum_{i=0}^{m}\|\Delta\widehat{\theta}^{\,i}\|^{2}
\le 8(k+1)^{2}\,\delta t\sum_{i=0}^{m}
\!\left(\|\Delta\overline{\theta}^{\,i}\|^{2}
+\|\Delta\overline{\theta}^{\,i-1}\|^{2}\right)
\le 16(k+1)^{2}\,\delta t\sum_{i=0}^{m}\|\Delta\overline{\theta}^{\,i}\|^{2}
\le \frac{16(k+1)^{2}C_{1}^{\theta}}{\gamma}.
\end{equation}
The same bound holds for
$\delta t\sum_{i=0}^{m}\|\Delta\widehat{\theta}^{\,i+1}\|^{2}$
by a one-step shift, and $(m+1)\delta t\le T$, so
\begin{equation}\label{eq:H_theta_sumbound}
\delta t\sum_{i=0}^{m}H_{\theta}^{\,i}
\le \delta t\sum_{i=0}^{m}\|\Delta\widehat{\theta}^{\,i}\|^{2}
+\delta t\sum_{i=0}^{m}\|\Delta\widehat{\theta}^{\,i+1}\|^{2}
+2(m+1)\delta t
\le \frac{32(k+1)^{2}C_{1}^{\theta}}{\gamma}+2T
=: M_{\theta}.
\end{equation}

\subsubsection{Step 2C. Combined Gr\"onwall argument and conclusion.}
Adding \eqref{eq:mom_err_gronwall_form} 
and \eqref{eq:temp_err_ready} then divided by $A=\tfrac{1}{2k}$ yields
\begin{align}
&\;\|\nabla\overline{\mathbf{e}}^{\,m+1}\|^{2}
+\|\nabla\overline{e}_{\theta}^{\,m+1}\|^{2}
+4\nu\delta t
\sum_{i=1}^{m}\|\Delta\overline{\mathbf{e}}^{\,i+1}\|^{2}
+4\gamma\delta t
\sum_{i=1}^{m}\|\Delta\overline{e}_{\theta}^{\,i+1}\|^{2}
\notag\\
\le&\;
2k\Lambda^{u}\,\delta t
\sum_{i=0}^{m}\|\nabla\overline{\mathbf{e}}^{\,i}\|^{2}
+2kC_{\mathrm{abs}}^{\theta}\,\delta t\sum_{i=0}^{m}H_{\theta}^{\,i}\,\|\nabla\overline{\mathbf{e}}^{\,i}\|^{2}  
+2kC_{\mathrm{Gr}}^{u}\,\delta t
\sum_{i=0}^{m}\|\widetilde{e}_{\theta}^{\,i}\|^{2}
\notag\\
&
+2k\bigl(C_{\mathrm{Gr}}^{u}+4(k+1)^{2}C_{\mathrm{reg}}^{u}C_{\mathrm{abs}}^{\theta}\bigr)\,\delta t
\sum_{i=0}^{m}\|\nabla\overline{e}_{\theta}^{\,i}\|^{2}
+2k\bigl(C_{2,\mathrm{trunc}}^{u}+C_{2,\mathrm{trunc}}^{\theta}\bigr)\delta t^{4}.
\label{eq:combined_err_pre}
\end{align}
Define
\begin{equation}\label{eq:Phi_Psi_def}
\Phi_{i}
= \|\nabla\overline{\mathbf{e}}^{\,i}\|^{2}
+\|\nabla\overline{e}_{\theta}^{\,i}\|^{2},
\quad
\Psi_{i}
= 4\nu\|\Delta\overline{\mathbf{e}}^{\,i}\|^{2}
+4\gamma\|\Delta\overline{e}_{\theta}^{\,i}\|^{2}.
\end{equation}
Note
$$\|\widetilde{e}_{\theta}^{\,i}\|^{2}
\le 2(k+1)^{2}(\|\overline{e}_{\theta}^{\,i}\|^{2}+\|\overline{e}_{\theta}^{\,i-1}\|^{2})
\le 2C^{2}(k+1)^{2}(\|\nabla\overline{e}_{\theta}^{\,i}\|^{2}+\|\nabla\overline{e}_{\theta}^{\,i-1}\|^{2}),$$
where the Poincare inequality is used in the last step.
Then \eqref{eq:combined_err_pre} takes the  Gr\"onwall form
\begin{equation}\label{eq:combined_gronwall_form}
\Phi_{m+1}+\delta t\sum_{i=1}^{m+1}\Psi_{i}
\le \delta t\sum_{i=0}^{m}G_{i}^{*}\,\Phi_{i}+C_{2}^{*}\,\delta t^{4},
\end{equation} 
where 
\begin{eqnarray}\label{eq:Lambda_combined}
G_{i}^{*} &=& 
2kC_{\mathrm{abs}}^{\theta}\,H_{\theta}^{\,i}
+\Lambda^{*}_{\mathrm{fix}},
\qquad
C_{2}^{*} = 2k\bigl(C_{2,\mathrm{trunc}}^{u}+C_{2,\mathrm{trunc}}^{\theta}\bigr),
\\
\Lambda^{*}_{\mathrm{fix}}
&=& 
2k \Lambda^{u} + 8k (k+1)^2 C_{\mathrm{Gr}}^{u} C^2 
+ 2k(C_{\mathrm{Gr}}^{u}+4(k+1)^2 C_{\mathrm{reg}}^{u} C_{\mathrm{abs}}^{\theta}).
\end{eqnarray}
By \eqref{eq:H_theta_sumbound}, the  Gr\"onwall coefficient satisfies the bound
\begin{equation}\label{eq:M_star_bound}
\delta t\sum_{i=0}^{m}G_{i}^{*}
\le \Lambda_{*}^{\mathrm{fix}}T+2kC_{\mathrm{abs}}^{\theta}\,M_{\theta}
=:M_{*}.
\end{equation}
Then Gr\"onwall Lemma~\ref{lem:discrete_gronwall_2} gives
$\Phi_{m+1} +\delta t\sum_{i=1}^{m+1}\Psi_{i}
\le \exp( M_{*})
\, C_{2}^{*} \delta t^{4}$,
$\forall 1\le m\le n$.
Setting
\begin{equation}\label{eq:C2_final_def}
C_{2}
= \exp( M_{*}) \, C_{2}^{*},
\end{equation}
we obtain
\begin{equation}\label{eq:combined_err}
\|\nabla\overline{\mathbf{e}}^{\,m+1}\|^{2}
+\|\nabla\overline{e}_{\theta}^{\,m+1}\|^{2}
+ 4\delta t\sum_{i=1}^{m+1}
( \nu\|\Delta\overline{\mathbf{e}}^{\,i}\|^{2}
+\gamma\|\Delta\overline{e}_{\theta}^{\,i}\|^{2})
\le C_2 \delta t^{4},
\quad 1\le m\le n.
\end{equation}
From
$\|\nabla\overline{\mathbf{u}}^{\,m+1}\|^{2}
\le 2\|\nabla\overline{\mathbf{e}}^{\,m+1}\|^{2}
+2\|\nabla\mathbf{u}(t^{m+1})\|^{2}$,
\eqref{eq:combined_err}, 
and \eqref{Creg_def_C1}, one deduces
\begin{equation}\label{eq:C3_step2_def}
\|\nabla\overline{\mathbf{u}}^{\,m+1}\|^{2}
\le 2C_{2}\delta t^{4}
+2C_{1}^{u}
\le 2C_{2}+2C_{1}^{u}
\le C_{3},
\quad
1\le m\le n,
\end{equation}
where $C_3$ is defined as 
\begin{equation}
C_3 = \max\{ 2C_2 + 2C^u_1, 2C_2 + 2C^\theta_1\}.
\label{C3_def}
\end{equation}
Here, the condition $\delta t\le\frac{1}{1+2C_{0}^{2}}<1$ is applied in \eqref{eq:C3_step2_def}. 
Similarly, one derives
\begin{equation}
\|\nabla\overline{\theta}^{\,m+1}\|^{2}\le C_{3},
\quad 1\le m\le n.
\label{eq:gradthetaC3}
\end{equation}

\subsection{Proof Step 3: Estimate for $|1 - \xi^{\,n+1}|$}

The goal of this step is to close the induction by proving $|1 - \xi^{\,n+1}|\le C_0 \delta t$ and derive the error bounds for the unbarred velocity and temperature, and the pressure.

\subsubsection{Step 3A. The SAV error equation.}
Define the auxiliary error
\begin{equation}\label{eq:s_def_step3}
s^{\,i} = r^{\,i} - r(t^{i}).
\end{equation}
Note there exists a general relation
\begin{equation}\label{eq:r_exact_step}
r(t^{i+1}) - r(t^{i})
= \delta t\,r_{t}(t^{i+1}) + T^{i},
\end{equation}
where the remainder is
\begin{equation}\label{eq:Tn_def}
T^{i}
= r(t^{i+1}) - r(t^{i}) - \delta t\,r_{t}(t^{i+1})
= -\int_{t^{i}}^{t^{i+1}}
(s - t^{i})\,r_{tt}(s)\,ds.
\end{equation}
Recall $r(t)=E({\bf u}(t),\theta(t))+\bar{C}$ from \eqref{r_def} and  Lemma~\ref{lem:energy_law} implies
\begin{equation}\label{eq:exact_energy_law}
r_{t}(t)
= -\nu\|\nabla\mathbf{u}(t)\|^{2}
-\frac{\gamma}{\alpha}\|\nabla\theta(t)\|^{2}
+(\mathbf{f}(t),\mathbf{u}(t))
+\frac{1}{\alpha}(g(t),\theta(t)).
\end{equation}
Differentiating $r(t)=E({\bf u}(t),\theta(t))+\bar{C}$ twice in time gives
\begin{equation}\label{eq:rtt_def}
r_{tt}(t)
= \|\mathbf{u}_{t}(t)\|^{2}+(\mathbf{u}(t),\mathbf{u}_{tt}(t))
+\frac{1}{\alpha}\bigl(\|\theta_{t}(t)\|^{2}+(\theta(t),\theta_{tt}(t))\bigr).
\end{equation}
Substituting \eqref{eq:rtt_def} into \eqref{eq:Tn_def}
and applying Cauchy--Schwarz to the inner products,
\begin{equation}\label{eq:Tn_bound_step3}
|T^{i}|
\le \delta t
\int_{t^{i}}^{t^{i+1}}
\!\left(
\|\mathbf{u}_{t}(s)\|^{2}
+\|\mathbf{u}(s)\|\|\mathbf{u}_{tt}(s)\|
+\frac{1}{\alpha}\|\theta_{t}(s)\|^{2}
+\frac{1}{\alpha}\|\theta(s)\|\|\theta_{tt}(s)\|
\right)\,ds.
\end{equation}
The discrete GSAV update
\eqref{eq:GSAV_Bouss_compact_d_corr} can be rewritten as
\begin{equation}\label{eq:r_update_step3}
r^{\,i+1}-r^{\,i}
= \delta t\, R^{i+1}_{\bar{E}}
\!\left(
-\nu\|\nabla\overline{\mathbf{u}}^{\,i+1}\|^{2}
-\frac{\gamma}{\alpha}
\|\nabla\overline{\theta}^{\,i+1}\|^{2}
+(\mathbf{f}^{\,i+1},\overline{\mathbf{u}}^{\,i+1})
+\frac{1}{\alpha}
(g^{\,i+1},\overline{\theta}^{\,i+1})
\right).
\end{equation}
where we denote
\begin{equation}
R^{i+1}_{\overline{E}}=
\frac{r^{\,i+1}}{E(\overline{\bf u}^{i+1}, \overline\theta^{i+1})+\overline{C}},
\qquad
R^{i+1}_{E}=
\frac{r^{\,i+1}}{E({\bf u}(t^{i+1}), \theta(t^{i+1}))+\overline{C}}.
\end{equation}
Subtracting \eqref{eq:r_exact_step} from
\eqref{eq:r_update_step3} and rearranging,
the error $s^{\,i+1}$ satisfies
\begin{align}
&s^{\,i+1} - s^{\,i}= - T^i
+ \nu\delta t\!\left(
\|\nabla\mathbf{u}(t^{i+1})\|^{2}
- R^{i+1}_{\overline{E}}
\|\nabla\overline{\mathbf{u}}^{\,i+1}\|^{2}
\right)
+\frac{\gamma\delta t}{\alpha}\!\left(
\|\nabla\theta(t^{i+1})\|^{2}
- R^{i+1}_{\overline{E}}
\|\nabla\overline{\theta}^{\,i+1}\|^{2}
\right)
\nonumber\\
&
+\delta t\!\left(
R^{i+1}_{\overline{E}}(\mathbf{f}^{\,i+1},\overline{\mathbf{u}}^{\,i+1})
- (\mathbf{f}^{i+1},\mathbf{u}(t^{i+1}))
\right)
+\frac{\delta t}{\alpha}\!\left(
R^{i+1}_{\overline{E}}
(g^{\,i+1},\overline{\theta}^{\,i+1})
- (g^{i+1},\theta(t^{i+1}))
\right).
\label{eq:s_error_step3}
\end{align}

\subsubsection{Step 3B. Preliminary estimates for $s^{n+1}$}
We bound each of the four main terms in
\eqref{eq:s_error_step3}.
Throughout, we use the following facts
from the weak stability Theorem~\ref{thm:weak_stability}
and the induction hypothesis: for $j=1,\ldots,N$,
\begin{equation*}
\|\mathbf{u}^{\,j}\|,\, \|\theta^{\,j}\|,\, r^{\,j} \le M_{T},
\qquad
\overline{E}^{\,j+1} + \overline{C} \ge \overline{C} \ge 1,
\end{equation*}
together with the uniform bounds on the exact solution
\begin{equation}\label{eq:exact_bounds_step3}
\|\mathbf{u}(t)\|^{2},\,
\|\nabla\mathbf{u}(t)\|^{2},\,
\|\theta(t)\|^{2},\,
\|\nabla\theta(t)\|^{2}
\le C_{301}
:= \max\{C_{\mathrm{reg}}^{u},\, C_{\mathrm{reg}}^{\theta}\},
\qquad t \in [0,T].
\end{equation}

For the velocity dissipation, we write
\begin{align}
&\nu\!\left(
\|\nabla\mathbf{u}(t^{i+1})\|^{2}
- R_{\overline{E}}^{\,i+1}
\|\nabla\overline{\mathbf{u}}^{\,i+1}\|^{2}
\right)
\notag\\
= &  \nu\|\nabla\mathbf{u}(t^{i+1})\|^{2}
\bigl(1 - R_{\overline{E}}^{\,i+1}\bigr)
+\nu R_{\overline{E}}^{\,i+1}
\!\left(
\|\nabla\mathbf{u}(t^{i+1})\|^{2}
-\|\nabla\overline{\mathbf{u}}^{\,i+1}\|^{2}
\right)
=: \nu W_{1}^{u,i} + \nu W_{2}^{u,i}.
\label{eq:vel_diss_split}
\end{align}
An identical splitting applies to the
temperature dissipation:
\begin{equation}\label{eq:temp_diss_split}
\frac{\gamma}{\alpha}\!\left(
\|\nabla\theta(t^{i+1})\|^{2}
- R_{\overline{E}}^{\,i+1}
\|\nabla\overline{\theta}^{\,i+1}\|^{2}
\right)
=: \frac{\gamma}{\alpha}W_{1}^{\theta,i}
+\frac{\gamma}{\alpha}W_{2}^{\theta,i}.
\end{equation}
Here $W_{1}^{\theta,i}:=\|\nabla\theta(t^{i+1})\|^{2}\bigl(1-R_{\overline{E}}^{\,i+1}\bigr)$
and $W_{2}^{\theta,i}:=R_{\overline{E}}^{\,i+1}\bigl(\|\nabla\theta(t^{i+1})\|^{2}-\|\nabla\overline{\theta}^{\,i+1}\|^{2}\bigr)$.

We decompose
\begin{equation}\label{eq:1_minus_xi_decomp}
1 - R_{\overline{E}}^{\,i+1}
= \bigl(1 - R_{E}^{\,i+1}\bigr)
+ \bigl(R_{E}^{\,i+1} - R_{\overline{E}}^{\,i+1}\bigr).
\end{equation}
Noting $1 - R_{E}^{\,i+1} = -s^{\,i+1}/(E(\mathbf{u}(t^{i+1}),\theta(t^{i+1}))+\overline{C})$
since $r(t^{i+1})=E(\mathbf{u}(t^{i+1}),\theta(t^{i+1}))+\overline{C}$,
and using $E(\mathbf{u}(t^{i+1}),\theta(t^{i+1}))+\overline{C}\ge 1$
and $r^{i+1}\le M_{T}$:
\begin{equation}\label{eq:1_minus_xi_bound}
\bigl|1 - R_{\overline{E}}^{i+1}\bigr|
\le |s^{i+1}|
+ M_{T}\bigl|E(\mathbf{u}(t^{i+1}),\theta(t^{i+1})) - \overline{E}^{i+1}\bigr|.
\end{equation}
For the energy difference, using
$|E(\mathbf{u},\theta)-E(\overline{\mathbf{u}},\overline{\theta})|
\le \frac{1}{2}(\|\mathbf{u}\|+\|\overline{\mathbf{u}}\|)
\|\mathbf{u}-\overline{\mathbf{u}}\|
+\frac{1}{2\alpha}(\|\theta\|+\|\overline{\theta}\|)
\|\theta-\overline{\theta}\|$,
the Poincare inequality, and
$
\|\mathbf{u}(t^{i+1})\|+\|\overline{\mathbf{u}}^{i+1}\|,
\|\theta(t^{i+1})\|+\|\overline{\theta}^{i+1}\|
 \le \sqrt{C_{301}}+M_{T}$, we get
\begin{equation}\label{eq:energy_diff_bound}
\left|
E(\mathbf{u}(t^{i+1}),\theta(t^{i+1}))
-\overline{E}^{i+1}
\right|
\le C_{302}
\!\left(
\|\overline{\mathbf{e}}^{i+1}\|
+\|\overline{e}_{\theta}^{i+1}\|
\right)
\le C_{302}C
\!\left(
\|\nabla\overline{\mathbf{e}}^{i+1}\|
+\|\nabla\overline{e}_{\theta}^{i+1}\|
\right),
\end{equation}
where
\begin{equation}\label{eq:C305_def}
C_{302}
= \max\left\{ \frac{1}{2},\, \frac{1}{2\alpha} \right\}
\; (\sqrt{C_{301}}+M_{T}).
\end{equation}
Inserting \eqref{eq:energy_diff_bound} to \eqref{eq:1_minus_xi_bound} gives
\begin{equation}\label{eq:1-Rfinal}
\bigl|1 - R_{\overline{E}}^{i+1}\bigr|
\le 
\max\{1, M_T C C_{302} \}
\left( |s^{i+1}| +\|\nabla\overline{\mathbf{e}}^{i+1}\|
+\|\nabla\overline{e}_{\theta}^{i+1}\|
\right).
\end{equation}
Therefore, defining
\begin{equation}\label{eq:C301_def}
C_{303}
= (\nu + \frac{\gamma}{\alpha} )
C_{301} \max\{1, M_T C C_{302} \},
\end{equation}
we can easily obtain
\begin{align}
|\nu W_{1}^{u,i}|
+\left|\frac{\gamma}{\alpha}W_{1}^{\theta,i}\right|
&\le C_{303}
\!\left(
|s^{i+1}|
+\|\nabla\overline{\mathbf{e}}^{i+1}\|
+\|\nabla\overline{e}_{\theta}^{i+1}\|
\right).
\label{eq:W1_bound_step3}
\end{align}
The same method together with $\|\nabla\overline{\mathbf{u}}^{i+1}\|\le\sqrt{C_{3}}$
from \eqref{eq:C3_step2_def} implies 
\begin{eqnarray}\label{eq:W2_vel_bound}
\nu|W_{2}^{u,i}|
&\le& \nu M_{T}
\!\left(
\|\nabla\overline{\mathbf{u}}^{i+1}\|
+\|\nabla\mathbf{u}(t^{i+1})\|
\right)
\|\nabla\overline{\mathbf{e}}^{i+1}\|
\le C_{304}^{u}
\|\nabla\overline{\mathbf{e}}^{i+1}\|,
\\
\label{eq:W2_temp_bound}
\frac{\gamma}{\alpha}|W_{2}^{\theta,i}|
&\le& C_{304}^{\theta}
\|\nabla\overline{e}_{\theta}^{i+1}\|,
\end{eqnarray}
where
\begin{equation}\label{eq:C302_def}
C_{304}^{u}
= \nu M_{T}
(\sqrt{C_{3}} + \sqrt{C_{301}}),
\qquad
C_{304}^{\theta}
= \frac{\gamma M_{T}}{\alpha}
(\sqrt{C_{3}} + \sqrt{C_{301}}).
\end{equation}

We split the velocity forcing difference as
\begin{align}
&\left|
R^{i+1}_{\overline{E}}
(\mathbf{f}^{i+1},\overline{\mathbf{u}}^{i+1})
-(\mathbf{f}^{i+1},\mathbf{u}(t^{i+1}))
\right|
\nonumber\\
\le & 
|(\mathbf{f}^{i+1},\mathbf{u}(t^{i+1}))|
\left|
1 - R^{i+1}_{\overline{E}}
\right|
+R^{i+1}_{\overline{E}}
\left|
(\mathbf{f}^{i+1},\overline{\mathbf{u}}^{i+1})
-(\mathbf{f}^{i+1},\mathbf{u}(t^{i+1}))
\right|
=: W_{3}^{u,i} + W_{4}^{u,i}.
\label{eq:forcing_split}
\end{align}
For $W_{3}^{u,i}$, using
$|(\mathbf{f}(t),\mathbf{u}(t))|
\le C_{f}\sqrt{C_{301}}$ and
\eqref{eq:1-Rfinal}:
\begin{equation}\label{eq:W3u_bound}
W_{3}^{u,i}
\le C_{305}^{u}
\!\left(
|s^{i+1}|
+\|\nabla\overline{\mathbf{e}}^{i+1}\|
+\|\nabla\overline{e}_{\theta}^{i+1}\|
\right),
\qquad
C_{305}^{u}
= C_{f}\sqrt{C_{301}} \max\{1, M_{T}C_{302}C \}.
\end{equation}
For $W_{4}^{u,i}$, using $r^{i+1}/(\overline{E}^{i+1}+\overline{C})\le M_{T}$ and the Poincare inequality leads to
\begin{equation}\label{eq:W4u_bound}
W_{4}^{u,i}
\le M_{T}C_{f}
\|\overline{\mathbf{e}}^{i+1}\|
\le M_{T}C_{f}C
\|\nabla\overline{\mathbf{e}}^{i+1}\|
=: C_{306}^{u}\|\nabla\overline{\mathbf{e}}^{i+1}\|.
\end{equation}
An identical argument applies to the temperature
forcing to generate
\begin{eqnarray}
\bigl| R^{i+1}_{\overline{E}} (g^{i+1},\overline{\theta}^{i+1}) - (g^{i+1}, \theta(t^{i+1})) \bigr|
\le C^\theta_{305} \bigl(|s^{i+1}|+\|\nabla\overline{e}^{i+1}\| + \|\nabla\overline{e}^{i+1}_\theta \| \bigr)
+ C^\theta_{306} \|\nabla\overline{e}^{i+1}_\theta \|,
\end{eqnarray}
where $C_{305}^{\theta}
= C_{g}\sqrt{C_{301}} \max\{1, M_{T}C_{302}C \}$
and $C_{306}^{\theta}=M_T C_g C$.

\medskip
Defining
\begin{equation}\label{eq:C311_def}
C_{307}
= \max\!\left\{
C_{303}+C_{305}^{u}+C_{305}^{\theta},\;
C_{304}^{u}+C_{306}^{u},
C_{304}^{\theta}+C_{306}^{\theta}
\right\},
\end{equation}
the sum of all terms in
\eqref{eq:s_error_step3} satisfies:
\begin{equation}\label{eq:s_increment_bound}
|s^{i+1}-s^{i}|
\le C_{307}\delta t
\!\left(
|s^{i+1}|
+\|\nabla\overline{\mathbf{e}}^{i+1}\|
+\|\nabla\overline{e}_{\theta}^{i+1}\|
\right)
+|T^{i}|.
\end{equation}

  \subsubsection{Step 3C. Summed estimate for $s^{n+1}$.}
Summing \eqref{eq:s_increment_bound} from
$i = 0$ to $m$, using $s^{0} = 0$
(since $r^{0} = E(\mathbf{u}^{0},\theta^{0})+\overline{C}
= r(t^{0})$):
\begin{align}
|s^{\,m+1}|
&\le C_{307}\delta t
\sum_{i=0}^{m}
|s^{\,i+1}|
+C_{307}\delta t
\sum_{i=0}^{m}
\!\left(
\|\nabla\overline{\mathbf{e}}^{\,i+1}\|
+\|\nabla\overline{e}_{\theta}^{\,i+1}\|
\right)
+\sum_{i=0}^{m}|T^{i}|.
\label{eq:s_sum_step3}
\end{align}
By the Cauchy-Schwarz inequality
and the estimate \eqref{eq:combined_err},
\begin{align}
\delta t\sum_{i=0}^{m}
\!\left(
\|\nabla\overline{\mathbf{e}}^{\,i+1}\|
+\|\nabla\overline{e}_{\theta}^{\,i+1}\|
\right)
&\le
\sqrt{T}\!
\left(
\delta t\sum_{i=0}^{m}
\|\nabla\overline{\mathbf{e}}^{\,i+1}\|^{2}
\right)^{1/2}
+\sqrt{T}\!
\left(
\delta t\sum_{i=0}^{m}
\|\nabla\overline{e}_{\theta}^{\,i+1}\|^{2}
\right)^{1/2}
\nonumber\\
&\le 2 T \cdot\sqrt{C_{2}} \delta t^{2},
\label{eq:gradient_sum_bound}
\end{align}
where we used that
$\delta t\sum_{i=0}^{m}\|\nabla\overline{\mathbf{e}}^{\,i+1}\|^{2}
\le T C_{2} \delta t^{4}$
from \eqref{eq:combined_err}.
From \eqref{eq:Tn_bound_step3} and the regularity
hypotheses on the exact solution, one achieves
\begin{equation}\label{eq:Tn_sum_bound}
\sum_{i=0}^{m}|T^{i}|
\le \delta t
\int_{0}^{T}
\!\left(
\|\mathbf{u}_{t}\|^{2}
+\|\mathbf{u}\|\|\mathbf{u}_{tt}\|
+\frac{1}{\alpha}\|\theta_{t}\|^{2}
+\frac{1}{\alpha}\|\theta\|\|\theta_{tt}\|
\right)\,ds
=: C_{\mathrm{trunc}}^{s}\delta t.
\end{equation}
Substituting \eqref{eq:gradient_sum_bound}
and \eqref{eq:Tn_sum_bound} into
\eqref{eq:s_sum_step3}, we obtain
\begin{equation}\label{eq:s_sum_gronwall_form}
(1-C_{307}\delta t)  |s^{\,m+1}|
\le C_{307}\delta t
\sum_{i=0}^{m}|s^{\,i}|
+2TC_{307}\sqrt{C_2}\,\delta t^{2}
+C_{\mathrm{trunc}}^{s}\delta t.
\end{equation}
We apply the Gr\"onwall Lemma~\ref{lem:discrete_gronwall_2} 
and obtain 
\begin{align}
|s^{\,n+1}|
\le \exp(2C_{307}T)
\,\left(
4TC_{307}\sqrt{C_2}\,\delta t^{2}
+2C_{\mathrm{trunc}}^{s}\delta t
\right)
\le C_{4}\delta t,
\label{eq:s_bound_step3}
\end{align}
where 
\begin{equation}\label{eq:C4_def}
C_{4}
= \max\left\{ 
\exp(4C_{307}T )\, (2T C_{307}\sqrt{C_{2}}
+ 2C_{\mathrm{trunc}}^{s} ),
\frac{\sqrt{C_{307}}}{2} \right\}.
\end{equation}
Here, we assume that $\delta t<\max\{1, \frac{1}{2C_{307}}\}$, which holds when $\delta t<\frac{1}{1+2C_0^2}$ for $C_0$ defined in \eqref{eq:C0_choice_step3}.

\subsubsection{Step 3D. Bound on $|1 - \xi^{n+1}|$.}

Recall from \eqref{eq:GSAV_Bouss_compact_e_corr}
that
$\xi^{\,n+1}
= r^{\,n+1}/(\overline{E}^{\,n+1}+\overline{C})$.
We decompose
\begin{equation}\label{eq:xi_decomp_step3}
1 - \xi^{\,n+1}
= \frac{\overline{E}^{\,n+1}+\overline{C}-r^{\,n+1}}
{\overline{E}^{\,n+1}+\overline{C}}
=
\frac{\bigl( \overline{E}^{\,n+1}- E(\mathbf{u}(t^{n+1}),\theta(t^{n+1})) \bigr)
-s^{\,n+1}}
{\overline{E}^{\,n+1}+\overline{C}}.
\end{equation}
Since $\overline{E}^{\,n+1}+\overline{C}\ge\overline{C}\ge 1$,
using \eqref{eq:s_bound_step3},
\eqref{eq:energy_diff_bound},
and \eqref{eq:combined_err} implies
\begin{align}
|1-\xi^{\,n+1}|
&\le
|s^{\,n+1}|
+\left|
\overline{E}^{\,n+1}-
E(\mathbf{u}(t^{n+1}),\theta(t^{n+1}))
\right|
\le
|s^{\,n+1}|
+C_{302}C
\!\left(
\|\nabla\overline{\mathbf{e}}^{\,n+1}\|
+\|\nabla\overline{e}_{\theta}^{\,n+1}\|
\right)
\nonumber\\
&\le\delta t\;
[ C_4 + 2CC_{302}\sqrt{C_2} \delta t]
\le \delta t\; C_5 (1+\delta t).
\label{eq:xi_bound_step3}
\end{align}
where
\begin{equation}\label{eq:C5_def}
C_{5}
= \max\!\left\{
1,\; C_{4},\; 2C_{302}C\sqrt{C_{2}}
\right\}.
\end{equation}
The constant $C_{5}$ is independent of
$C_{0}$ and $\delta t$.
We choose 
\begin{equation}\label{eq:C0_choice_step3}
C_{0} = 2C_{5}.
\end{equation}
Then according to the choice $\delta t\le \frac{1}{1+2C_0^2}$, it follows that
\begin{align}
|1-\xi^{n+1}| &\le \left(C_5 + \frac{C_5}{1+2C_0^2} \right) \delta t
=\left(C_5 + \frac{C_5}{1+8C_5^2} \right) \delta t
\le \left( \frac{C_0}{2} + \frac{1}{4\sqrt{2}}\right) 
\delta t
\le C_0 \delta t,
\end{align}
where the fact $C_0\ge 2$ is used in the last step.

\subsubsection{Step 3E. Final velocity and temperature error estimates}

We now derive bounds for the unbarred errors
$\mathbf{e}^{\,n+1}$ and $e_{\theta}^{\,n+1}$.
Recall $\mathbf{u}^{\,n+1}
= \eta^{\,n+1}\overline{\mathbf{u}}^{\,n+1}$
and $\theta^{\,n+1}
= \eta^{\,n+1}\overline{\theta}^{\,n+1}$,
so that
\begin{equation}\label{eq:err_unbar_split}
\mathbf{e}^{\,n+1}
= \overline{\mathbf{e}}^{\,n+1}
+(\eta^{\,n+1}-1)\overline{\mathbf{u}}^{\,n+1},
\qquad
e_{\theta}^{\,n+1}
= \overline{e}_{\theta}^{\,n+1}
+(\eta^{\,n+1}-1)\overline{\theta}^{\,n+1}.
\end{equation}
From the completed induction,
$|1-\xi^{\,n+1}|\le C_{0}\delta t$,
and $1-\eta^{\,n+1}=(1-\xi^{\,n+1})^{2}$,
so
\begin{equation}\label{eq:eta_error_step3}
|\eta^{\,n+1}-1|
= (1-\xi^{\,n+1})^{2}
\le C_{0}^{2}\delta t^{2}.
\end{equation}
Using \eqref{eq:eta_error_step3} and
$\|\nabla\overline{\mathbf{u}}^{\,n+1}\|\le\sqrt{C_{3}}$
from \eqref{eq:C3_step2_def}:
\begin{equation}\label{eq:grad_nabla_split}
\|\nabla(\mathbf{u}^{\,n+1}
-\overline{\mathbf{u}}^{\,n+1})\|^{2}
= |\eta^{\,n+1}-1|^{2}
\|\nabla\overline{\mathbf{u}}^{\,n+1}\|^{2}
\le C_{0}^{4}\delta t^{4}\cdot C_{3}.
\end{equation}
Therefore, by the triangle inequality
and \eqref{eq:combined_err}, one attains
\begin{align}
\|\nabla\mathbf{e}^{\,n+1}\|^{2}
&\le 2\|\nabla\overline{\mathbf{e}}^{\,n+1}\|^{2}
+2\|\nabla(\mathbf{u}^{\,n+1}
-\overline{\mathbf{u}}^{\,n+1})\|^{2}
\le 
(2C_{2}+2C_{3} C_{0}^4)\delta t^{4}.
\label{eq:unbar_vel_grad}
\end{align}
Identically for the temperature:
\begin{equation}\label{eq:unbar_temp_grad}
\|\nabla e_{\theta}^{\,n+1}\|^{2}
\le \! (2C_{2}+2C_{3} C_{0}^4)\delta t^{4}.
\end{equation}

Since $\mathbf{e}^{\,i}
= \overline{\mathbf{e}}^{\,i}
+(\eta^{\,i}-1)\overline{\mathbf{u}}^{\,i}$,
the triangle inequality gives
\begin{equation}\label{eq:Delta_split}
\|\Delta\mathbf{e}^{\,i}\|^{2}
\le 2\|\Delta\overline{\mathbf{e}}^{\,i}\|^{2}
+2|\eta^{\,i}-1|^{2}
\|\Delta\overline{\mathbf{u}}^{\,i}\|^{2}
\le 2\|\Delta\overline{\mathbf{e}}^{\,i}\|^{2}
+2C_{0}^{4}\delta t^{4}
\|\Delta\overline{\mathbf{u}}^{\,i}\|^{2}.
\end{equation}
Multiplying by $\nu\delta t$ and summing
from $i=0$ to $n+1$, then using
$\nu\delta t\sum\|\Delta\overline{\mathbf{e}}^{\,i}\|^{2}
\le C_{2}\delta t^{4}$
from \eqref{eq:combined_err} and
$\nu\delta t\sum\|\Delta\overline{\mathbf{u}}^{\,i}\|^{2}
\le C_{1}^{u}$ from \eqref{eq:C1_bnd_for_step2}:
\begin{equation}\label{eq:lap_sum_vel}
\nu\delta t
\sum_{i=0}^{n+1}
\|\Delta\mathbf{e}^{\,i}\|^{2}
\le 2C_{2}\delta t^{4}
+2C_{0}^{4}C_{1}^{u}\delta t^{4}
\le (2C_{2}+2C_0^4 C_{1}^{u})\delta t^{4}.
\end{equation}
An identical argument for the temperature gives, 
using $\gamma\delta t\sum\|\Delta\overline{e}_\theta^{\,i}\|^{2}
\le C_{2}\delta t^{4}$
from \eqref{eq:combined_err}
and $\gamma\delta t\sum\|\Delta\overline\theta\|^{2}\le C_{1}^{\theta}$ from \eqref{C_1theta_def},  
\begin{equation}\label{eq:lap_sum_temp}
\gamma\delta t
\sum_{i=0}^{n+1}
\|\Delta e_{\theta}^{\,i}\|^{2}
\le \bigl(2C_{2}+2C_0^4 C_{1}^{\theta}) \delta t^{4}.
\end{equation}

\subsubsection{Step 3F. Pressure error estimate}

We return to the pressure error equation
\eqref{eq:press_err_eq}. Taking
$q = e_{p}^{\,i}$ in \eqref{eq:press_err_eq}:
\begin{equation}\label{eq:press_err_q_ep}
\|\nabla e_{p}^{\,i}\|^{2}
= \!\left(
\mathbf{u}(t^{i})\cdot\nabla\mathbf{u}(t^{i})
-\overline{\mathbf{u}}^{\,i}\cdot\nabla\overline{\mathbf{u}}^{\,i},
\nabla e_{p}^{\,i}
\right)
+\nu(\nabla p_{s}(\overline{\mathbf{e}}^{\,i}),\nabla e_{p}^{\,i})
-(\overline{e}_{\theta}^{\,i}\mathbf{e}_{2},\nabla e_{p}^{\,i}).
\end{equation}
We bound each term on the right-hand side
of \eqref{eq:press_err_q_ep}.
For the nonlinear term, we use
$\mathbf{u}\cdot\nabla\mathbf{u}
-\overline{\mathbf{u}}\cdot\nabla\overline{\mathbf{u}}
= -\overline{\mathbf{e}}^{\,i}\cdot\nabla\overline{\mathbf{u}}^{\,i}
-\mathbf{u}(t^{i})\cdot\nabla\overline{\mathbf{e}}^{\,i}$
and the estimates in \eqref{convection_estimate3} to derive, along with Young's inequality,
\begin{align}
\left|
\!\left(
\mathbf{u}(t^{i})\cdot\nabla\mathbf{u}(t^{i})
-\overline{\mathbf{u}}^{\,i}\cdot\nabla\overline{\mathbf{u}}^{\,i},
\nabla e_{p}^{\,i}
\right)
\right|
&\le
\!\left(
C\|\nabla\overline{\mathbf{e}}^{\,i}\|
\|\Delta\overline{\mathbf{u}}^{\,i}\|
+C\|\mathbf{u}(t^{i})\|_{2}
\|\nabla\overline{\mathbf{e}}^{\,i}\|
\right)
\|\nabla e_{p}^{\,i}\|
\nonumber\\
&\le
\frac{1}{4}
\|\nabla e_{p}^{\,i}\|^{2}
+2C^{2}
\|\nabla\overline{\mathbf{e}}^{\,i}\|^{2}
\!\left(
\|\Delta\overline{\mathbf{u}}^{\,i}\|^{2}
+C_{\mathrm{reg}}^{u}
\right).
\label{eq:press_nonlin_bound}
\end{align}
For the Stokes pressure term, we apply \eqref{eq:stokes_pressure_estimate} with $\varepsilon = \tfrac{1}{4}$ to get
\begin{align}
\nu|(\nabla p_{s}(\overline{\mathbf{e}}^{\,i}),
\nabla e_{p}^{\,i})|
\le \frac{1}{4}\|\nabla e_{p}^{\,i}\|^{2}
+\nu^{2}\!\left[
\frac{3}{4}\|\Delta\overline{\mathbf{e}}^{\,i}\|^{2}
+C_{S}(\tfrac{1}{4})
\|\nabla\overline{\mathbf{e}}^{\,i}\|^{2}
\right].
\label{eq:press_stokes_bound}
\end{align}
The temperature buoyancy term can be bounded by the Cauchy-Schwarz inequality and the Poincare inequality:
\begin{equation}\label{eq:press_temp_bound}
|(\overline{e}_{\theta}^{\,i}\mathbf{e}_{2},
\nabla e_{p}^{\,i})|
\le \|\overline{e}_{\theta}^{\,i}\|
\|\nabla e_{p}^{\,i}\|
\le C
\|\nabla\overline{e}_{\theta}^{\,i}\|
\|\nabla e_{p}^{\,i}\|
\le \frac{1}{4}\|\nabla e_{p}^{\,i}\|^{2}
+C^{2}
\|\nabla\overline{e}_{\theta}^{\,i}\|^{2}.
\end{equation}
Substituting \eqref{eq:press_nonlin_bound},
\eqref{eq:press_stokes_bound}, and
\eqref{eq:press_temp_bound} into
\eqref{eq:press_err_q_ep} to get
\begin{equation}\label{eq:press_pw_bound}
\tfrac{1}{4}\|\nabla e_{p}^{\,i}\|^{2}
\le 2C^{2}\|\nabla\overline{\mathbf{e}}^{\,i}\|^{2}
\!\left(\|\Delta\overline{\mathbf{u}}^{\,i}\|^{2}
+C_{\mathrm{reg}}^{u}\right)
+\tfrac{3\nu^{2}}{4}\|\Delta\overline{\mathbf{e}}^{\,i}\|^{2}
+\nu^{2}C_{S}(\tfrac{1}{4})\|\nabla\overline{\mathbf{e}}^{\,i}\|^{2}
+C^{2}\|\nabla\overline{e}_{\theta}^{\,i}\|^{2}.
\end{equation}
Multiplying \eqref{eq:press_pw_bound} by $4\delta t$
and summing from $i=0$ to $n+1$:
\begin{align}
\delta t\sum_{i=0}^{n+1}\|\nabla e_{p}^{\,i}\|^{2}
&\le 8C^{2}\,\delta t\sum_{i=0}^{n+1}
\|\nabla\overline{\mathbf{e}}^{\,i}\|^{2}\|\Delta\overline{\mathbf{u}}^{\,i}\|^{2}
+\bigl[8C^{2}C_{\mathrm{reg}}^{u}+4\nu^{2}C_{S}(\tfrac{1}{4})\bigr]
\delta t\sum_{i=0}^{n+1}\|\nabla\overline{\mathbf{e}}^{\,i}\|^{2}
\nonumber\\
&\quad
+3\nu^{2}\,\delta t\sum_{i=0}^{n+1}\|\Delta\overline{\mathbf{e}}^{\,i}\|^{2}
+4C^{2}\,\delta t\sum_{i=0}^{n+1}\|\nabla\overline{e}_{\theta}^{\,i}\|^{2}.
\label{eq:press_sum_step3}
\end{align}
Applying \eqref{eq:combined_err} at each $i\le n+1$ gives
$\max_{0\le i\le n+1}
\{ \|\nabla\overline{\mathbf{e}}^{\,i}\|^{2}, 
\|\nabla\overline{e}_{\theta}^{\,i}\|^{2} \}
\le C_{2}\delta t^{4}$.
Thus, $\delta t\sum_{i=0}^{n+1}\|\nabla\overline{\mathbf{e}}^{\,i}\|^{2} \le TC_{2}\delta t^{4}$ and 
$\delta t\sum_{i=0}^{n+1}\|\nabla\overline{e}_{\theta}^{\,i}\|^{2} \le TC_{2}\delta t^{4}$.
The bound in \eqref{eq:combined_err} leads to 
$\delta t\sum_{i=0}^{n+1}\|\Delta\overline{\mathbf{e}}^{\,i}\|^{2} \le \tfrac{C_{2}}{4\nu}\,\delta t^{4}$.
The estimate in \eqref{eq:C1_bnd_for_step2} gives
\begin{equation}\label{eq:cross_press_bound}
\delta t\sum_{i=0}^{n+1}
\|\nabla\overline{\mathbf{e}}^{\,i}\|^{2}\|\Delta\overline{\mathbf{u}}^{\,i}\|^{2}
\le \max_{0\le i\le n+1}\|\nabla\overline{\mathbf{e}}^{\,i}\|^{2}
\cdot\delta t\sum_{i=0}^{n+1}\|\Delta\overline{\mathbf{u}}^{\,i}\|^{2}
\le \tfrac{C_{1}^{u}C_{2}}{\nu}\,\delta t^{4}.
\end{equation}
Substituting the above estimates 
into \eqref{eq:press_sum_step3} results in 
\begin{equation}\label{eq:press_final_step3}
\delta t\sum_{i=0}^{n+1}\|\nabla e_{p}^{\,i}\|^{2}
\le C_{p}^{*}\,\delta t^{4},
\end{equation}
where
\begin{equation}\label{eq:Cp_star_def}
C_{p}^{*}
:= C_{2}\!\left[
\tfrac{8C^{2}C_{1}^{u}}{\nu}
+8T C^{2}C_{\mathrm{reg}}^{u}
+4T \nu^{2}C_{S}(\tfrac{1}{4})
+\tfrac{3\nu}{4}
+4C^{2}T
\right].
\end{equation}

\subsubsection{Step 3G. Conclusion of the proof}

Define the final error constant
\begin{equation}\label{eq:C53_final}
C_{Bou}
= \max\!\left\{
2C_{2}+2C_{3}C_0^4,\;
2C_{2}+2C_{1}^{u}C_0^4,\;
2C_{2}+2C_{1}^{\theta}C_0^4,\;
C_{p}^{*}
\right\}
\end{equation}
where $C_{0} = 2C_{5}$ is given by
\eqref{eq:C0_choice_step3}.
Collecting
\eqref{eq:combined_err},
\eqref{eq:unbar_vel_grad},
\eqref{eq:unbar_temp_grad},
\eqref{eq:lap_sum_vel},
\eqref{eq:lap_sum_temp},
and \eqref{eq:press_final_step3},
we obtain, for all $n+1$ with $(n+1)\delta t\le T$
with $\delta t\le\frac{1}{1+2C_{0}^{2}}$, the error estimate \eqref{error_analysis} holds.

\end{proof}

\begin{remark}
\label{rem:C53_dependence}
The factor $\nu^{-1}$ enters through the Stokes
pressure estimate $C_S(\varepsilon)\sim\varepsilon^{-3}$ in
\eqref{eq:stokes_pressure_estimate} and the stability condition
$\kappa^u\ge 0$ in \eqref{kappa_u_def}, which forces the 
parameters $\varepsilon_p,\varepsilon_1',\varepsilon\sim\nu$.  The  factor
$\gamma^{-1}$ enters through the temperature stability condition
$\kappa^\theta\ge 0$ in \eqref{kappa_theta_def}, which forces
$\varepsilon_c,\varepsilon_{9..12}\sim\gamma$. Tracing these factors forward
through the constants of the proof,
one finds that
\begin{eqnarray}
C_0,\, C_{\mathrm{Bou}} &\sim& \exp\!\Bigl(\exp\!\bigl(\exp\!\bigl(\exp(\mathrm{poly}(\nu^{-1},\gamma^{-1}))\bigr)\bigr)\Bigr).
\end{eqnarray}
The appearance of ``$\exp$'' is due to the four successive applications of the discrete Gr\"onwall lemma in the proof (Steps~1A, 1B, 2C, 3C). 
Here ``$\mathrm{poly}(\nu^{-1},\gamma^{-1})$'' denotes a polynomial in $\nu^{-1}$ and $\gamma^{-1}$.
\end{remark}

\begin{remark}[Time-step restriction]
\label{rem:dt_restriction}
The time-step condition $\delta t\le(1+2C_{0}^{2})^{-1}$
depends on $C_{0}$, which in turn depends
on $\nu^{-1}$ and $\gamma^{-1}$ through $C_{2}$.
Therefore, as $\nu\to 0$ or $\gamma\to 0$,
the admissible time step size $\delta t\to 0$.
\end{remark}


\section{Numerical Results}
\label{sec:numerical}
\subsection{Example 1}
\label{sec_numerical_Eg1}
We test the convergence of the GSAV scheme described in Section\,\ref{sec:GSAV_scheme} to solve a problem with the exact solution
\begin{equation}
\left\{
\begin{aligned}
u_1    &= \phantom{-}\sin(2\pi x_2)\,\sin^2(\pi x_1)\,\cos t, \\
u_2    &= -\sin(2\pi x_1)\,\sin^2(\pi x_2)\,\cos t, \\
\theta &= \sin(\pi x_1)\,\sin(\pi x_2)\,\cos t, \\
p      &= \cos(\pi x_1)\, x_2^3\,\cos t.
\end{aligned}
\right.
\end{equation}
on the domain $\Omega=[-1,1]^2$ with the no-slip boundary condition. The  parameter values are $\nu=\gamma=\alpha=\overline{C}=1$. Although \eqref{eq:Cbar_choice} would require $\overline{C}\sim 10^{5}$ for this manufactured solution, the scheme is stable in practice with $\overline{C}=1$; the theoretical lower bound is sufficient but not necessary.
The spatial discretization uses the Legendre Galerkin spectral method described in \cite{HuangShen2023} and \cite{shen2011spectral} with $N_{\text{mode}}=64$ modes in each spatial dimension such that the spatial error is negligible compared with time discretization errors. In Figure\,\ref{fig_cat001}, we plot the convergence rate of the $L^2$ and $H^1$ errors of the velocity and temperature at $T = 1$  with $k = 4$. We observe the second order convergence of velocity and temperature as proved in Theorem\,\ref{thm:error_2d_general_k}.

\begin{figure}[H]
\begin{center}
\includegraphics[scale=0.3]{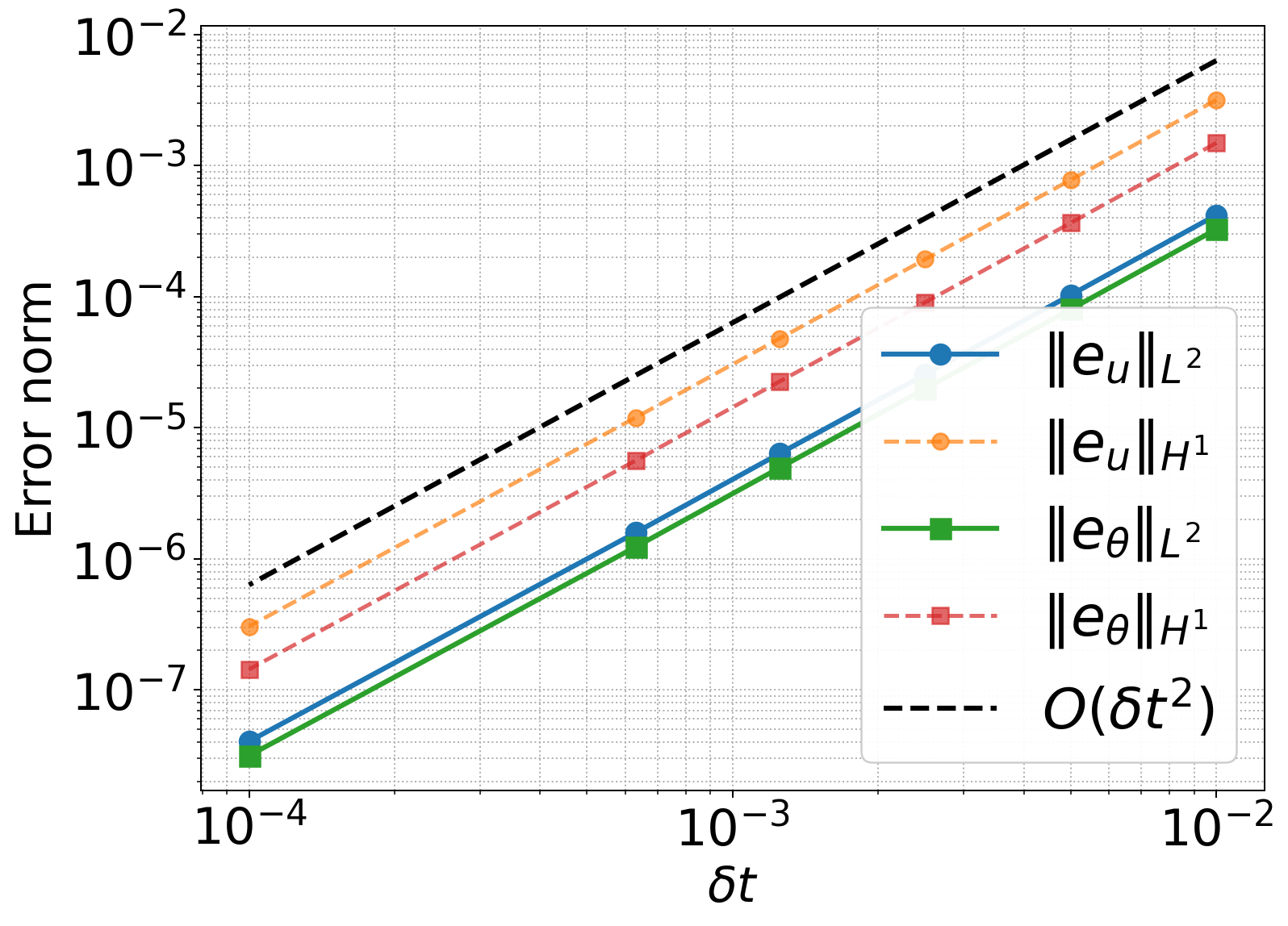}
\caption{ The second order convergence of the GSAV scheme on Example 1. }
\label{fig_cat001}
\end{center}
\end{figure}

\subsection{Example 2: a stratification simulation}
\label{sec_numerical_Eg2}
In the original Boussinesq problem \eqref{eq:momentum}--\eqref{eq:incompressibility}, we set the equilibrium solution of the temperature as $\Theta_{he}=x_2$, that is, $\alpha=1$. In the perturbed system \eqref{eq:bouss}, the external forces are ${\bf f}=g=0$. 
The boundary condition is 
$\mathbf{u}=\mathbf{0}$  and $\theta = 0$ on $\partial\Omega$.  Then the total
temperature is
$\Theta(x_1,x_2,t) = \theta(x_1,x_2,t) +  x_2$. 
The viscosity is $\nu=10^{-2}$, and the thermal diffusivity is $\gamma=10^{-4}$. 
The initial condition is set as $\mathbf{u}_{0}= \mathbf{0}$ and
\begin{align}
\theta_{0}({\bf x}) &= \phi({\bf x})
\Bigg[\;
A_h\,\exp\!\left(
   -\frac{s_1({\bf x})}{2\sigma^{2}}
\right) -
2A_h\,\exp\!\left(
   -\frac{s_2({\bf x})}{\sigma^{2}}
\right)
\Bigg], 
\label{eq:ic}
\end{align}
where
$\phi({\bf x})=(1-x_1^{2})(1-x_2^{2})$, 
$A_h = 10$, 
$s_1({\bf x})=(x_1-x_h)^{2}+(x_2-y_h)^{2}$,
$s_2({\bf x})=(x_1-x_c)^{2}+(x_2-y_c)^{2}$,
$(x_h,y_h) = (\,0.6,\,-0.7\,)$, 
$(x_c,y_c) = (-0.8,\,0.6)$, 
$\sigma = 0.2$.
As shown in Figure\,\ref{fig:snapshots} at $t=0$, the initial  $\theta$ has one cold blob in the upper left region and a hot blob in the lower right region.

The spatial discretization is the same as in Example\,1. We run to final time $T=10^{4}$ using time step size $\delta t = 2\times 10^{-4}$ and time shift $k=4$. The stabilization constant is set to $\overline{C}=1000$, which satisfies~\eqref{eq:Cbar_choice}. A convergence study with $N_{\text{mode}}=128,\,256,\,512$ and the correspondingly refined time step sizes $\delta t = 5\times 10^{-4},\,2\times 10^{-4},\,10^{-4}$ yielded essentially indistinguishable results.
Figure~\ref{fig:energy} shows the
evolution of the $L^2$ norms of velocity and temperature perturbations. Some representative snapshots of the temperature perturbation are shown in Figure\,\ref{fig:snapshots}. These results reveal three regimes of the dynamics.

First, the buoyancy transient regime: $0\!\le\!t\!\lesssim\!2$. $\|\mathbf{u}\|_{L^{2}}$ rises from $\approx 4\!\times\!10^{-4}$ to a peak $\approx 0.94$ at $t\approx 2$, while $\|\theta\|_{L^{2}}$ drops sharply as the warm blob in the lower right rises and the cold blob in the upper left sinks, converting the available potential energy stored in the initial Gaussian blobs into kinetic energy. This regime is captured in the early snapshots ($t=0,1,2,3$) of Figure~\ref{fig:snapshots}.

Second, the internal-wave oscillation regime: $2\!\lesssim\!t\!\lesssim\!60$. The kinetic energy decays through oscillatory bursts whose period is 4.44, agreeing with the theoretical prediction \eqref{eq:period_prediction} with $\alpha=1$  in Appendix\,\ref{app:wave_period}. 
These oscillations appear as the visible ripples in $|\mathbf{u}|_{L^{2}}$ shown in Figure~\ref{fig:energy}(left). The $L^2$ norm of $\theta$ exhibits the same oscillation period, but with a much smaller amplitude. 
The intermediate snapshots ($t=5,10,20,50$) show how the released kinetic energy organises into large-scale internal-wave motion that is gradually damped by viscosity.

Similar oscillations have been observed in~\cite{BelaynehChenNadamWuZheng2026}, in a closely related setting where thermal diffusion is absent and the eventual steady-state temperature profile is nonlinear.  In that regime the oscillations in the temperature field are far more pronounced than in the present, fully-diffusive case, and the authors describe them as ``seesaw-like''.  We refer to that work for additional snapshots illustrating these oscillations.

Third, the exponential decay regime: $t\!\gtrsim\!100$.
Once the wave activity has been damped by viscosity, the remaining motion projects onto the slowest Stokes/temperature eigenmode, and both $L^2$ norms decay exponentially.
Numerically (Figure~\ref{fig:energy}, right), $\|\mathbf{u}\|_{L^2}$ and $\|\theta\|_{L^2}$ decay as $\approx \exp(-2.8\times 10^{-3}\,t)$ and $\approx \exp(-2.4\times 10^{-3}\,t)$, respectively, for $t\gtrsim 2\times 10^{3}$.
The kinetic norm decays slightly faster than the temperature norm because $\nu=10^{-2}\gg\gamma=10^{-4}$.
Both rates are consistent with Lemma~\ref{lemma:exp_decay}, which gives the $L^2$-norm decay rate upper bound $\exp(-\lambda_{\min}\cdot \min(\nu,\gamma)\,t)\approx \exp(-4.93\times 10^{-4}\,t)$, where $\lambda_{\min}= \pi^2/2$ according to \eqref{eq:gravest_wavenumbers}.

\begin{figure}[H]
\centering
\begin{tabular}{@{}cc@{}}
\includegraphics[scale=0.3]{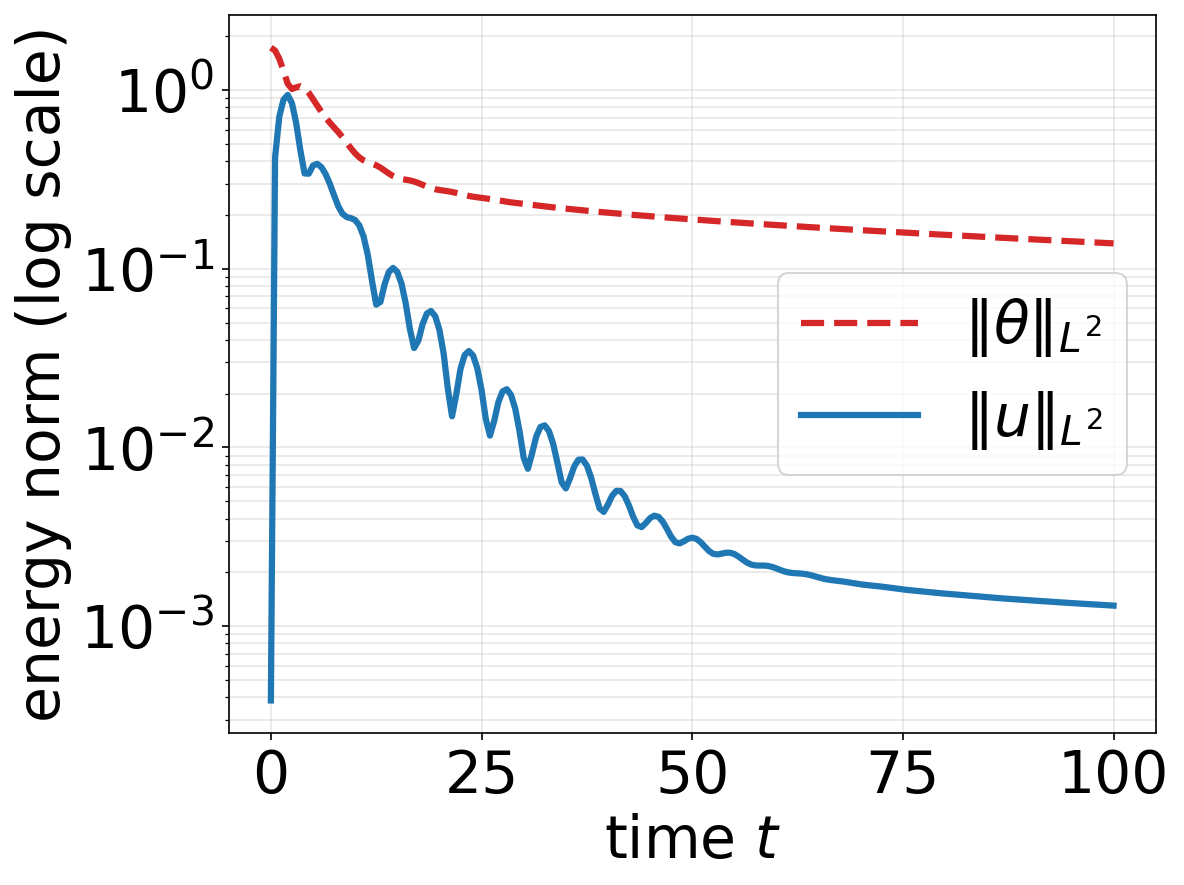} &
\includegraphics[scale=0.3]{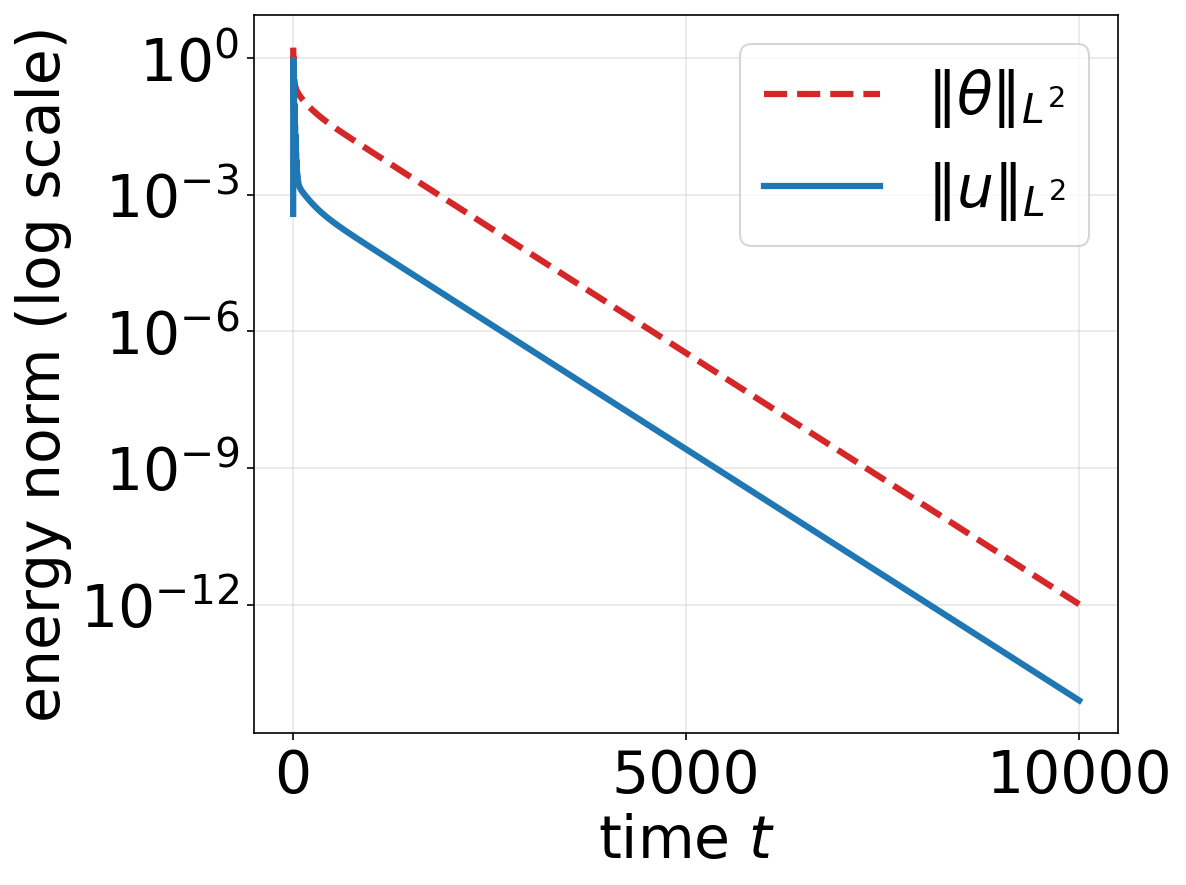} \\[2pt]
\end{tabular}
\caption{Time evolution of $\|\mathbf{u}\|_{L^{2}}$ (blue) and  $\|\theta\|_{L^{2}}$ (red) of Example 2.}
\label{fig:energy}
\end{figure}

\begin{figure}[H]
\centering
\begin{tabular}{@{}ccc@{}}
\includegraphics[width=0.31\linewidth]{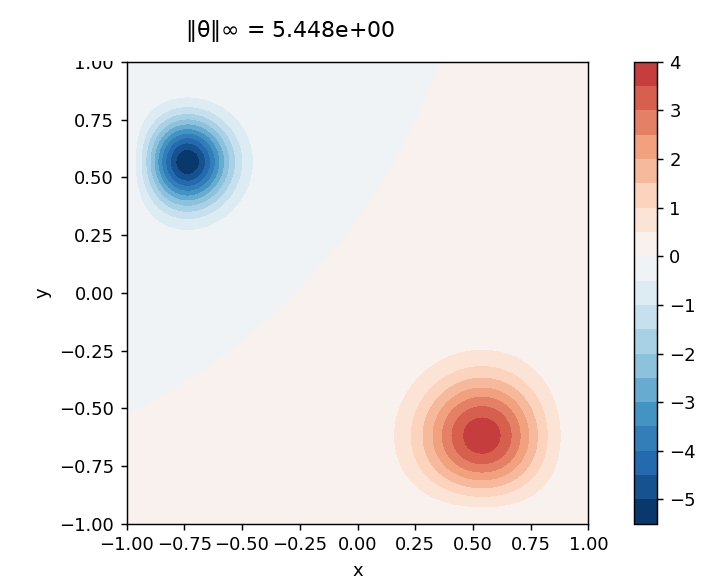} &
\includegraphics[width=0.31\linewidth]{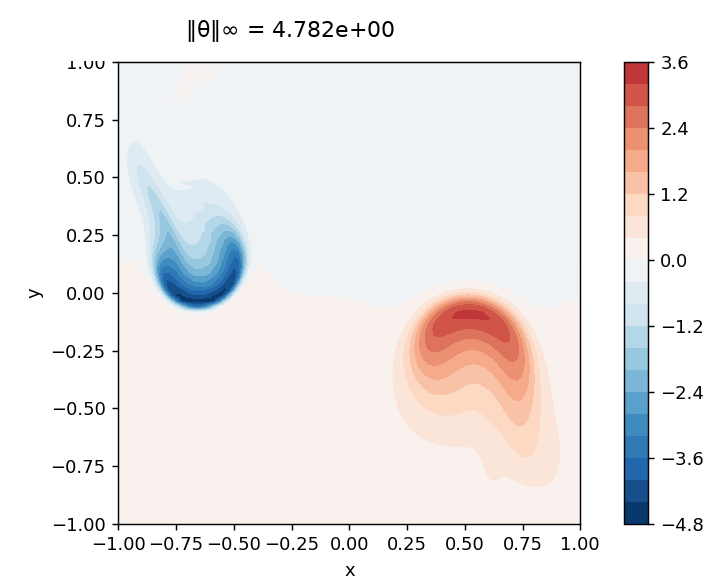} &
\includegraphics[width=0.31\linewidth]{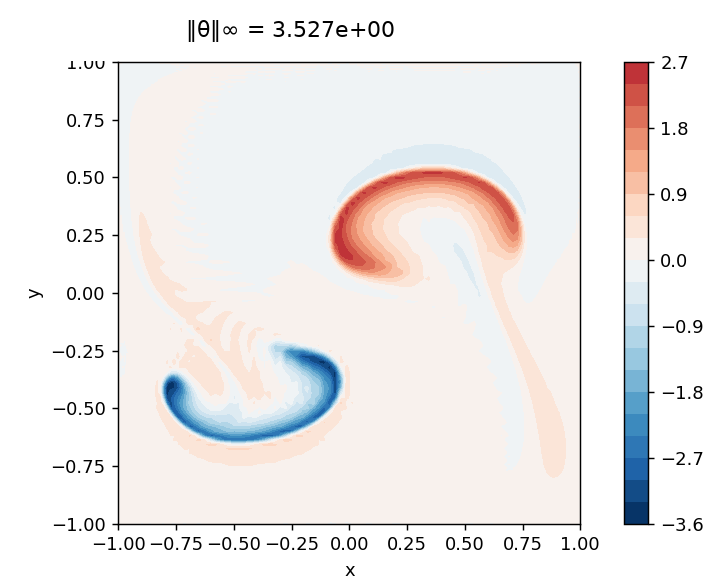} \\
{\footnotesize $t=0$} & {\footnotesize $t=1$} & {\footnotesize $t=2$} \\[4pt]
\includegraphics[width=0.31\linewidth]{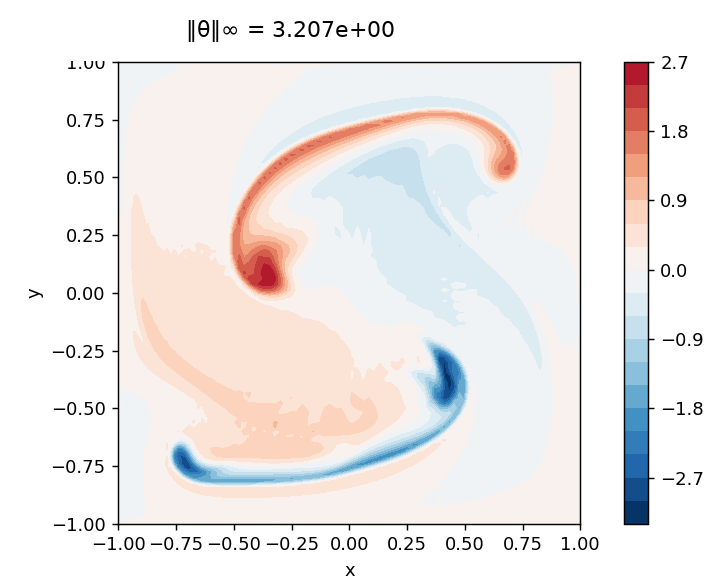} &
\includegraphics[width=0.31\linewidth]{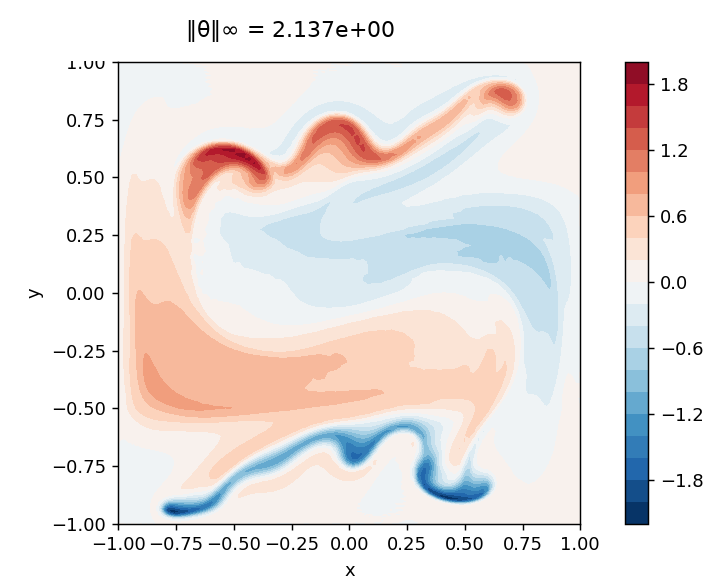} &
\includegraphics[width=0.31\linewidth]{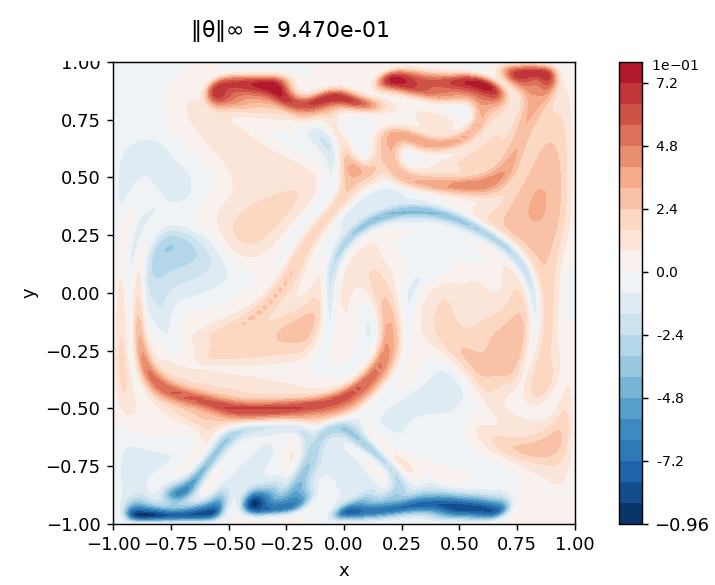} \\
{\footnotesize $t=3$} & {\footnotesize $t=5$} & {\footnotesize $t=10$} \\[4pt]
\includegraphics[width=0.31\linewidth]{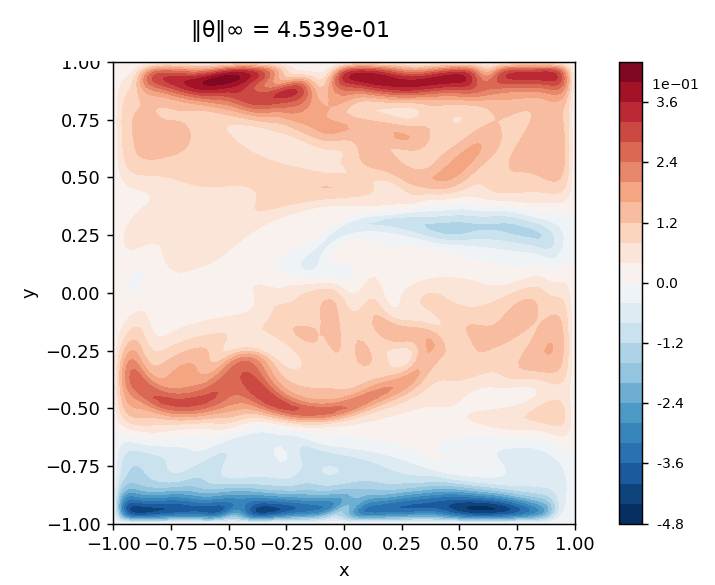} &
\includegraphics[width=0.31\linewidth]{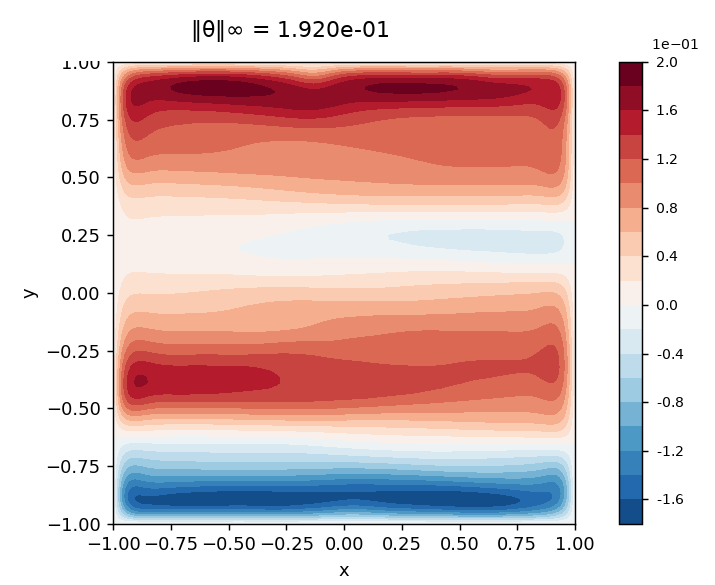} &
\includegraphics[width=0.31\linewidth]{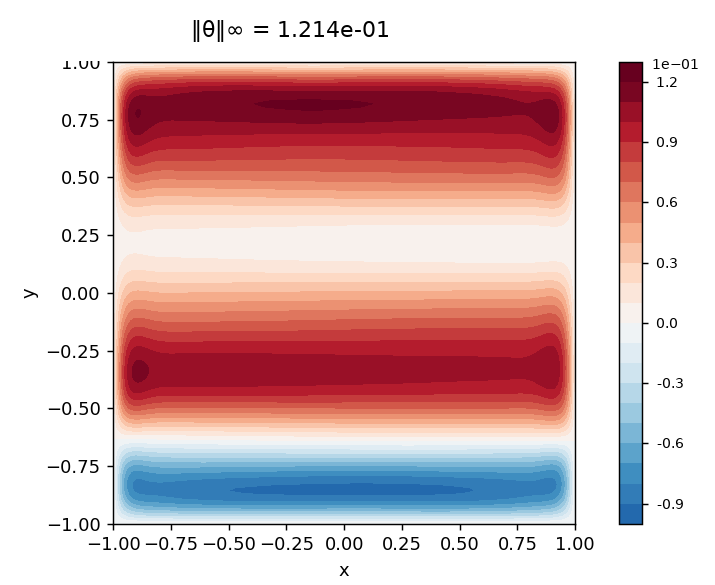} \\
{\footnotesize $t=20$} & {\footnotesize $t=50$} & {\footnotesize $t=100$} \\[4pt]
\includegraphics[width=0.31\linewidth]{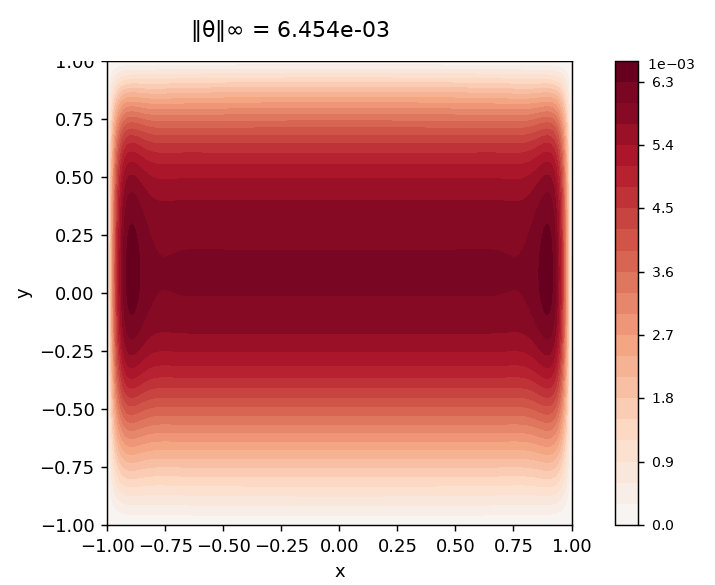} &
\includegraphics[width=0.31\linewidth]{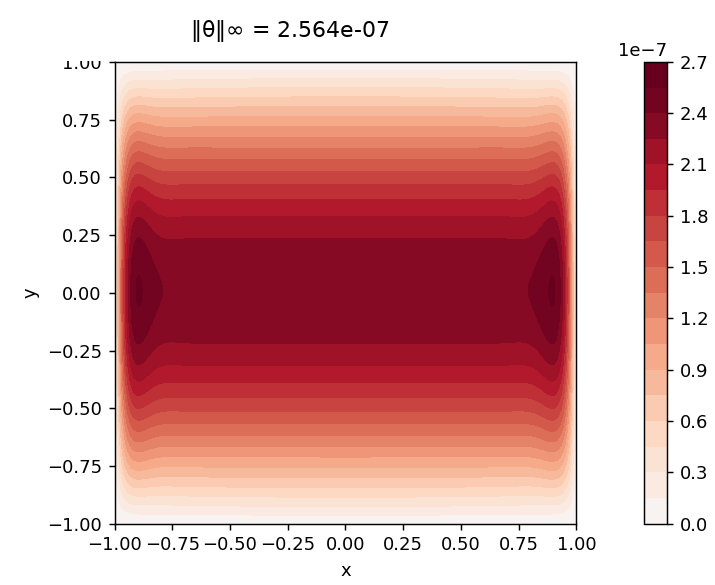} &
\includegraphics[width=0.31\linewidth]{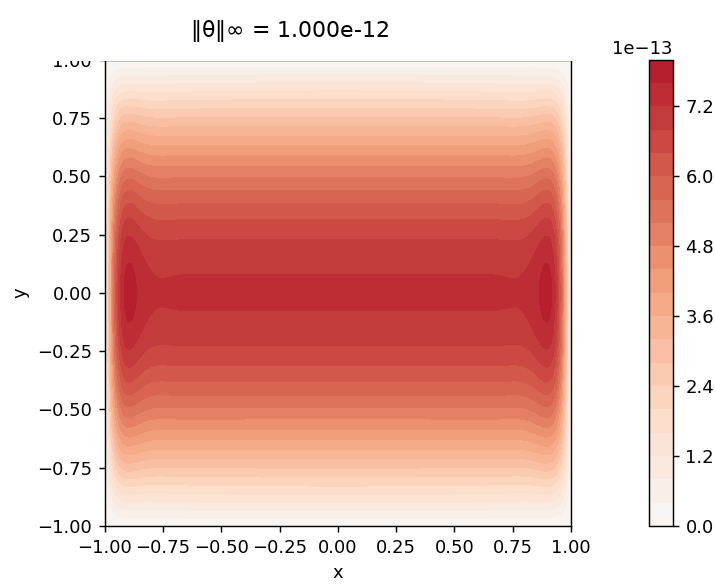} \\
{\footnotesize $t=1000$} & {\footnotesize $t=5000$} & {\footnotesize $t=10000$} \\
\end{tabular}
\caption{Example 2. Snapshots of the temperature perturbation $\theta(x,y,t)$ at
         times $t=0,1,2,3,5,10,20,50,100,1000,5000,10000$ (left-to-right,
         top-to-bottom). }
\label{fig:snapshots}
\end{figure}

\section{Discussion}\label{sec:discussion}

We have proposed and analyzed a second-order consistent-splitting GSAV scheme for the perturbed Boussinesq system~\eqref{eq:bouss} arising from a stable linear hydrostatic equilibrium with constant lapse rate.
The time discretization adopts the time-shifted BDF2 framework of Huang and Shen~\cite{HuangShen2023}.
The nonlinear convection $\mathbf{u}\cdot\nabla\mathbf{u}$ and advection $\mathbf{u}\cdot\nabla\theta$, together with the linear buoyancy $\theta\,\mathbf{e}_2$
and stratification $\alpha v$ couplings, are all treated explicitly through a second order extrapolation at the time level $t^{n+k}$, so that each time step
reduces to a small number of decoupled linear systems with constant coefficients.

\medskip
Theorem~\ref{thm:weak_stability} establishes an \emph{unconditional weak
stability} property: provided the stabilization constant $\overline{C}$ exceeds
an explicit threshold determined by the magnitudes of the forcing terms,
the solution ${\bf u}^n$ and $\theta^n$ remain uniformly bounded for all $n$, with no restriction on the time step.
Second, Theorem~\ref{thm:error_2d_general_k} establishes optimal second-order temporal convergence.
Numerical experiments in Section~\ref{sec:numerical} confirm the predicted
second-order rate on a manufactured solution and illustrate the long-time
relaxation toward hydrostatic balance in a stratified-flow simulation.

\medskip
A careful tracing shows that the error constant $C_{\mathrm{Bou}}$ contains negative powers of $\nu$ and $\gamma$, arising through a chain of estimates: the Stokes pressure bound $C_S(\varepsilon)\sim\varepsilon^{-3}$, the stability conditions $\kappa^u\ge 0$ and $\kappa^\theta\ge 0$, and the discrete Gr\"onwall inequality. The cumulative effect is a quadruply-exponential dependence of $C_{\mathrm{Bou}}$ on $\nu^{-1}$ and $\gamma^{-1}$ (see Remark~\ref{rem:C53_dependence}). A similar negative-power dependence was observed by Alhomsi, Wu, and Zheng~\cite{AlhomsiWuZhengSubmitted} for the consistent-splitting GSAV scheme of Huang and Shen~\cite{HuangShen2023} applied to the Navier--Stokes equations, and traced to the explicit treatment of the convection term. Their numerical experiments show that this scheme indeed blows up as $\nu\to 0^+$, whereas a fully implicit discretization solved by Newton's method remains accurate. For the perturbed Boussinesq system considered here, the situation is expected to be even less favorable, since the temperature equation introduces an additional explicit advection $\mathbf{u}\cdot\nabla\theta$ together with the bidirectional buoyancy--stratification coupling.

In Example~2 of Section~\ref{sec_numerical_Eg2}, we also tested smaller values of the viscosity and thermal diffusivity, with these runs blowing up at the present spatial and temporal resolutions. Although further mesh and time-step refinement may rescue them, this failure shows that the scheme is not robust: it requires an increasingly fine mesh and a smaller time step size to remain stable as the diffusion is reduced. A natural direction for future work is therefore to develop a more robust scheme that handles all viscosity regimes without requiring a prohibitively fine mesh or time step size.

\section*{Declarations}
\paragraph{Acknowledgments.}
J.\,Wu was partially supported by the National Science Foundation of the United States
(Grant No.\,DMS2104682 and DMS2309748). X.\,Zheng was partially supported by NSF grant No.\,DMS2309747.
This work was supported in part by computational resources
and services provided by HPCC of the Institute for Cyber-Enabled Research at Michigan State University
through a collaboration program of Central Michigan University, USA.

\paragraph{Data and Code Availability}
The datasets generated and analysed during the current study, and the
source code implementing the proposed scheme, are available from the
corresponding author upon reasonable request.

\appendix

\section{Proof of Lemma~\ref{lem:energy_law}}
\label{app:energy_law}

\begin{proof}
Taking the $L^2(\Omega)$ inner product of the momentum equation in
\eqref{eq:bouss} with $\mathbf{u}$ gives
\begin{equation}\label{eq:mom_inner}
(\partial_t\mathbf{u},\mathbf{u})
+(\mathbf{u}\cdot\nabla\mathbf{u},\mathbf{u})
+(\nabla p,\mathbf{u})
-\nu(\Delta\mathbf{u},\mathbf{u})
=(\theta\,\mathbf{e}_2,\mathbf{u})
+(\mathbf{f},\mathbf{u}).
\end{equation}
By the product rule, $(\partial_t\mathbf{u},\mathbf{u})
=\frac{1}{2}\frac{d}{dt}\|\mathbf{u}\|^2$.
Since $\nabla\cdot\mathbf{u}=0$ and $\mathbf{u}=0$ on $\partial\Omega$,
integration by parts yields
\[
(\mathbf{u}\cdot\nabla\mathbf{u},\mathbf{u})
=\frac{1}{2}\int_\Omega\mathbf{u}\cdot\nabla|\mathbf{u}|^2\,d\mathbf{x}
=-\frac{1}{2}\int_\Omega(\nabla\cdot\mathbf{u})|\mathbf{u}|^2\,d\mathbf{x}
=0,
\qquad
(\nabla p,\mathbf{u})
=-(p,\nabla\cdot\mathbf{u})
=0,
\]
and $-\nu(\Delta\mathbf{u},\mathbf{u})=\nu\|\nabla\mathbf{u}\|^2$.
Writing $\mathbf{u}=(u,v)$, the buoyancy term satisfies
$(\theta\,\mathbf{e}_2,\mathbf{u})=(\theta,v)$.
Substituting these identities into \eqref{eq:mom_inner} gives
\begin{equation}\label{eq:mom_identity}
\frac{1}{2}\frac{d}{dt}\|\mathbf{u}\|^2
+\nu\|\nabla\mathbf{u}\|^2
=(\theta,v)+(\mathbf{f},\mathbf{u}).
\end{equation}
Next, taking the $L^2(\Omega)$ inner product of the temperature equation
in \eqref{eq:bouss} with $\frac{1}{\alpha}\theta$ gives
\begin{equation}\label{eq:temp_inner}
\frac{1}{\alpha}(\partial_t\theta,\theta)
+\frac{1}{\alpha}(\mathbf{u}\cdot\nabla\theta,\theta)
+(v,\theta)
-\frac{\gamma}{\alpha}(\Delta\theta,\theta)
=\frac{1}{\alpha}(g,\theta).
\end{equation}
The same integration by parts argument, now using $\theta=0$ on
$\partial\Omega$, shows that $\frac{1}{\alpha}(\mathbf{u}\cdot\nabla\theta,\theta)=0$
and $-\frac{\gamma}{\alpha}(\Delta\theta,\theta)=\frac{\gamma}{\alpha}\|\nabla\theta\|^2$.
Hence \eqref{eq:temp_inner} reduces to
\begin{equation}\label{eq:temp_identity}
\frac{1}{2\alpha}\frac{d}{dt}\|\theta\|^2
+(v,\theta)
+\frac{\gamma}{\alpha}\|\nabla\theta\|^2
=\frac{1}{\alpha}(g,\theta).
\end{equation}
Adding \eqref{eq:mom_identity} and \eqref{eq:temp_identity}, the
coupling terms $(\theta,v)$ and $(v,\theta)$ cancel exactly, and
recalling the definition \eqref{eq:energy} yields \eqref{eq:energy_law}.
\end{proof}

\section{Proof of Lemma~\ref{lemma:exp_decay}}
\label{app:exp_decay}

\begin{proof}
With $\mathbf{f}=\mathbf{0}$ and $g=0$, the energy law
\eqref{eq:energy_law} reduces to
\begin{equation}\label{eq:energy_unforced}
\frac{d}{dt}E(\mathbf{u},\theta)
\;=\;
-\nu\|\nabla\mathbf{u}\|^2\;-\;\frac{\gamma}{\alpha}\|\nabla\theta\|^2.
\end{equation}
Since $\mathbf{u}=\mathbf{0}$ and $\theta=0$ on $\partial\Omega$,
both $\mathbf{u}(\cdot,t)$ and $\theta(\cdot,t)$ belong to
$H^1_0(\Omega)$.  The Poincare inequality therefore yields
\begin{equation}\label{eq:poincare}
\|\nabla\mathbf{u}\|^2\;\ge\;\lambda_1\,\|\mathbf{u}\|^2,
\qquad
\|\nabla\theta\|^2\;\ge\;\lambda_1\,\|\theta\|^2.
\end{equation}
Substituting \eqref{eq:poincare} into \eqref{eq:energy_unforced} and
recalling the definition $E=\tfrac12\|\mathbf{u}\|^2+\tfrac{1}{2\alpha}\|\theta\|^2$,
\begin{equation*}
\frac{d}{dt}E
\;\le\;
-\nu\lambda_1\|\mathbf{u}\|^2-\frac{\gamma\lambda_1}{\alpha}\|\theta\|^2
\;=\;
-2\lambda_1\!\left(\nu\cdot\tfrac12\|\mathbf{u}\|^2
+\gamma\cdot\tfrac{1}{2\alpha}\|\theta\|^2\right)
\;\le\;
-2\lambda_1\min(\nu,\gamma)\,E.
\end{equation*}
Hence $E$ satisfies the differential inequality
$\tfrac{d}{dt}E\le -2\lambda_1\min(\nu,\gamma)\,E$, and  direct integration yields \eqref{eq:exp_decay}.
\end{proof}

\section{Wave oscillation period}
\label{app:wave_period}
Here we derive the oscillation period of the internal wave of the Boussinesq system \eqref{eq:momentum}, \eqref{eq:temperature}, \eqref{eq:incompressibility} with the hydrostatic state \eqref{def_he}.
The operator $-\Delta$  on $\Omega=(-1,1)^{2}$ with the homogeneous Dirichlet boundary condition has 
the eigenfunctions 
$\varphi_{m,n}(x_{1},x_{2})
 =\sin\!\bigl(\tfrac{m\pi(x_{1}+1)}{2}\bigr)
  \sin\!\bigl(\tfrac{n\pi(x_{2}+1)}{2}\bigr)$
for $m,n\ge 1$, with eigenvalues
$\lambda_{m,n}=k_{1}^{\,2}+k_{2}^{\,2}$ where
$k_{1}=m\pi/2$ and $k_{2}=n\pi/2$.  The smallest-eigenvalue mode corresponds to $m=n=1$, giving
\begin{equation}\label{eq:gravest_wavenumbers}
k_{1}=k_{2}=\frac{\pi}{2},
\qquad
|\mathbf{k}|=\frac{\pi}{\sqrt{2}},
\qquad
\lambda_{\min}=\lambda_{1,1}= \frac{\pi^2}{2}.
\end{equation}

Linearizing \eqref{eq:bouss} about $(\mathbf{u},\theta)\equiv 0$, dropping
the dissipative and forcing terms, and substituting a plane-wave ansatz
$(\mathbf{u},p,\theta)\propto\exp[i(\mathbf{k}\!\cdot\!\mathbf{x}-\omega t)]$
yields, after eliminating $\widehat p$ and $\widehat\theta$, the standard
dispersion relation
\begin{equation}\label{eq:dispersion}
\omega^{2}\;=\;N^{2}\,\frac{k_{1}^{\,2}}{|\mathbf{k}|^{2}},
\qquad N=\sqrt{\alpha},
\end{equation}
in which $N$ is the Brunt--V\"ais\"al\"a frequency \cite{Gill1982,Vallis2017}.   Substituting
\eqref{eq:gravest_wavenumbers} into \eqref{eq:dispersion},
the smallest internal-wave frequency on $\Omega=(-1,1)^{2}$ is
\begin{equation}\label{eq:omega1}
\omega_{1}\;=\sqrt{\alpha} \,\frac{k_{1}}{|\mathbf{k}|}
\;=\;\frac{\sqrt{\alpha}}{\sqrt{2}}.
\end{equation}

When a field oscillates with frequency $\omega_1$, then its absolute value and thus the $L^2$ oscillates with frequency $2\omega_1$ with period
\begin{equation}\label{eq:period_prediction}
T\;=\;\frac{\pi}{\omega_{1}}\;=\;\pi\sqrt{\frac{2}{\alpha}}.
\end{equation}

\bibliographystyle{plain}
\bibliography{ref}
\end{document}